\documentclass[reqno]{amsart}

\usepackage[
left=3cm,
right=3cm,
top=3cm, 
bottom=3cm,
a4paper]{geometry}

 \usepackage{setspace}

\let\tempone\itemize
\let\temptwo\enditemize
\renewenvironment{itemize}{\tempone \vspace{5pt}\addtolength{\itemsep}{0.5\baselineskip}}{\vspace{5pt} \temptwo}

\let\tempenum\enumerate
\let\tempenumtwo\endenumerate
\renewenvironment{enumerate}{\tempenum \vspace{5pt} \addtolength{\itemsep}{0.5\baselineskip}}{ \vspace{5pt} \tempenumtwo}

\makeatletter
\let\origsection\section
\renewcommand\section{\@ifstar{\starsection}{\nostarsection}}
\newcommand\nostarsection[1]
{\sectionprelude\origsection{#1}\sectionpostlude}
\newcommand\starsection[1]
{\sectionprelude\origsection*{#1}\sectionpostlude}
\newcommand\sectionprelude{%
	\vspace{1em} 
}
\newcommand\sectionpostlude{%
	\vspace{1em}   
}
\makeatother

\newcommand\Item[1][]{%
  \ifx\relax#1\relax  \item \else \item[#1] \fi
  \abovedisplayskip=0pt\abovedisplayshortskip=0pt~\vspace*{-\baselineskip}}

\makeatletter
\let\origsubsection\subsection
\renewcommand\subsection{\@ifstar{\starsubsection}{\nostarsubsection}}
\newcommand\nostarsubsection[1]
{\sectionprelude\origsubsection{#1}\subsectionpostlude}
\newcommand\starsubsection[1]
{\subsectionprelude\origsubsection*{#1}\subsectionpostlude}
\newcommand\subsectionprelude{%
	\vspace{0.25em} 
}
\newcommand\subsectionpostlude{%
	\vspace{0.0em}   
}
\makeatother

\usepackage[T1]{fontenc}
\usepackage{lmodern}
\usepackage[utf8]{inputenc}
\usepackage{bm}
\usepackage{mathrsfs}
\usepackage{bbm}
\usepackage{faktor} 
\usepackage{mathtools}

\usepackage{esvect}
\usepackage[UKenglish]{babel}

\usepackage{amssymb}
\usepackage{amsthm}
\usepackage{amsmath}
\usepackage{mathrsfs}
\usepackage{tikz-cd}

\usepackage[scr=dutchcal, bb=boondox]{mathalpha}

\usepackage{tikz}

\usetikzlibrary{decorations.markings}
\usetikzlibrary{shapes.geometric, patterns, angles, quotes}

\usetikzlibrary{hobby}

\usepackage{amsmath,amssymb,amsthm,graphicx,amscd,amsfonts,latexsym, }
\usepackage{pstricks}
\usepackage[normalem]{ulem}

\usepackage{lscape}
\usepackage{hyperref} 

\hypersetup{
	colorlinks=true,
	citecolor=blue,
	linkcolor=blue,
	filecolor=magenta,      
	urlcolor=cyan,
}
\usepackage[capitalize,noabbrev,nameinlink]{cleveref}
\usepackage{tikz}

\usetikzlibrary{arrows}
\usetikzlibrary{decorations}
\usetikzlibrary{shapes}
\usetikzlibrary{decorations.pathreplacing,decorations.markings} 

\usetikzlibrary{calc,intersections,through,backgrounds}

\usepackage{graphicx}
\usepackage{caption} 
\usepackage{float}
\usepackage{upgreek}

\usetikzlibrary{shapes.geometric,positioning}
\usetikzlibrary{arrows,decorations.pathmorphing,decorations.pathreplacing}

\usepackage{ifthen} 

\makeatletter
\newcases{rightalignedcases}{\quad}{%
\hfil$\m@th\displaystyle{##}$\hfil}{$\m@th\displaystyle{##}$\hfil}{\lbrace}{.}
\makeatother

\newcounter{pos} 
\tikzset{									
	initcounter/.code={\setcounter{pos}{0}},
	style between/.style n args={3}{
		postaction={
			initcounter,
			decorate,
			decoration={
				show path construction,
				curveto code={
					\addtocounter{pos}{1}
					\pgfmathtruncatemacro{\min}{#1 - 1}
					\ifthenelse{\thepos < #2 \AND \thepos > \min}{
						\draw[#3]
						(\tikzinputsegmentfirst)
						..
						controls (\tikzinputsegmentsupporta) and (\tikzinputsegmentsupportb)
						..
						(\tikzinputsegmentlast);
					}{}
				}
			}
		},
	},
}

\tikzset{
    clip even odd rule/.code={\pgfseteorule}, 
    invclip/.style={
        clip,insert path=
            [clip even odd rule]{
                [reset cm](-\maxdimen,-\maxdimen)rectangle(\maxdimen,\maxdimen)
            }
    }
}

\usepackage{epstopdf}

\makeatletter
\newcommand{\colim@}[2]{%
	\vtop{\m@th\ialign{##\cr
			\hfil$#1\operator@font colim$\hfil\cr
			\noalign{\nointerlineskip\kern1.5\ex@}#2\cr
			\noalign{\nointerlineskip\kern-\ex@}\cr}}%
}
\newcommand{\colim}{%
	\mathop{\mathpalette\colim@{\rightarrowfill@\textstyle}}\nmlimits@
}
\makeatother

\DeclareMathOperator{\Hom}{Hom}

\DeclareMathOperator{\End}{End}

\DeclareMathOperator{\Aut}{Aut}
\DeclareMathOperator{\Perf}{Perf}

\newcommand{\Dfd}[1]{\mathsf{D}_{\operatorname{fd}}(#1)}

\DeclareMathOperator{\Img}{Im}
\DeclareMathOperator{\DPic}{\mathcal{D}Pic}
\DeclareMathOperator{\DPicScheme}{\mathscr{DPic}}

\DeclareMathOperator{\Out}{Out}
\DeclareMathOperator{\OutO}{Out_{\circ}}
\DeclareMathOperator{\RHom}{RHom}

\newcommand{\invex}{{\scriptstyle \text{\rm !`}}}

\newcommand{\Ob}[1]{\operatorname{Ob}(#1)}

\newcommand{\filt}[2]{{#1}^{(#2)}}

\DeclareMathOperator{\ad}{ad}

\DeclareMathOperator{\PGL}{PGL}
\DeclareMathOperator{\rad}{rad}

\tikzset{
	set arrow inside/.code={\pgfqkeys{/tikz/arrow inside}{#1}},
	set arrow inside={end/.initial=>, opt/.initial=},
	/pgf/decoration/Mark/.style={
		mark/.expanded=at position #1 with
		{
			\noexpand\arrow[\pgfkeysvalueof{/tikz/arrow inside/opt}]{\pgfkeysvalueof{/tikz/arrow inside/end}}
		}
	},
	arrow inside/.style 2 args={
		set arrow inside={#1},
		postaction={
			decorate,decoration={
				markings,Mark/.list={#2}
			}
		}
	},
}

\makeatletter
\tikzset{commutative diagrams/.cd,arrow style=tikz,diagrams={>=latex'}}\tikzset{join/.code=\tikzset{after node path={%
			\ifx\tikzchainprevious\pgfutil@empty\else(\tikzchainprevious)%
			edge[every join]#1(\tikzchaincurrent)\fi}}}
\makeatother

\tikzset{>=stealth',every on chain/.append style={join},
	every join/.style={->}}

\tikzset{every loop/.style={min distance=25mm,in=50,out=100,looseness=5}}

\tikzcdset{
	cells={font=\everymath\expandafter{\the\everymath\displaystyle}},
}

\newtheorem{prf}{Proof}[section]

\theoremstyle{remark}

\newtheoremstyle{ownTheoremStyle}
{1em}
{1em}
{\itshape}
{}
{\bfseries}
{.}
{ }
{}

\newtheoremstyle{ownDefinitionStyle}
{1em}
{1em}
{}
{}
{\bfseries}
{.}
{ }
{}

\theoremstyle{ownTheoremStyle}

\newtheorem{thm}[prf]{Theorem}

\newtheorem{Introthm}{Theorem}

\newtheorem{lem}[prf]{Lemma}
\newtheorem{prp}[prf]{Proposition}
\newtheorem{conjecture}[prf]{Conjecture}
\newtheorem{cor}[prf]{Corollary}

\theoremstyle{ownDefinitionStyle}
\newtheorem{exa}[prf]{Example}

\newtheorem{definition}[prf]{Definition}
\newtheorem{rem}[prf]{Remark}

\newcommand{\quotient}[2]{{\left.\raisebox{.2em}{$#1$}\middle/\raisebox{-.2em}{$#2$}\right.}}

\numberwithin{equation}{section}

\newcommand{\cC}{\mathcal{C}}
\newcommand{\cD}{\mathcal{D}}
\newcommand{\cE}{\mathcal{E}}
\newcommand{\cF}{\mathcal{F}}
\newcommand{\cG}{\mathcal{G}}
\newcommand{\cH}{\mathcal{H}}
\newcommand{\cK}{\mathcal{K}}
\newcommand{\cL}{\mathcal{L}}
\newcommand{\cM}{\mathcal{M}}

\newcommand{\cU}{\mathcal{U}}
\newcommand{\cV}{\mathcal{V}}
\newcommand{\cW}{\mathcal{W}}

\newcommand{\cY}{\mathcal{Y}}

\newcommand{\bA}{\mathbb{A}}
\newcommand{\bB}{\mathbb{B}}

\newcommand{\bD}{\mathbb{D}}
\newcommand{\bE}{\mathbb{E}}

\newcommand{\bU}{\mathbb{U}}
\newcommand{\bV}{\mathbb{V}}
\newcommand{\bW}{\mathbb{W}}

\newcommand{\op}{\operatorname{op}}

\DeclareMathOperator{\MCG}{\operatorname{MCG}}

\newcommand{\HHH}{\operatorname{HH}}
\newcommand{\HH}{\operatorname{H}}

\newcommand{\gr}{\operatorname{gr}}

\newcommand{\Fuk}{\operatorname{Fuk}}

\DeclareMathOperator{\Hqe}{\mathsf{Hqe}}

\DeclareMathOperator{\dgcat}{\mathsf{dgcat}}
\DeclareMathOperator{\Acatc}{\mathsf{A_{\infty}-cat^c}}

\DeclareMathOperator{\HAcat}{\mathsf{Ho(A_{\infty}-cat^c})}

\DeclareMathOperator{\qrep}{\mathsf{qrep}^r}
\DeclareMathOperator{\Fun}{\mathsf{Fun}}

\newcommand{\h}{\mathsf{h}}
\DeclareMathOperator{\sh}{sh}
\newcommand{\sfZ}{\mathsf{Z}}

\newcommand{\BCH}[2]{\operatorname{BCH}(#1, #2)}
\newcommand{\nBCH}{\operatorname{BCH}}

\newcommand{\IsoInfty}{\operatorname{Iso}^{\infty}}
\newcommand{\Iso}{\operatorname{Iso}}

\newcommand{\Graph}{\mathsf{\cG_{\Bbbk}}}
\newcommand{\cocat}{\mathsf{cocat}}

\newcommand{\coder}{\operatorname{coder}}
\newcommand{\res}{\operatorname{res}}

\newcommand{\MC}{\mathsf{MC}} 

\title[Integration, derived Picard groups and Uniqueness of lifts]{Integration of Hochschild cohomology, derived Picard groups and uniqueness of lifts}
\author{Sebastian Opper}
\address{Sebastian Opper, Charles University, Faculty of Mathematics and Physics, Ke Karlovu 3, 121 16 Praha 2, Czech Republic}
\email{opper@karlin.mff.cuni.cz}

\begin{document}
\begin{abstract}The paper introduces a partial integration map from the first Hochschild cohomology of any cohomologically unital $A_\infty$-category over a field of characteristic zero to its derived Picard group. We discuss useful properties such as injectivity, naturality and its relation with the Baker-Campbell-Hausdorff formula. 
Based on the image of the integration map we propose a candidate for the identity component of the derived Picard group in the case of finite-dimensional graded algebras. As a first application of the integration map it is shown that the vanishing of its domain is a necessary condition for the uniqueness of lifts of equivalences from the homotopy category to the $A_\infty$-level. The final part contains applications to derived Picard groups of wrapped and compact Fukaya categories of cotangent bundles and their plumbings and an outlook on applications to derived Picard groups of partially wrapped Fukaya categories after Haiden-Katzarkov-Kontsevich.
\end{abstract}
\maketitle
\setcounter{tocdepth}{1}
\tableofcontents

\section*{Introduction}
\noindent The fruitful interplay between Lie groups and Lie algebras permeates many areas of modern mathematics from rational homotopy theory to the deformation theory of algebraic objects such as algebras, schemes and operads. In the classical case, the tangent space to any Lie group $G$ at its neutral element is naturally a Lie algebra. It possible to pass from the Lie algebra back to the original group through integration: to ``explore'' the Lie group, one identifies an element in its Lie algebra with a vector field on $G$ and transports the neutral element under the corresponding flow. If $G$ is compact or nilpotent, then the set of points which are reachable through this procedure coincides with the identity component $G_{\circ} \subseteq G$. Formally this yields a surjection from the Lie algebra onto $G_{\circ}$ which generalises the real exponential $e^{-}:\mathbb{R} \rightarrow \mathbb{R}_{\geq 0}$ and the matrix exponential. In general, the exponential fails to be injective and for nilpotent $G_{\circ}$, its kernel is measured by the fundamental group of $G_{\circ}$. 
The pullback of the group structure under the exponential can be expressed locally through the Baker-Campbell-Hausdorff (BCH) formula. It states that the product of elements $x$ and $y$ in the Lie algebra is expressed by an infinite series
\begin{displaymath}
\BCH{x}{y}=x+y - \frac{1}{2} [x,y] + \frac{1}{12}\big[x, [x,y]\big] + \cdots,
\end{displaymath}
\noindent involving iterated brackets of $x$ and $y$.

In an algebraic setting, a similar connection between Lie algebras and groups exists between the Hochschild cohomology $\HHH^{\bullet}(A,A)$ of an algebra $A$ and its derived Picard group $\DPic(A)$. The former is a well-studied invariant and is formally defined as
\begin{displaymath}
\HHH^{\bullet}(A,A)=\Hom^{\bullet}_{\cD(A^{\op} \otimes A)}(A,A),
\end{displaymath}
\noindent the graded endomorphism algebra of $A$ which is considered as an object in the the derived category of $A$-bimodules. In addition to its multiplication, $\HHH^{\bullet}(A,A)$ carries a compatible Lie bracket which was first introduced by Gerstenhaber \cite{GerstenhaberCohomologyStructure}, now commonly referred to as the \textit{Gerstenhaber bracket}. On the other hand, the derived Picard group takes role of the derived symmetry group of the derived category $\cD(A)$. On a formal level it is defined as the group of isomorphism classes in $\cD(A^{\op} \otimes A)$ which are invertible with respect to the derived tensor product and its elements induce autoequivalences of $\cD(A)$. The fact that derived categories are quite large by nature makes derived Picard groups generally rather elusive to computations. 

The connection between $\HHH^{\bullet}(A,A)$ and $\DPic(A)$ was established in work of Keller \cite{KellerInvarianceHigherStructures} which states that $\HHH^{\bullet}(A,A)$, equipped with its Gerstenhaber Lie bracket, is the graded Lie algebra of $\DPic(A)$.  Because of their role as symmetry groups of triangulated categories from algebra and geometry, the study of derived Picard groups is of considerable interest but even to compute them partially usually requires a rather good understanding of the whole category. In contrast to that, many examples in the literature show that computations for Hochschild cohomology are much more viable, partially due to the fact that they can be performed by means of any set of generators instead of the entire category. In this paper we develop, in analogy with the classical situation, a procedure to recover parts of the ``identity component'' of $\DPic(A)$ through Hochschild cohomology.

Our first result shows that, over fields of characteristic zero, there is indeed a general exponential which converges on a subspace of the first Hochschild cohomology defined in terms of its weight (tensor length) filtration.
\begin{Introthm}[\Cref{TheoremNaturalityExponentialQuasiIsos} \&  \Cref{CorollaryExponentialInjectiveMainTheorem}]\label{IntroTheoremA}
Let $\bA$ be a cohomologically unital $A_\infty$-category over a field $\Bbbk$ of characteristic $0$. There exists a group homomorphism
\begin{displaymath}
	\begin{tikzcd}
\exp_{\bA}: \HHH^1_+(\bA, \bA) \arrow{r} & \DPic(\bA). 
\end{tikzcd}
\end{displaymath}
\noindent  Here, $\HHH^1_+(\bA, \bA)$ denotes the classes in the first Hochschild cohomology group of $\bA$ with weight greater than $1$ equipped with the group multiplication defined by the BCH series.  The map $\exp_{\bA}$ is injective if the image of the characteristic morphism $\Pi^0:\HHH^{0}(\bA, \bA) \rightarrow Z^0_{\gr}(\HH^{\bullet}(\bA))$ into the degree $0$ part of the graded center of the graded  homotopy category $\HH^{\bullet}(\bA)$ contains all units.

\end{Introthm}
\noindent  The generality of \Cref{IntroTheoremA} allows potential applications in a wide range of scenarios from  algebra and geometry. For example, it applies to Fukaya categories over a Novikov field of characteristic zero. Such categories rarely have strict identity morphisms but are nevertheless cohomologically unital. This flexibility allows computations without the need to alter the structure of the original category.  The injectivity condition in \Cref{IntroTheoremA} is satisfied in a number of different scenarios and, for example, holds in either of the following cases:
\begin{itemize}
	\item $\bA$ is formal, that is, is quasi-isomorphic to a graded $\Bbbk$-linear category.
	
	\item $Z^0_{\gr}(\HH^0(\bA)\big)\cong \prod_{X \in \Ob{\bA}}\Bbbk$, e.g.~if $\End_{\HH^0(\bA)}(X)\cong \Bbbk$ for all $X \in \Ob{\bA}$. For example this holds for all proper $A_\infty$-categories over an algebraically closed field, whose objects form a (possibly infinite) simple-minded collection in the sense of \cite{KoenigYangSilting}.
	\item $\bA$ is an $E_{\infty}$-algebra, e.g.~any commutative, minimal $A_\infty$-algebra or $\bA=C^{\ast}(X, \Bbbk)$, the dg algebra of singular cochains on a topological space $X$;
	\item more generally, $\bA$ is an $E_2$-algebra, e.g.~the Hochschild complex of any $A_\infty$-category.
\end{itemize}
\noindent  We remark that the characteristic morphism $\Pi^0$ is part of a long exact sequence which can aid the verification of the injectivity conditions in \Cref{IntroTheoremA}, see \Cref{RemarkConnectingMorphism}. The image of the exponential consists, roughly speaking, of all possible enhancements of the identity functor of $\HH^0(\bA)$ to an element of $\DPic(\bA)$. Formally, such elements can be represented in the form of special $A_\infty$-endofunctors of $\bA$ enhancing the identity and a consequence is that the image of $\exp_{\bA}$ purely captures higher information which is invisible on the level of homotopy categories. We note however that the associated  autoequivalences of $\cD(\bA)$ are non-trivial and triviality is merely an artifact of restriction to the subcategory $\HH^0(\bA) \subseteq \cD(\bA)$. 

While the full Hochschild cohomology is a Morita invariant, this is no longer true for the subspace $\HHH^1_+(\bA, \bA)$ and as a result, different generators of an enhanced triangulated category will generally witness different portions of the derived Picard group through their exponential maps. We expect that the image of $\exp_{\bA}$ lies in the identity component of $\DPic(\bA)$ and describe a candidate subgroup of $\DPic(\bA)$ for this component in the case of formal $A_\infty$-categories. If $\bA$ is an (ordinary) algebra, then the candidate group reduces to the identity component of the outer automorphism group of $A$ which coincides with the identity component of $\DPic(\bA)$ as shown by Yekutieli \cite{Yekutieli}. In general however our candidate group is larger and contains the image of the exponential map $\exp_{\bA}$. We conjecture the Morita invariance of the candidate group under certain conditions, see \Cref{Conjecture} and \Cref{RemarkWeakerConditionsConjecture}.

 While \Cref{IntroTheoremA} is flexible enough to work with a given $A_\infty$-category directly, computations of Hochschild cohomology often become easier when the original $A_\infty$-category is replaced by a quasi-equivalent one. The following naturality result shows that such simplifications are compatible with the integration of Hochschild classes.

\begin{Introthm}[ \Cref{TheoremNaturalityExponentialQuasiIsos}]\label{IntroTheoremNaturality}
	The exponential map is natural with respect to quasi-equivalences: every quasi-equivalence $F: \bA \rightarrow \bB$  between cohomologically unital $A_\infty$-categories induces a commutative square of group homomorphisms
	\begin{displaymath}
		\begin{tikzcd}[row sep=4em]
			\HHH^1_+(\bA, \bA) \arrow{rr}{\exp_{\bA}} \arrow{d}[swap]{\sim} && \DPic(\bA) \arrow{d}{\sim} \\
			\HHH^1_+(\bB, \bB) \arrow{rr}{\exp_{\bB}} & &\DPic(\bB).
		\end{tikzcd}
	\end{displaymath} 
	\noindent The vertical isomorphisms are compatible with the composition of functors and only depend on the \textit{weak equivalence class} of $F$, that is, its isomorphism class in the homotopy category of the functor category $\Fun(\bA, \bB)$. 
\end{Introthm}
  
\subsection*{Outline of the proof strategy}\ \medskip 

\noindent Before we proceed to describe some implications of \Cref{IntroTheoremA} and \Cref{IntroTheoremNaturality} we would like to provide an overview of the methods which are involved in their proof. In order to define the exponential map we resort to the theory of pre-Lie algebras. Loosely speaking, these are objects between the notions of algebras and Lie algebras which are equipped with a non-associative product. A well-known example is the Hochschild complex of an $A_\infty$-category with its composition product. Over a field of characteristic zero and under suitable completeness assumptions, one can use the product of a pre-Lie algebra to define an exponential self-map. Previously, this exponential has been successfully used in the study of the deformation theory of algebras over operads in work by Dotsenko-Shadrin-Vallette \cite{DotsenkoShadrinVallettePreLie}. In the present paper, we apply it to establish the existence of the integration map of \Cref{IntroTheoremA}.

 The first step is to use the pre-Lie exponential on the Hochschild complex to construct an \textit{injective} intermediate  exponential $\HHH^1_+(\bA, \bA) \hookrightarrow \Aut^{\infty, h}(\bA)$ through which the map $\exp_{\bA}$ from \Cref{IntroTheoremA} factors as follows:

\begin{equation}\label{EquationIsomorphismHomotopyIntegration}
	\begin{tikzcd}
		\HHH^1_+(\bA, \bA) \arrow{rr}{\exp_{\bA}} \arrow[hookrightarrow]{dr}[swap]{} && \DPic(\bA), \tag{$\ast$} \\
		& \Aut^{\infty, h}(\bA) \arrow{ur}[swap]{\pi}.
	\end{tikzcd}
\end{equation}
\noindent The codomain of this intermediate exponential are homotopy classes of certain $A_\infty$-endofunctors of $\bA$. In the context of derived Picard groups however, weak equivalence of functors (see \Cref{IntroTheoremA}) is the correct notion and our injectivity condition in \Cref{IntroTheoremA} results from a comparison between the two. The connecting morphism $\Pi^0$ appears naturally in the proof of the injectivity condition and draws a connection to the notion of \textit{$A_\infty$-centre} due to Briggs-Gelinas \cite{BriggsGelinas}. In three of the four above example classes for which $\exp_{\bA}$ is injective, the fact that the injectivity condition of \Cref{IntroTheoremA} are satisfied can be derived from observations from \cite{BriggsGelinas}. In analogy to the classical case, we view the kernels of the maps  $\exp_{\bA}$ and $\pi$ from \eqref{EquationIsomorphismHomotopyIntegration} as analogues of the fundamental group of a Lie group. 
 
Using properties of the pre-Lie exponential from \cite{DotsenkoShadrinVallettePreLie}, we show that it induces a bijection between certain $A_\infty$-functors and Hochschild cocycles.  At this point it is not clear that pre-Lie exponential descends to cohomology and hence maps cohomologous cocycles to homotopic $A_\infty$-functors. The reason is that these two equivalence relations are a priori not related in any obvious way and to overcome the issue we translate it to a problem in deformation theory. Functors in the image of the pre-Lie exponential are in bijection with certain Maurer-Cartan elements of the Hochschild complex which is equipped with its natural structure of an $A_\infty$-algebra. This well-known structure is determined through the \textit{brace algebra structure} on the Hochschild complex which extends the pre-Lie product by operations of higher arity. The strategy is now to link cohomology and homotopy through various notions of equivalence for Maurer-Cartan elements.

 For Hochschild complexes of dg categories, we see that the relations of homotopy and weak equivalence of $A_\infty$-functors correspond to the notions of \textit{dg gauge equivalence} and \textit{homotopy gauge equivalence} for the associated Maurer-Cartan elements. These notions were studied in recent work of Chuang-Holstein-Lazarev \cite{ChuangHolsteinLazarev} where the one first appears under the name ``gauge equivalence''.  For Hochschild complexes of ``good'' dg categories\footnote{These dg categories usually go under the name of ``semi-free'' or ``quasi-free'' dg categories.}, we show that dg gauge equivalence agrees with the classical and widely studied notion of \textit{gauge equivalence} for Maurer-Cartan elements. Gauge equivalence is known to be equivalent to the fourth equivalence relation of \textit{homotopy} (or \textit{Quillen homotopy}) and after a number of calculations we show that two Hochschild $1$-cocycles are cohomologous if and only if their associated Maurer-Cartan elements are Quillen homotopic. This establishes the existence and injectivity of the intermediate exponential in \eqref{EquationIsomorphismHomotopyIntegration} for ``good'' dg categories. The BCH formula is a consequence of the same formula for pre-Lie exponentials.
 
 The second step is to show that the injectivity and existence of the intermediate exponential are invariant properties under quasi-isomorphism. From here the general case of \Cref{IntroTheoremA} follows from the fact that every cohomologically unital $A_\infty$-category is quasi-isomorphic to a ``good'' dg category. The proof of the second step is tightly linked with the proof of naturality of $\exp_{\bA}$ as stated in \Cref{IntroTheoremNaturality}. At its core, naturality is a consequence of Keller's result \cite{KellerInvarianceHigherStructures} that the Hochschild complex of a dg category and its higher structure are Morita invariant and hence invariant under quasi-equivalence. This includes the $A_\infty$-algebra and brace algebra structure on the Hochschild complex which are bundled together in the form of a $B_\infty$-structure. In contrast to this, the weight filtration $\HHH^1(\bA, \bA)$ which is used to define the subspace $\HHH^1_+(\bA, \bA)$, is not Morita invariant but is known to be invariant under quasi-isomorphisms of $A_\infty$-algebras as shown by Briggs-Gelinas \cite{BriggsGelinas}. The strategy in both \cite{KellerInvarianceHigherStructures} and \cite{BriggsGelinas} is to construct a quasi-isomorphism between Hochschild complexes which preserve one the respective structures. However, the two constructions do not agree and it is not clear whether there exists a \textit{single} quasi-isomorphism which preserves both the higher structure and the weight filtration simultaneously. After an extension of their results to $A_\infty$-categories and quasi-equivalences, it is shown that  Keller's quasi-isomorphisms are compositions of quasi-isomorphisms from \cite{BriggsGelinas} and hence preserve both structures. From there we are able to carry out the second step and establish naturality of the exponential and, in a more restrictive form, also for the intermediate exponential (\Cref{TheoremNaturalityExponentialQuasiIsos}).
 
Finally, we would like to mention that most of the methods in this paper apply equally to certain $B_\infty$-algebras which, like the Hochschild complex, are $A_\infty$-algebras with a compatible dg Lie algebra structure. The proofs suggest that the intermediate exponential reflects a close relationship between the Maurer-Cartan space of the underlying ``Gerstenhaber'' dg Lie algebra of such an object and the Maurer-Cartan space of its underlying $A_\infty$-structure, see \Cref{RemarkMaurerCartanSpaces}.
Nonetheless, we see the concreteness of the exponential map in this paper as an advantage as it can be used to produce explicit $A_\infty$-functors instead of abstract equivalence classes of such. Examples where this proves helpful are discussed in \Cref{SectionPartiallyWrapped}.
\medskip
\subsection*{Applications}\ \medskip

\noindent As a first application of \Cref{IntroTheoremA} we show that the vanishing of $\HHH^1_+(\bA, \bA)$ is a necessary condition for the uniqueness of lifts of equivalences between homotopy categories. A \textit{lift} of a $\Bbbk$-linear functor $f$ between homotopy categories of $A_\infty$-categories is an $A_\infty$-functor $F$ such that $\HH^0(F)$ and $f$ are isomorphic. It is now known that such lifts do not always exist even in natural situations but do exist under favourable circumstances, see \cite{CanonacoStellariTourExistenceUniqueness} for an overview. A related question with a generally negative answer is to which extent lifts of $f$ are unique up to weak equivalence. The problem is strongly connected to the uniqueness problem for Fourier-Mukai kernels. Given well-behaved schemes $X$ and $Y$, every object $\cE \in \cD(X \times Y)$ in the derived category of of their product defines a \textit{Fourier-Mukai transform} in the form of a functor
\begin{displaymath}
	\begin{tikzcd}
	\Phi_{\cE}: \Perf(X) \arrow{r} & \cD(Y).
\end{tikzcd}
\end{displaymath}
\noindent The object $\cE$ is referred to as \textit{a Fourier-Mukai kernel} of $\Phi_{\cE}$ and it is natural to ask whether $\Phi_{\cE}$ determines $\cE$ uniquely up to isomorphism. Under suitable assumptions on $X$ and $Y$, e.g.~ when they are quasi-projective, the problem turns out to be equivalent to the uniqueness problem of lifts for $\Phi_{\cE}$ \cite{GenoveseUniquenessLifting}. The interested reader can find overviews on the subject in \cite{GenoveseUniquenessLifting, CanonacoStellariTourExistenceUniqueness}. A sufficient criterion for the uniqueness of $A_\infty$-lifts of certain functors  was given in work by Genovese \cite{GenoveseUniquenessLifting}. In particular, his theorem applies to fully-faithful functors and it is not difficult to see that the domain $\bA$ of a functor in his theorem satisfies $\HHH^1_+(\bA, \bA)=0$. Using the exponential map we show that this is a very frequent consequence of uniqueness, at least in characteristic $0$.
 \begin{Introthm}[{\Cref{TheoremUniquenessLiftingVanishing}}]\label{IntroTheoremC}
Let $\bA, \bB$ be cohomologically unital $A_\infty$-categories over a field of characteristic $0$ such that $\exp_{\bA}$ is injective (e.g.~under the condition in \Cref{IntroTheoremA}). If $f: \HH^{0}(\bA) \rightarrow \HH^{0}(\bB)$ is a $\Bbbk$-linear equivalence which admits a lift and whose lifts are all weakly equivalent, then $\HHH^1_+(\bA, \bA)=0$.
 \end{Introthm}
\noindent In other words,  \Cref{IntroTheoremC} shows that  in characteristic zero, there is a ``global'' necessary condition for the uniqueness of lifts of equivalences which is independent of the functor. We expect \Cref{IntroTheoremC} to admit generalizations to arbitrary liftable $\Bbbk$-linear functors $f: \HH^0(\bA) \rightarrow \HH^0(\bB)$ in characteristic zero. A sketch of a possible proof is given in \Cref{RemarkGeneralizationNecessaryCondition}. Interestingly, if $\bA$ is a graded algebra, the vanishing of $\HH^1_+(\bA, \bA)$ also appears as a sufficient condition for the affineness of its moduli space of minimal $A_\infty$-structures, cf.~\cite[Corollary 3.2.5]{PolishchukModuliofCurves}. 

 The final part of the paper discusses applications to derived Picard groups of  Fukaya categories. First, we consider the \textit{wrapped Fukaya category}  $\cW$ of the cotangent bundle $T^{\ast}M$ of a smooth manifold $M$ in the sense of \cite{AbouzaidSeidelOpenStringAnalogue, GanatraPardonShendeCovariantlyFunctorialWrappedFukayaCategory}. This is an $A_\infty$-category whose objects are possibly non-compact Lagrangian submanifolds of $T^{\ast}M$ endowed with additional data and satisfying various conditions. It contains the \textit{compact Fukaya category} $\cF \subseteq \cW$ of $T^{\ast}M$ which consists of compact Lagrangian submanifolds.
Under suitable assumptions on $M$, results by  Abouzaid \cite{AbouzaidGenerationFukayaCategory, AbouzaidFukayaCategoryPlumbings}, Chantraine-Dimitroglou Rizell-Ghiggini-Golovko \cite{ChantraineDimitroglouRizellGhigginiGolovko} and  Ganatra-Pardon-Shende \cite{GanatraPardonShendeSectorialDescent} show that these categories are generated by a single cotangent fibre whose endomorphism algebra in $\cW$ is quasi-isomorphic to the Pontryagin algebra of chains $C_{-\ast}=C_{-\ast}(\Omega_q M, \Bbbk)$ on the based (Moore) loop space at any point $q \in M$. This is dual to the generation result in \cite{FukayaSeidelSmithCategoricalViewpoint} which implies that the zero section $Z \subseteq T^{\ast}M$ generates the compact Fukaya category $\cF$ of $T^{\ast}M$, again under suitable assumptions on $M$. The endomorphism algebra of $Z$ is quasi-isomorphic to the dg algebra of cochains $C^{\ast}=C^{\ast}(M, \Bbbk)$ on $M$, equipped with the cup product.  We recall that if $M$ is simply-connected, then $C_{-\ast}$ and $C^{\ast}$ are Koszul dual and there exists an isomorphism of Lie algebras $\HHH^{\bullet}(C_{-\ast}, C_{-\ast}) \cong \HHH^{\bullet}(C^{\ast}, C^{\ast})$ which exchanges the weight filtration on either side of the isomorphism with the so-called \textit{shearing filtration}  from \cite{BriggsGelinas} on the opposite side. One can use \Cref{IntroTheoremA} to show the following.
\begin{Introthm}[\Cref{CorollaryTheoremE}]\label{IntroTheoremE} Assume that $\Bbbk$ is a field of characteristic $0$. Let $M$ be a smooth compact connected manifold and let $\cW$ (resp.~$\cF$) denote the wrapped (resp.~compact) Fukaya category of $T^{\ast}M$. There exists an embedding of groups
\begin{displaymath}
	\begin{tikzcd}
\HHH^1_+\!(C_{-\ast}, C_{-\ast}) \arrow[hookrightarrow]{r} &  \DPic(\cW).
\end{tikzcd}
\end{displaymath}
\noindent If $M$ is further simply-connected and closed, there also exists an embedding
\begin{displaymath}
	\begin{tikzcd}
	\HHH^1_+\!(C^{\ast}, C^{\ast}) \arrow[hookrightarrow]{r} &  \DPic(\cF).
	\end{tikzcd}
\end{displaymath}
\end{Introthm}
\noindent \Cref{IntroTheoremE} also extends in a certain form to Fukaya categories of plumbings of cotangent bundles along a tree, see  \Cref{SectionApplicationsFukayaCategories}. The second embedding of \Cref{IntroTheoremE} is particularly interesting if $M$ is \textit{formal} in the sense that $C^{\ast}$ is quasi-isomorphic to its cohomology. Well-known  examples are compact, simply-connected Lie groups as well as compact K\"{a}hler manifolds over a field of characteristic zero as shown in work by  Deligne-Griffiths-Morgan-Sullivan \cite{DeligneGriffithsMorganSullivan}. As a final remark to \Cref{IntroTheoremE} we also mention the relationship of $\HHH^{\bullet}(C_{-\ast},C_{-\ast})$ with string topology and symplectic homology. For simply-connected $M$, there is an isomorphism
\begin{displaymath}
\HHH^{\bullet}(C_{-\ast},C_{-\ast}) \cong \HH_{\bullet+\dim M}(\cL M),
\end{displaymath}
between the Hochschild cohomology of $C_{-\ast}$ and the singular homology of the free loop space $\cL M$ index shifted by the dimension of $M$ as shown in work by Cohen-Jones \cite{CohenJones}. Moreover, as follows from \cite{FelixThomas} the algebra and Lie algebra structure on the left hand side corresponds on the right hand side to the loop operations  due to Chas-Sullivan \cite{ChasSullivanStringTopology}. Ganatra's thesis \cite[Theorem 1.1]{GanatraPhDThesis} (see also \cite[Theorem 1.4]{ChantraineDimitroglouRizellGhigginiGolovko}) provides yet another interpretation and shows that  $\HHH^{\bullet}(C_{-\ast}, C_{-\ast})$ is isomorphic to the symplectic homology of $T^{\ast}M$. It is an interesting problem to determine the image of the weight and shearing filtrations on $\HHH^{\bullet}(C_{-\ast},C_{-\ast})$ under the above isomorphisms as well as their topological and geometric interpretations.

 \Cref{IntroTheoremA} is also useful for the understanding of derived Picard groups of partially wrapped Fukaya categories of surfaces in the sense of Haiden-Katzarkov-Kontsevich \cite{HaidenKatzarkovKontsevich}. With the exception of the ``fully wrapped'' case, these categories are generated by an object whose endomorphism algebra belongs to the class of so-called \textit{graded gentle algebras}. The derived Picard group of an \textit{ungraded} gentle algebra was previously studied by the author in \cite{OpperDerivedEquivalences} where it was shown to be an extension of a mapping class group by an algebraic group which can be seen to agree with the identity component of the derived Picard group. The group from \Cref{Conjecture} and the integration map from this paper are an essential tool in the study of the graded case in \cite{OpperGradedGentle} which is discussed in \Cref{SectionPartiallyWrapped}.

\subsection*{Organization of the paper}\ \medskip

\noindent The first three sections are mostly of expository nature and review the relevant portion of the theory of $A_\infty$-categories, cocategories as well as pre-Lie algebras and brace algebras. \Cref{SectionDerivedPicardGroups} then recalls some results regarding the homotopy category of dg categories alongside its connection to $A_\infty$-functor categories and derived Picard groups. There we also discuss certain subgroups of the derived Picard group which conjecturally take the role of the identity component. In \Cref{IntegrationTheoryCompletePreLieAlgebras} we recall the results by Dotsenko-Shadrin-Vallette on pre-Lie exponentials and establish a first connection between the cohomology of a pre-Lie algebra (e.g.~Hochschild cohomology) and certain equivalence classes of $\infty$-isotopies (e.g.~certain $A_\infty$-functors). These results are then expanded on in  \Cref{SectionIntegrationHochschildClasses} in the context of brace algebras such as the Hochschild complex of an $A_\infty$-category. In particular, we establish a relationship between the Hochschild cohomology and homotopy classes of Maurer-Cartan elements of the Hochschild complex with its natural $A_\infty$-structure. \Cref{SectionIntegrationHochschildClasses} also contains the proofs of \Cref{IntroTheoremA} and \Cref{IntroTheoremNaturality}. The final section discusses applications to derived Picard groups of Fukaya categories.

\subsection*{Acknowledgment}
I like to thank Benjamin Briggs and Mattia Ornaghi for answering my questions on their work, Severin Barmaier for his explanations of deformation theory and Matthew Habermann for sharing his knowledge about Fukaya categories and pointing me to several examples. I also like to express my gratitude to Isaac Bird, Matthew Habermann, Jordan Williamson and Alexandra Zvonareva for feedback on an earlier draft of this paper. While working on this project, I was supported by the Charles University Research Center programs UNCE/SCI/022 and  UNCE24/SCI/022 as well as the Primus grant PRIMUS/23/SCI/006. I was also partially supported by the Cooperatio program of Charles University.

\section{Pre-Lie algebras, brace algebras and symmetric brace algebras}
\noindent We recall the notions of pre-Lie algebras and their higher relatives, \textit{brace algebras}, as well as their connection to their symmetric counterparts. These types of structure are naturally present on the Hochschild complex of an $A_\infty$-category and subsume many well-known structures on the Hochschild complex of an ordinary algebra such as the differential, the Gerstenhaber Lie bracket and the cup product.
\label{SectionPreLieAlgebrasExponentialMaps}

\subsection{Pre-Lie algebras}\label{SectionPreLieAlgebras}\ \medskip

\noindent Throughout the paper we apply the Koszul sign rule to tensor products of graded maps, that is, if $f: P \rightarrow U$ and $g:Q \rightarrow V$ are homogeneous graded linear maps between graded vector spaces, then for all homogeneous $p \in P$, $q \in Q$,
\begin{displaymath}
(f \otimes g)(p \otimes q) \coloneqq (-1)^{|g|\cdot |p|} f(p) \otimes g(q),
\end{displaymath}
\noindent where $|{-}|$ denotes the degree of an element.

\begin{definition}A (left unital) \textbf{pre-Lie algebra} $(V, \star, \mathbf{1})$ is a graded vector space $V$ equipped with a bilinear multiplication $\star:V \otimes V \rightarrow V$ of degree $0$ subject to the condition
	\begin{equation}\label{EquationPreLieAlgebra}
		u \star (v \star w) - (u \star v) \star w = (-1)^{|v||w|} \big(u \star (w \star v) - (u \star w) \star v\big),
	\end{equation}
	and an element $\mathbf{1} \in V$ such that $\mathbf{1} \star v=v$ for all $v \in V$. A \textbf{morphism} of pre-Lie algebras $f: V \rightarrow W$ is a graded linear map of degree $0$ such that $f(\mathbf{1})=\mathbf{1}$ and which commutes with the $\star$-products on $V$ and $W$.
\end{definition}
\noindent In other words, a pre-Lie algebra is equipped with a non-associative multiplication whose associator $\nabla(u,v,w)$,  the left hand side of \eqref{EquationPreLieAlgebra}, is graded symmetric with respect to $v$ and $w$.  Graded algebras are identified with pre-Lie algebras with $\nabla=0$. Every pre-Lie algebra gives rise to a graded Lie algebra through anti-symmetrisation, that is,
\begin{equation}
	[u, v] \coloneqq u \star v - (-1)^{|u||v|} v \star u,
\end{equation}
\noindent defines a graded Lie bracket on $V$. Hence, pre-Lie algebras sit in between graded algebras and graded Lie algebras.

\subsection{Example: the Hochschild space of a vector space}\label{SectionHochschildComplexPreLieAlgebra}\ \medskip

\noindent For a graded $\Bbbk$-vector space $V$, we consider the graded vector space $C(V)$ whose $n$-th homogeneous component is given by
\begin{displaymath}
\begin{array}{ccc} \displaystyle C(V)^n = \prod_{p=0}^{\infty} \Hom_{\Bbbk}^{n}\Big({\big(V[1]\big)}^{\otimes p}, V[1]\Big), & & {V[1]}^{\otimes 0}\coloneqq \Bbbk.
\end{array}
\end{displaymath}

\noindent  The space $C(V)$ is a special case of the Hochschild space of a $\Bbbk$-graph which we recall later and plays a vital role for the theory of $A_\infty$-categories. The slightly awkward looking choice of grading through $V[1]$ has the advantage that we can hide and ignore the inevitable signs when dealing with $A_\infty$-categories for the majority of the paper. The vector space $C(V)$ becomes a pre-Lie algebra under the \textbf{composition product} $\star$ defined for homogeneous $f:{V[1]}^{\otimes m} \rightarrow V[1]$, $g: {V[1]}^{\otimes n} \rightarrow V[1]$, via the formula
	\begin{equation}
		f \star g \coloneqq \sum_{i=0}^{m-1} f\Big(\operatorname{Id}_{V[1]}^{\otimes i} \otimes \, g \otimes \operatorname{Id}_{V[1]}^{\otimes (m-1-i)}\Big) \in \Hom_{\Bbbk}({V[1]}^{\otimes (m+n-1)}, V[1]\big).
	\end{equation}
	\noindent Note that the Koszul sign rule introduces additional signs upon evaluation of $f \star g$ at elements. The fact that $(C(V), \star)$ defines a pre-Lie algebra, follows for example from the graded symmetry of the right side of the equation \cite[Lemma 1.2.]{GetzlerCartanHomotopy}. It has a left unit $\mathbf{1}$ given by the identity of $V[1]$.

\subsection{Brace algebras and symmetric brace algebras}\ \medskip

\noindent The pre-Lie algebra structure on $C=C(V)$ from \Cref{SectionHochschildComplexPreLieAlgebra} is the lowest order operation of the richer structure of a \textit{brace algebra} and that of its associated \textit{symmetric brace algebra}. Rewriting $f \star g$ as $f\{g\}_1$, it forms part of a hierarchy of multilinear brace operations 
\noindent defined on $f \in \Hom_{\Bbbk}({V[1]}^{\otimes m}, V[1])$ and functions $g_1, \dots, g_r$ with $g_i \in \Hom_{\Bbbk}({V[1]}^{\otimes n_i}, V[1])$ as
\begin{equation}\label{EquationDefinitionBracesHochschildComplex}
	f\{g_1, \dots, g_r\}_r \coloneqq \sum  f\Big(\operatorname{Id}_{V[1]}^{\otimes i_1} \otimes \, g_1 \otimes \operatorname{Id}_{V[1]}^{\otimes i_2} \otimes \, g_2 \otimes \cdots \otimes g_r \otimes \operatorname{Id}_{V[1]}^{\otimes i_r}\Big),	
\end{equation}
\noindent where the sum is taken over all integer-valued tuples $(i_1, \dots, i_r)$  such that $i_j \geq 0$ and $\sum_{j=1}^r i_j=m-r$. If $r$ is clear from the context, we drop it from the notation and simply write $f\{g_1, \dots, g_r\}$.  As for its pre-Lie algebra structure, additional signs are introduced upon evaluation according to the Koszul sign rule. The brace operations satisfy the quadratic relations
\begin{multline}\label{EquationIdentifiesBraceAlgebra}
	f\{g_1, \dots, g_r\}\{h_1, \dots, h_s\} =  \sum \varepsilon \cdot f\big\{h_1, \dots, h_{i_1}, g_1\{h_{i_1+1}, \dots, h_{j_1}\}, h_{j_1+1}, \dots, h_{i_r}, \\  g_r\{h_{i_r+1}, \dots, h_{j_r}\}, h_{j_r+1}, \dots, h_{s} \big\},
\end{multline}
\noindent where the sum is taken  over all sequences $0 \leq i_1 \leq j_1 \leq \cdots \leq i_r \leq j_r \leq s$ and where $\varepsilon$ denotes the Koszul sign of the permutation
\begin{multline*}
	(g_1, \dots, g_r, h_1, \dots, h_s) \mapsto  \big(h_1, \dots, h_{i_1}, g_1, h_{i_1+1}, \dots, h_{j_1}, h_{j_1+1}, \dots, h_{i_r}, g_{r}, h_{i_r+1}, \dots, h_{j_r},\\ h_{j_r+1}, \dots, h_{s}\big).
\end{multline*}
\noindent Abstracting the previous definitions leads to the following notion as given in \cite{GerstenhaberVoronov}.
\begin{definition}
	A (left unital) \textbf{brace algebra} is a graded vector space together with an element $\mathbf{1} \in V^0$ and a collection of multi-linear brace operations $-\{\cdots\}_r: V^{\otimes (r+1)} \rightarrow V$ of degree $0$ for each $r \geq 1$ which satisfy the identities \eqref{EquationIdentifiesBraceAlgebra} as well as $\mathbf{1}\{\cdots\}_r=0$ for all $r \geq 2$ and $\mathbf{1}\{g\}=g$ for all $g \in V$.  A \textbf{morphism} of brace algebras is a graded linear map $f: V \rightarrow W$ of degree $0$ such that $f(\mathbf{1})=\mathbf{1}$ and $f(v\{v_1, \dots, v_r\})=f(v) \{f(v_1), \dots, f(v_r)\}$ for all $v,v_1, \dots, v_r \in V$.  

\end{definition}
\noindent There is a  more compact way to rephrase the notion of brace algebras in the language of tensor coalgebras which we revisit in \Cref{SectionBInfinityAlgebras}. The sign conventions and degree of the brace operations in \cite{GerstenhaberVoronov} are different from ours which agree with the conventions in \cite{LadaMarkl}.\medskip  

\noindent Upon symmetrisation of the brace operations we obtain a hierarchy of symmetric brace operations $-\langle-, \dots, - \rangle_r: V^{\otimes (r+1)} \rightarrow V$ defined on homogeneous elements via the formula
\begin{equation}
	f\langle g_1, \dots, g_r  \rangle_r \coloneqq \sum_{\sigma \in S_r} \epsilon \cdot f\{g_{\sigma(1)}, \dots, g_{\sigma(r)}\}_r,
\end{equation}
\noindent where $\varepsilon$ denotes the Koszul sign of the permutation $(g_1, \dots, g_r) \mapsto (g_{\sigma(1)}, \dots, g_{\sigma(r)})$. In particular, $f\{g\}=f\langle g \rangle_1$ and, more generally, for $n \geq 2$ we have
\begin{equation}\label{EquationPowerSymmetrisation}
	f\langle g, \dots, g \rangle_n = n!\cdot  f\{g, \dots, g\}_n,
\end{equation}
\noindent if $|g|$ is even and  
\begin{displaymath}
	f\langle g, \dots, g\rangle_n=0,
\end{displaymath}
\noindent whenever $|g|$ is odd. The symmetric braces satisfy similar identities as their non-symmetric counterparts, see \cite[Definition 2]{LadaMarkl}, leading to the notion of a \textit{symmetric brace algebra} introduced by Lada and Markl \cite[Definition 2]{LadaMarkl}. An interesting and non-trivial result by Oudom and Guin is the following.

\begin{thm}{\cite{OudomGuin}}
	Any symmetric brace algebra $V$ is uniquely determined by its first order operation $- \star - \coloneqq -\langle-\rangle_1$ which endows $V$ with the structure of a pre-Lie algebra. This correspondence induces an equivalence between the category of symmetric brace algebras and the category of pre-Lie algebras.  
\end{thm}
\noindent In \cite[Proposition 4.4]{OudomGuin} the interested reader can find explicit formulas for the operations of the associated symmetric brace algebra structure in terms of the  $\star$-product.

\section{Graphs, cocategories and Hochschild spaces}\label{SectionCocategories}
\noindent We recall the notion of $\Bbbk$-graphs which are essentially ``many-object'' versions of a vector space. Afterwards we discuss cocategories along with the properties of the tensor cocategory associated with a $\Bbbk$-graph. Finally, we tend to the Hochschild space of a $\Bbbk$-graph which carries the structure of a brace algebra much as in the previous case of a vector space. This provides us with a convenient setup for the following section on $A_\infty$-categories. This section is purely expository.

\subsection{Graphs and cocategories}\ \medskip

\noindent We recall the basic notions for cocategories. We will be brief and the reader may consult \cite{CanonacoOrnaghiStellari, KellerInvarianceHigherStructures} for further details. A compact introduction to the special but analogous case of coalgebras is found in \cite[Section 1.2]{LodayValletteAlgebraicOperads}.
\begin{definition}
	A \textbf{$\Bbbk$-graph} (or \textbf{$\Bbbk$-quiver}) $\cG$ consists of a set of objects $\Ob{\cG}$ together with graded $\Bbbk$-vector space  $\cG(X,Y)$ for each pair $(X,Y) \in \Ob{\cG}^2$. A \textbf{morphism} $\varphi: \cG \rightarrow \cH$ between $\Bbbk$-graphs is the datum of a map $\varphi:\Ob{\cG} \rightarrow \Ob{\cH}$ of sets together with a $\Bbbk$-linear map $\varphi_{XY}:\cG(X,Y) \rightarrow \cH(\varphi(X), \varphi(Y))$ for each pair $(X,Y) \in \Ob{\cG}^2$.
\end{definition}
\noindent Morphisms of $\Bbbk$-graphs are composed in the natural way.

\begin{definition}
	A (non-unital) \textbf{cocategory} is a $\Bbbk$-graph $\cC$ equipped with $\Bbbk$-linear maps (``cocomposition'')
	\begin{displaymath}
		\begin{tikzcd}
		\Delta_{\cC}: \cC(X,Y) \arrow{r} & \bigoplus_{Z \in \cC}\cC(Z,Y) \otimes \cC(X, Z),
		\end{tikzcd}
	\end{displaymath}
	\noindent for every pair $X, Y \in \Ob{\cC}$, which satisfy the coassociativity condition 
	\begin{displaymath}
(\Delta_{\cC} \otimes \operatorname{Id}_{\cC}\big) \circ \Delta_{\cC} = (\operatorname{Id}_{\cC} \otimes \Delta_{\cC} \big) \circ \Delta_{\cC}.
	\end{displaymath}
\noindent  A \textbf{cofunctor} $F: \cC \rightarrow \cD$ is a morphism of degree $0$ between the underlying $\Bbbk$-graphs such that 
	\begin{displaymath}
		\big(F \otimes F\big) \circ \Delta_{\cC} =\Delta_{\cD} \circ  F.
	\end{displaymath}
\end{definition}
\noindent As for algebras, there are dual notions of counital and coaugmented cocategories as well as counital and coaugmented graphs which are defined in the obvious way. If $\cC$ and $\cD$ are coaugmented, then a cofunctor $F$ is \textbf{counital} if it commutes with the counits. Every coaugmented cocategory $\cC$ defines a non-unital cocategory $\overline{\cC}$ with
\begin{displaymath}
\overline{\cC}(X,Y) \coloneqq \begin{cases}\cC(X,Y) & \text{if $X \neq Y$}; \\ \quotient{\cC(X,X)}{\Bbbk \cdot \mathbf{1}_X} & \text{if $X=Y$}, \end{cases}
\end{displaymath}
\noindent where $\mathbf{1}_X \in \cC(X,X)$ denotes the image of the coaugmentation and the natural cocomposition induced by $\Delta_{\cC}$. Equivalently, $\overline{C}(X,Y)$ is canonically identified with a direct summand of $\cC(X,Y)$ given by the respective kernel of the counit, cf.~\cite[Section 1.2]{LodayValletteAlgebraicOperads}. Under this identification, the cocomposition $\Delta_{\overline{\cC}}(f)$ of $f \in \overline{\cC}(X,Y)$ can be expressed as  $\Delta_{\overline{\cC}}(f)=\Delta_{\cC}(f) - f \otimes \mathbf{1}_X +\mathbf{1}_Y \otimes f$.

 A non-unital cocategory $\cC$ is \textbf{conilpotent} if for all $X, Y \in \Ob{\cC}$ and all $f \in \cC(X,Y)$, $\Delta_{\cC}^{(n)}(f)=0$ for $n \gg 0$, where $\Delta_{\cC}^{(1)}=\operatorname{Id}_{\cC}$, $\Delta_{\cC}^{(2)}=\Delta_{\cC}$ and 
\begin{displaymath}
\Delta_{\cC}^{(n+1)}\coloneqq \big(\Delta_{\cC}^{(n)} \otimes \operatorname{Id}\big) \circ \Delta_{\cC}=\big(\operatorname{Id} \otimes \Delta_{\cC}^{(n)}\big) \circ \Delta_{\cC}.
\end{displaymath}
\noindent A coaugmented cocategory $\cC$ is conilpotent if $\overline{\cC}$ is conilpotent. We denote by $\cocat$ the category of small coaugmented conilpotent cocategories with counital cofunctors as morphisms.

\subsection{Tensor cocategories}\ \medskip

\noindent There is a forgetful functor $\cG(-)$ from the category of small $\Bbbk$-linear categories to $\Graph$ which forgets all composition maps. It has a left adjoint which assigns to a $\Bbbk$-graph $\cG$ the free graded $\Bbbk$-linear category $T^a(\cG)$ (also known as  \textbf{path category} or \textbf{tensor category}) with objects $\Ob{\cG}$ which is generated by the graded vector spaces $\cG(X,Y)$. In particular,
\begin{equation}\label{EquationDefinitionPathCategory}
	\Hom_{T^a(\cG)}(X,Y) \coloneqq \delta_{X,Y} \cdot \Bbbk \oplus \bigoplus_{n=1}^{\infty} \bigoplus_{X=U_0, \dots, U_n=Y} {\cG(U_{n-1},U_n) \otimes \cdots \otimes \cG(U_{0}, U_1)},
\end{equation}
\noindent with the induced grading and with the composition given by tensor concatenation. The symbol $\delta_{XY}$ denotes the Kronecker delta and the unit in $\delta_{XX} \cdot \Bbbk$ serves as the the identity morphism of $X \in \Ob{\cG}$ which we will denote by $\mathbf{1}_X \in T(\cG)(X,X)$. If $\cG$ has a single object, $T^{a}(\cG)$ recovers the tensor bialgebra of a vector space. Analogously, for any $\cG$, the underlying $\Bbbk$-graph of $T^a(\cG)$ admits the structure of a cocategory $T(\cG)$, called the \textbf{tensor cocategory} of $\cG$. Its cocomposition law is given by
\begin{displaymath}
	\begin{tikzcd}[row sep=0.5em]
		\Delta: T(\cG)(X,Y) \arrow{r} & \bigoplus_{Z \in \Ob{\cG}} T(\cG)(Z,Y) \otimes T(\cG)(X, Z), \\
		f_1 \otimes \cdots \otimes f_n \arrow[mapsto]{r} & \sum_{i=1}^{n} (f_n \otimes \cdots \otimes f_{i+1}) \otimes (f_{i} \otimes \cdots \otimes f_1),
	\end{tikzcd}
\end{displaymath}
\noindent where $\Delta(f)=\mathbf{1}_Y \otimes f + f \otimes \mathbf{1}_X$ and $\Delta(\mathbf{1}_X) = \mathbf{1}_X \otimes \mathbf{1}_X$ for all $X \in \Ob{\cG}$ and all $f \in \cG(X,Y)$. The tensor cocategory is naturally coaugmented by the inclusions $\Bbbk \cdot \mathbf{1}_X \hookrightarrow T(\cG)(X,X)$ with counits provided by the projections $T(\cG)(X,X) \rightarrow \Bbbk \cdot \mathbf{1}_X$. 

\noindent  The cocategory $T(\cG)$ is weight graded by tensor length, that is, the parameter $n$ in \eqref{EquationDefinitionPathCategory} and the homogeneous part of weight $0$ is spanned by all $\{\mathbf{1}_X \, | \, X \in \Ob{\cG}\}$. As a result, $T(\cG)$ is filtered by $\Bbbk$-graphs $W_iT(\cG)$, $i \geq 0$.  For all $X, Y \in \Ob{\cG}$, the corresponding term $W_iT(\cG)(X,Y)$ in the resulting filtrations 
\begin{displaymath}
	T(\cG)(X,Y)=W_0T(\cG)(X,Y) \supseteq W_1T(\cG)(X,Y) \supseteq \dots,
\end{displaymath} 
\noindent is defined by restricting to elements of weight $n \geq i$. The is also a corresponding ascending filtration of of $T(\cG)(X,Y)$ by sub cocategories $\overline{W}_iT(\cG)$, $i \geq 0$ of $T(\cG)$ defined by all elements of weight $n \leq i$. The assignment $\cG \mapsto T(\cG)$ (resp.~$\cG \mapsto \overline{T}(\cG)$) can be extended to a functor from the category of $\Bbbk$-graphs to $\cocat$ (resp.~the category of non-unital cocategories). When working with tensor cocategories one often resorts to the following universal property.
\begin{prp}[Universal property]\label{PropositionUniversalPropertyTensorCocategory} The construction $T(-)$ 
is the right adjoint to the forgetful functor $\cocat \rightarrow \Graph$, $\cC \mapsto \overline{\cC}$.
\end{prp}
\noindent  For $i \geq 0$, we denote by $\pi_i$ the projection onto the factor assigned to $n=i$ in \eqref{EquationDefinitionPathCategory}. The following instance of \Cref{PropositionUniversalPropertyTensorCocategory} will be used throughout the paper.

\begin{cor}
Let $\cG, \cH$ be $\Bbbk$-graphs. There exists a bijection
\begin{displaymath}
 \Hom_{\Graph}\big(\overline{T}(\cG), \cH\big) \cong \Hom_{\cocat}\big(T(\cG), T(\cH)\big).
\end{displaymath}
The image of a $\Bbbk$-graph morphism $\Phi: \overline{T}(\cG) \rightarrow \cH$  is the unique counital cofunctor $\widetilde{\Phi}: T(\cG) \rightarrow T(\cH)$ such that $\pi_1 \circ \widetilde{\Phi}$ restricts to $\Phi$ on $\overline{T}(\cG)$.
\end{cor}
\noindent Explicitly, for all $n \geq 1$, one has $\widetilde{\Phi}(\mathbf{1}_{X})=\mathbf{1}_{\Phi(X)}$ for all $X \in \Ob{\cG}$ and

\begin{displaymath}
	\pi_n\circ \widetilde{\Phi} = \Phi^{\otimes n} \circ \Delta_{T(\cG)}^{(n)},
\end{displaymath}
\noindent where $\Phi$ is canonically extended to $T(\cG)$ so that $\Phi(\mathbf{1}_X)=\mathbf{1}_{\Phi(X)}$.
\begin{definition}
Let $F, G:\cC_1 \rightarrow \cC_2$ be cofunctors between coaugmented conilpotent cocategories. An \textbf{$(F,G)$-derivation} of degree $m$ is a collection of graded linear maps $D:\cC_1(X,Y) \rightarrow  \cC_2(X,Y)$ of degree $m$ such that $D(\mathbf{1}_X)=0$ for all $X \in \Ob{\cC}$ and
\begin{displaymath}
	\Delta_{\cC_2} \circ D=\big(G \otimes D + D \otimes F\big) \circ \Delta_{\cC_1}.
\end{displaymath}
We denote by $\coder(F, G)$ the graded vector space of $(F, G)$-derivations and write $\coder(F)$ for $\coder(F,F)$ whose elements we simply call $F$-derivations.
\end{definition}
\noindent The space $\coder(F,G)$ inherits a descending filtration 
\begin{displaymath}
\coder(F,G)=W_0\coder(F,G) \supseteq W_1\coder(F,G) \supseteq \cdots,
\end{displaymath} 
\noindent from any ascending filtration  $0=\overline{W}_{-1}\cC_1 \subseteq \overline{W}_0\cC_1 \subseteq \cdots$ of $\cC_1$ by sub cocategories with $W_i\coder(F,G)$ consisting of all coderivations which vanish on $\overline{W}_{i-1}\cC_1$. In particular,  if $\cC_1=T(\cU)$ is a tensor cocategory, then $\coder(F,G)$ inherits a canonical descending filtration from the subcocategories $\overline{W}_iT(\cU)$ discussed above. Pre- and post-composition with a cofunctor $H$ induce graded linear maps
\begin{equation}\label{EquationNaturalityCoderivations}
\begin{tikzcd}[row sep=1em]
H^{\ast}:\coder(F, G) \arrow{r}  & \coder(F \circ H, G \circ H), \\  H_{\ast}:\coder(F, G) \arrow{r} & \coder(H \circ F, H \circ G).
\end{tikzcd}
\end{equation} 
\noindent Similar to cofunctors, tensor cocategories enjoy the following coextension property with respect to coderivations.
\begin{prp}\label{PropositionExtensionToDerivations}Let $F, G: \cC \rightarrow T(\cG)$ be cofunctors of coaugmented cocategories. There is a natural bijection between $\coder(F, G)$ and the vector space consisting of collections of linear maps 
	\begin{displaymath}
		\big\{\phi_{XY}: \cC(X,Y) \rightarrow \cG(X,Y) \big\}_{X,Y \in \Ob{\cC}}.
	\end{displaymath}
\end{prp}
\noindent Given a family $\phi=\{\phi_{XY}\}$, we denote by $D_{\phi}: \cC \rightarrow T(\cG)$ the corresponding $(F, G)$-coderivation. It is the unique $(F, G)$-coderivation such that $\pi_1 \circ D_{\phi}=\phi$ and is explicitly given by
\begin{displaymath}
	\pi_n \circ D_\phi = \sum_{i=0}^{n-1} \Big( (\pi_i\circ G) \otimes \phi \otimes (\pi_{n-i-1}\circ F) \Big) \circ \Delta_{\cC}^{(3)}.
\end{displaymath}
\noindent As a direct consequence of this formula, we see that for any cofunctor $H:T(\cG) \rightarrow T(\cH)$ and all $\phi=\{\phi_{XY}\}$ which vanish on $\overline{W}_mT(\cG)$ we have
	\begin{equation}\label{EquationInducedMapsInteractionWeightFiltrations}
\pi_1 \circ H_{\ast}(D_{\phi})=\sum_{n \geq 1}  H^n \circ \bigg(\sum_{i=0}^{n-1} \Big((\pi_i\circ G) \otimes \phi \otimes (\pi_{n-i-1}\circ F) \Big)  \bigg) \circ \Delta_{\cC}^{(3)}.
	\end{equation}
\noindent Thus $H_{\ast}(D_{\phi})=D_{\psi}$ for some $\psi$ which vanishes identically on $\overline{W}_mT(\cG)$. On the other hand, since  $H$ is compatible with the weight filtrations on its domain and its codomain, so is $H^{\ast}$.
\begin{cor}\label{CorollaryNaturalMapsDerivationPreserveWeightFiltrations} 
	If $H$ is a cofunctor between tensor cocategories, then $H_{\ast}$ and $H^{\ast}$ are morphisms of filtered objects in the category of graded vector spaces with respect to the canonical filtrations.
\end{cor}
\noindent  Moreover an easy calculation shows the following well-known fact.
\begin{lem}Let $\cC$ be a cocategory. Then  $\coder(\operatorname{Id}_{\cC})$ is a graded Lie algebra with bracket
\begin{displaymath}
	[D_1, D_2] \coloneqq D_1 \circ D_2 - (-1)^{|D_1||D_2|} D_2 \circ D_1.
\end{displaymath}
\end{lem}

\subsubsection{The differential graded case}\ \medskip

\noindent All definitions and results that were mentioned in the previous sections, e.g.~$\Bbbk$-graphs, tensor cocategories and their properties, have straightforward generalisations in the world of cochain complexes. The presence of differentials on the spaces $\cG(X,Y)$ of a $\Bbbk$-graph $\cG$ turn $T(\cG)$ into a dg cocategory, that is, a cocategory $\cC$ equipped with a coderivation $d \in \coder(\operatorname{Id}_{\cC})$ of degree $1$ such that $d \circ d=0$. Dg cofunctors $F: \cC \rightarrow \cD$ have to commute with the differentials, i.e.~$F: \cC(X,Y)\rightarrow \cD(X,Y)$ is a cochain map. For every pair of dg cofunctors $(F, G)$, the graded vector space $\coder(F, G)$ inherits its differential via the usual formula
\begin{displaymath}
\begin{tikzcd}
D \arrow[mapsto]{r} & 	d_{\cD} \circ D - (-1)^{|D|} D \circ d_{\cC}.
\end{tikzcd} 
\end{displaymath}
\noindent As before, dg cofunctors induce filtered cochain maps between spaces of coderivations. Dual to the case of dg functors, there is a natural notion of homotopy for dg cofunctors.
\begin{definition}
Let $F, G: \cC \rightarrow \cD$ be dg cofunctors between dg cocategories. Then, $F$ and $G$ are \textbf{homotopic} if $F(X)=G(X)$ for all $X \in \Ob{\cC}$ and there exists an $(F, G)$-coderivation $H$ of degree $-1$ such that
\begin{displaymath}
F - G = d_{\cD} \circ H + H \circ d_{\cC}.
\end{displaymath}
That is, for all $X, Y \in \Ob{\cC}$, $H: \cC(X,Y) \rightarrow \cD(X,Y)$ provides a cochain homotopy between the maps $F, G: \cC(X,Y) \rightarrow \cD(X,Y)$. 
\end{definition}

\subsection{The Hochschild space of a $\Bbbk$-graph}\label{SectionHochschildSpaceGraph}\ \medskip

\noindent For any $\Bbbk$-graph $\cG$ and any integer $n \in \mathbb{Z}$, we denote by $\cG[n]$ the $\Bbbk$-graph with the same objects and graded vector spaces $\cG(X,Y)[n]$, $X, Y \in \Ob{\cG}$.

\begin{definition}\label{DefinitionHochshildSpaceGraph}
	Let $\cG$ be a $\Bbbk$-graph. Its \textbf{Hochschild space} $C(\cG)$ is the graded vector space whose $n$-th homogeneous component is
	\begin{displaymath}
		\begin{aligned}
			C^n(\cG)  & = \prod_{p=0}^{\infty} \prod_{U_0, \dots, U_p \in \Ob{\cG}} \Hom_\Bbbk^{n}\Big(\cG(U_0, U_1)[1] \otimes \cdots \otimes \cG(U_{p-1}, U_p)[1], \cG(U_0,U_p)[1]\Big).
		\end{aligned}
	\end{displaymath}
\end{definition}
\noindent A homogeneous function $f: \cG(U_0, U_1) \otimes \cdots \otimes \cG(U_{d-1}, U_d) \rightarrow \cG(U_0,U_d)$ of degree $n$ from a morphism space between the \textit{unshifted} tensor categories is canonically identified with an element in $C(\cG)$ of degree $d-1+n$. We refer to elements in $C(\cG)$ of this kind as \textbf{elementary} and to the integer $w(f)\coloneqq d$ as their \textbf{weight}.  Elements in $C(\cG)$ are determined by their components, which are elementary functions. The bijection $\phi \mapsto D_{\phi}$ from  \Cref{PropositionExtensionToDerivations} yields an isomorphism of graded vector spaces 
\begin{displaymath}
	\coder(\operatorname{Id}_{T(\cG[1])}) \cong C(\cG),
\end{displaymath}
\noindent under which  the canonical filtration on $\coder(\operatorname{Id}_{T(\cG)[1]})$ corresponds to the filtration
\begin{displaymath}
	C(\cG)=W_0C(\cG) \supseteq W_1C(\cG) \supseteq \dots,
\end{displaymath}
\noindent called the \textbf{weight filtration}. In explicit terms, $W_iC(\cG)$ consists of all products of elementary functions of weight at least $i$. In particular, the filtration endows $C(\cG)$  with the structure of a complete normed topological vector space.

\subsection{The brace algebra structure on the Hochschild space}\label{SectionBraceAlgebraStructureHochschildSpace}\ \medskip

\noindent As in the case of a single vector space, the Hochschild space of a $\Bbbk$-graph carries the structure of a brace algebra. The general case is reduced to this special one. To do so, let $\Bbbk[\cC]$ denote the category algebra associated to a graded $\Bbbk$-linear category $\cC$, namely
\begin{displaymath}
	\Bbbk[\cC] \coloneqq \bigoplus_{X,Y \in \Ob{\cC}} \Hom_{\cC}(X,Y),
\end{displaymath}
\noindent which inherits an obvious multiplication from the composition in $\cC$. If $\cC$ is just a $\Bbbk$-graph, the same definition still makes sense and provides a functor from $\Graph$ into the category of graded vector spaces. The collection of projections $\pi^{XY}:\Bbbk[\cG] \twoheadrightarrow \cG(X,Y)$ and inclusions $\iota^{XY}: \cG(X,Y) \hookrightarrow \Bbbk[\cG]$ induce a canonical injection\footnote{This is already explained in \cite{KellerInvarianceHigherStructures}.} of graded vector spaces
\begin{equation}\label{EquationInclusionHochschildCategoryAlgebra}
	C(\cG) \hookrightarrow C(\Bbbk[\cG]),
\end{equation}
\noindent which maps every elementary function $f: \cG(U_0, U_1) \otimes \cdots \otimes \cG(U_{d-1}, U_d) \rightarrow \cG(U_0,U_d)$ to $\iota^{U_0U_d} \circ f \circ (\pi^{U_0U_1} \otimes \cdots \otimes \pi^{U_{d-1} U_d})$.
\begin{definition}Let $\cG$ be a $\Bbbk$-graph. The brace algebra structure on $C(\cG)$ is the restriction of the brace algebra structure on $C(\Bbbk[\cG])$ via \eqref{EquationDefinitionBracesHochschildComplex} along the inclusion\footnote{Because $\pi^{XY}$ and $\iota^{XY}$  preserve degrees, the inclusion does not introduce additional signs to the Koszul signs in \eqref{EquationDefinitionBracesHochschildComplex}.} $C(\cG) \subseteq C(\Bbbk[\cG])$.
\end{definition}  
\noindent It follows that $C(\cG)$ is a pre-Lie algebra by virtue of the composition product $f \star g=f \{g\}$. Its graded Lie algebra has another incarnation in terms of coderivations.

\begin{lem}\label{LemmaIsomorphismLieAlgebrasStarProductDerivations}
For every $\Bbbk$-graph $\cG$, the linear isomorphism $\coder(\operatorname{Id}_{T(\cG[1])}) \cong C(\cG)$ from  \Cref{PropositionExtensionToDerivations} is an isomorphism of graded Lie algebras.
\end{lem}
\begin{proof}
Follows from \eqref{EquationInducedMapsInteractionWeightFiltrations} which  shows that  $\pi_1 \circ (D_\phi \circ D_\psi) = \phi \star \psi$ for the $\operatorname{Id}_{T(\cG)}$-coderivations associated to sets of functions $\phi$ and $\psi$. 
\end{proof}
\noindent Finally, for later reference we record the following observation which follows by inspection of the definitions.
\begin{lem}\label{LemmaHochschildSpaceBraceRespectFiltration}
The brace operations respect the index shifted weight filtration of the Hochschild complex, that is, for $W_i^{[1]}C(\cG)\coloneqq W_{i+1}(\cG)$ and  $g_0, \dots, g_r \in C(\cG)$ with $g_j \in W_{i_j}^{[1]}C^{n_j}(\cG)$, one has
\begin{displaymath}
g_0 \big\{g_1, \dots, g_r\big\} \in W_{i_0+ \cdots + i_r}^{[1]}C^{n_0 + \cdots + n_r}(\cG).
\end{displaymath}
\end{lem}

\subsection{Restricted naturality}\ \medskip

\noindent  The Hochschild space satisfies a restricted naturality property with respect to inclusions. Let $\iota: \cG \rightarrow \cH$ be a fully faithful inclusion of $\Bbbk$-graphs, by which we mean a morphism which is injective on objects and for which the associated maps $\iota_{XY}: \cG(X,Y) \rightarrow \cH(X,Y)$ are isomorphisms for all $X, Y \in \Ob{\cG}$. Then, one obtains a $\Bbbk$-linear graded restriction map
\begin{displaymath}
\begin{tikzcd}
	\res(\iota): C(\cH) \arrow[twoheadrightarrow]{r} & C(\cG),
	\end{tikzcd}
\end{displaymath}
\noindent which is in fact a morphism of brace algebras due to the commutative diagram
\begin{equation}
	\begin{tikzcd}
		C(\cH) \arrow{d}[swap]{\res(\iota)} \arrow[hookrightarrow]{r} &  C(\Bbbk[\cH]) \arrow{d}{\res(\Bbbk[\iota])} \\
		C(\cG) \arrow[hookrightarrow]{r} &  C\big(\Bbbk[\cG]\big).
	\end{tikzcd}
\end{equation}

\section{\texorpdfstring{Reminder on $A_\infty$-categories}{Reminder on A-infinity categories}}

\noindent We recall the relevant theory of $A_\infty$-categories including functor categories and Hochschild complexes.
\subsection{$A_\infty$-categories via Maurer-Cartan elements in Hochschild spaces}\label{SectionHochschildComplex} \label{SectionAInfinityCategories}\ \medskip

\begin{definition}
Let $V$ be a pre-Lie algebra. An element $\mu \in V^1$ is a \textbf{Maurer-Cartan element} if $\mu \star \mu=0$.
\end{definition}
\begin{definition}A (flat) \textbf{$A_\infty$-structure} on a $\Bbbk$-graph $\cG$ is a Maurer-Cartan element $\mu \in W_1C^1(\cG)$ of the underlying pre-Lie algebra. A (small) \textbf{$A_\infty$-category} is a pair $(\bA, \mu_{\bA})$ consisting of an $A_\infty$-structure $\mu_{\bA}$ on a $\Bbbk$-graph $\bA$.
\end{definition}
\noindent By \Cref{LemmaIsomorphismLieAlgebrasStarProductDerivations}, then $\mu \in W_1C(\cG)$ is an $A_\infty$-structure on $\cG$ if and only if the corresponding degree $1$ coderivation $D_{\mu} \in \coder(\operatorname{Id}_{T(\cG[1])})$ satisfies $D_{\mu} \circ D_{\mu}=0$. Hence one obtains the following.
\begin{cor}
The correspondence $\mu \mapsto \widehat{\mu}$ induces a bijection between $A_\infty$-structures on a $\Bbbk$-graph $\cG$ and dg cocategory structures on $T(\cG[1])$. 
\end{cor}
\noindent If $\bA$ is an $A_\infty$-category we occasionally also write $\Hom_{\bA}(-,-)$ for $\bA(-,-)$. The structure element $\mu_{\bA}$ is equivalent to a collection of functions 
\begin{displaymath}
\mu^i_{\bA}: \bA(A_{i-1}, A_i)[1] \otimes \cdots \otimes \bA(A_0, A_1)[1] \rightarrow \bA(A_0, A_i)[1],
\end{displaymath}
\noindent  of degree $1$ for each sequence of objects $A_0, \dots, A_i \in \bA$, $i \geq 1$. The Maurer-Cartan condition translates into the well-known quadratic associativity relations
\begin{displaymath}
0 =	\sum_{0 \leq s < r \leq n} \mu_{\bA}^{n+s-r+1}\Big(\mathbf{1}_{n-1, n} \otimes \cdots \otimes \mathbf{1}_{r-1, r} \otimes \mu_{\bA}^{r-s-1} \otimes \mathbf{1}_{s-1, s} \otimes  \cdots \otimes \mathbf{1}_{0, 1}\Big), 
\end{displaymath}
\noindent for each sequence of objects $A_0, \dots, A_n \in \bA$, where $\mathbf{1}_{i, i+1}$ is short for $\operatorname{Id}_{\bA(A_i, A_{i+1})[1]}$.  After desuspension the operations $\mu_{\bA}^i$ are equivalent to functions 
\begin{displaymath}
m_{\bA}^i: \bA(A_{i-1}, A_i) \otimes \cdots \otimes \bA(A_0, A_1) \rightarrow \bA(A_0, A_i),
\end{displaymath}
\noindent of degree $2-i$ which satisfy the $A_\infty$-equations in the sign conventions of \cite{KellerAInfinityAlgebrasInRepresentationTheory, GetzlerJones} after evaluation, see \cite[Section 1.2, Lemma]{KellerAInfinityAlgebrasInRepresentationTheory}. Due to the Koszul sign rule, $\mu_{\bA}^i$ and $m_{\bA}^i$ are related by
\begin{equation}\label{EquationSignRelationDesuspendedOperations}
\mu_{\bA}^i(a_i, \dots, a_1) = (-1)^{\sum_{n=1}^{i}{(n-1)|a_n|}} m_{\bA}^i\big(a_i[-1], \dots, a_1[-1]\big)[1].
\end{equation}
\noindent In the literature, the structure provided by the $\mu_{\bA}^i$ is sometimes referred to as a \textit{shifted $A_\infty$-algebra} whereas the collection of their desuspension form what is usually known as an $A_\infty$-structure. For our purposes, shifted structures are more convenient to work with and to distinguish the two, we occasionally use the term \textit{unshifted} to refer to the latter structure provided by the operations $m_{\bA}^i$.

\begin{exa}[Dg categories and $\Bbbk$-linear categories] The composition law and differentials in a dg category $\bD$ correspond to the desuspended operations $m^2_{\bD}$ and $m^1_{\bD}$ associated to an $A_\infty$-category structure $\mu$ on $\bD$ with $\mu^i=0$ for all $i \geq 3$.  Likewise, graded $\Bbbk$-linear categories correspond to the vanishing of all operations except the binary one.
\end{exa}

\begin{exa}[$A_\infty$-algebra of a brace algebra]\label{ExampleAInfinityStructureFromBraceAlgebra}
Let $V$ be a  brace algebra and let $\mu \in V^1$ be a Maurer-Cartan element, that is, $\mu \star \mu=0$. 
Then as observed by Getzler \cite{GetzlerCartanHomotopy}, $V$ inherits the structure of a (shifted) $A_\infty$-algebra with operations 
\begin{equation}\label{EquationAInfinityStructureHochscchildComplex}
\begin{aligned}
\mu_V^1(f) & \coloneqq \big[\mu, f\big] =\mu \star f - (-1)^{|f|} f \star \mu,\\
\mu_V^i(f_1, \dots, f_i) & \coloneqq \mu\{f_1, \dots, f_i\}, \, i \geq 2.
\end{aligned}
\end{equation}
\noindent In other words, every Maurer-Cartan element in a brace algebra endows it with two compatible structures: one of a dg Lie algebra and one of an $A_\infty$-algebra. 
\end{exa}

\begin{definition}[Minimal $A_\infty$-categories]
An $A_\infty$-category $\bA$ is \textbf{minimal} if $\mu_{\bA}^1=0$ identically.
\end{definition}

\begin{definition}[Strict unitality]
An $A_\infty$-category $\bA$ is \textbf{strictly unital} if for all $A \in \Ob{\bA}$, there exists an element $\operatorname{Id}_{A} \in \bA(A,A)^0$ such that for all $f \in \bA(A,A')$, $m_{\bA}^2(f, \operatorname{Id}_{A})=f=m_{\bA}^2(\operatorname{Id}_{A'}, f)$ and $m_{\bA}^i(f_n, \dots, f_1)=0$ whenever $i\geq 3$ and one of the entries is the form $\operatorname{Id}_{A}$. 
\end{definition}
\noindent The identity morphisms in a strictly unital $A_\infty$-category are necessarily unique.

\subsection{Homotopy categories and $c$-unitality}\ \medskip

\noindent For any $A_\infty$-category $\bA$ and every pair $A,A' \in \Ob{\bA}$, the operation $m_{\bA}^1: \bA(A,A') \rightarrow \bA(A,A')$ turns $\bA(A,A')$ into a cochain complex and the non-associative composition law $\mu_{\bA}^2$ descends to cohomology with respect to $m_{\bA}^1$ and becomes associative.
\begin{definition}
The \textbf{homotopy category} $\HH^0(\bA)$ of an $\bA$ is the non-unital $\Bbbk$-linear category with objects $\Ob{\bA}$ and morphisms $\HH^0(\bA)(A,A')\coloneqq\HH^0(\bA(A,A'))$ for all $A, A' \in \Ob{\bA}$. Its composition law is the one inherited from $m_{\bA}^2$. The \textbf{graded homotopy category} $\HH^{\bullet}(\bA)$ is the graded $\Bbbk$-linear category defined analogously with $\HH^{\bullet}(\bA)(A,A')\coloneqq\HH^{\bullet}(\bA(A,A'))$.
\end{definition}
\begin{rem}\label{RemarkUnderlyingGradedCategory}
If $\mu_{\bA}^3$ or $\mu_{\bA}^1$ vanish identically, then $m_{\bA}^2$ is associative and $\bA$ itself is a graded $\Bbbk$-linear category by forgetting all other operations which agrees with its homotopy category.
\end{rem}

\noindent The homotopy category of a strictly unital $A_\infty$-category $\bA$ is unital and hence an ordinary graded $\Bbbk$-linear category. Often this is the case even when $\bA$ itself is not strictly unital. Examples of such kind are Fukaya categories of symplectic manifolds which usually only admit units up to homotopy. 
\begin{definition}
An $A_\infty$-category $\bA$ is \textbf{cohomologically unital} (or \textbf{$c$-unital}) if $\HH^0(\bA)$ is unital.
\end{definition}
\noindent As we recall in \Cref{PropositionYonedaRectification}, the difference between strict and $c$-unitality is more important from a practical than a theoretical perspective.

\subsection{$A_\infty$-functors and functor categories}\ \medskip

\noindent We recall the notions of $A_\infty$-functors between $A_\infty$-categories. If the domain category is small they form the objects of an $A_\infty$-category. We also recall several notions of equivalences between $A_\infty$-functors. 

\subsubsection{$A_\infty$-functors}\ \medskip 

\noindent Throughout this section let $\bA$ and $\bB$ denote $A_\infty$-categories.
\begin{definition}
An $A_\infty$-functor $F: \bA \rightarrow \bB$ is a dg cofunctor $B(F):T(\bA[1]) \rightarrow T(\bB[1])$.
\end{definition}
\noindent Explicitly, an $A_\infty$-functor $F$ amounts to a map of sets $F^0:\Ob{\bA} \rightarrow \Ob{\bB}$ along with maps 
\begin{displaymath}F^i: \bA(A_{i-1}, A_i)[1] \otimes \cdots \bA(A_0, A_1)[1] \rightarrow \bB\big(F^0(A_0, F^0(A_i)\big)[1],
\end{displaymath}
\noindent of degree $0$ for all $i \geq 1$ and $A_0, \dots, A_i \in \Ob{\bA}$. The collection of the various maps $F^i$ for all sequences $A_0, \dots, A_i \in \Ob{\bA}$ is referred to the \textbf{$i$-th Taylor coefficient} of $F$ and is simply denoted by $F^i$. The condition that $B(F)$ commutes with differentials is equivalent to the condition that for all $n \geq 1$ and all $A_0, \dots, A_n \in \Ob{\bA}$,
\begin{equation}\label{EquationAInfinityFunctors}
	\begin{aligned}
\sum \mu_{\bA}^{d}\big(F^{i_1} \otimes \cdots \otimes F^{i_d}\Big)  & = \sum_{0 \leq s < r \leq n} F^{n+s-r+1}\Big(\mathbf{1}_{n-1,n} \otimes \cdots \otimes  \mathbf{1}_{r-1,r} \otimes \mu_{\bA}^{r-s-1}  \otimes \\ & \phantom{=} \mathbf{1}_{s-1,s} \otimes \cdots \otimes \mathbf{1}_{0,1}\Big),
\end{aligned} 
\end{equation}
\noindent where $\mathbf{1}_{i,i+1}=\operatorname{Id}_{\bA(A_i, A_{i+1})[1]}$ as before and the sum on the left hand side ranges over all $d, i_1, \dots, i_d \geq 1$ such that $\sum_{i_j}=n$. We generally refer to an $A_\infty$-functor $F: \bA \rightarrow \bB$ through its associated Taylor coefficients $\{F^i\}_{i \geq 0}$. Every $A_\infty$-functor $F: \bA \rightarrow \bB$ induces cochain maps 
\begin{equation}\label{EquationCochainMapsFunctor}
	F^1: \bA(A,A') \rightarrow \bB\big(F^0(A),F^0(A')\big),
\end{equation}
\noindent for all $A, A' \in \Ob{\bA}$ and hence non-unital functors $\HH^{\bullet}(F): \HH^{\bullet}(\bA) \rightarrow \HH^{\bullet}(\bB)$ and $\HH^{0}(F): \HH^{0}(\bA) \rightarrow \HH^{0}(\bB)$.

\begin{definition}[Composition of Functors]
	Let  $F: \bU \rightarrow \bV$, $G:\bV \rightarrow \bW$ be $A_\infty$-functors. Then their composition  is the $A_\infty$-functor $G \circ F: \bU \rightarrow \bW$ whose Taylor coefficients are ${(G \circ F)}^0=G^0 \circ F^0$ and 
	\begin{equation}\label{EquationCompositionFunctors}
		(G \circ F)^d \coloneqq \sum {G^d\left( F^{i_1} \otimes \cdots \otimes F^{i_d}\right)},
	\end{equation}

	\noindent for all $n \geq 1$, where the sum ranges over all $d, i_1, \dots, i_d \geq 1$, $\sum_{j=1}^d{i_j}=n$ and where as usual the definition is applied to all sequences $A_0, \dots, A_n \in \Ob{\bA}$.
\end{definition}

\noindent Of course, the cofunctor $B(G \circ F)$ associated to $G \circ F$ is nothing but the composition of the cofunctors $B(G)$ and $B(F)$.  Subsequently, we recall a number of properties of $A_\infty$-functors which will appear throughout the paper.

\begin{definition}
Let $F: \bA \rightarrow \bB$ be an $A_\infty$-functor between $A_\infty$-categories.
\begin{enumerate}
	\item The functor $F$ is \textbf{strict} if $F^i=0$ for all $i \geq 2$.
	\item  If $\bA$ and $\bB$ are strictly unital, then $F$ is \textbf{strictly unital} if $F^1(\operatorname{Id}_A)=\operatorname{Id}_{F^0(A)}$ for all $A \in \Ob{\bA}$ and for all $i \geq 2$, $F^i$ vanishes on all elementary tensors which contain at least one identity morphism.
	\item If $\bA$ and $\bB$ are cohomologically unital, then $F$ is \textbf{cohomologically unital} (or \textbf{$c$-unital}) if $\HH^{0}(F)$ is unital. 
	\item   An $A_\infty$-functor $F$ is a \textbf{quasi-isomorphism} if $\HH^{\bullet}(F)$ is an isomorphism of non-unital categories. A cohomologically unital $A_\infty$-functor $G$ is a \textbf{quasi-equivalence} if  $\HH^{\bullet}(G)$ is an equivalence of categories.
\end{enumerate}
\noindent We denote by $\Acatc$ the category of small $c-$unital $A_\infty$-categories with $c$-unital $A_\infty$-functors as morphisms.
\end{definition}
\noindent Equivalently, $G$ is a quasi-equivalence if \eqref{EquationCochainMapsFunctor} is a quasi-isomorphism for all $A, A' \in \Ob{\bA}$ and $\HH^0(G)$ is essentially surjective.

\subsubsection{Rectification of $A_\infty$-categories}\ \medskip 

\noindent  We collect two results which allow one to replace $A_\infty$-categories by ``better'' ones up to a notion of equivalence.

\begin{prp}[{\cite[Corollary 2.14]{SeidelBook}}]\label{PropositionYonedaRectification}Every cohomologically unital $A_\infty$-category is quasi-isomorphic to a dg category.
\end{prp}
\noindent The proposition follows from an $A_\infty$-version of the Yoneda embedding which produces a quasi-isomorphism between an $A_\infty$-category and its image which is a dg category. In particular, every cohomologically unital $A_\infty$-category is quasi-isomorphic to a strictly unital one. In the other direction, we also have the following.

\begin{prp}[Minimal models]\label{PropositionExistenceMinimalModel}Every $A_\infty$-category $\bA$ is quasi-isomorphic to a minimal $A_\infty$-category.
\end{prp}

\subsubsection{Functor categories}\label{SectionFuntorCategories}\ \medskip

\noindent We fix $A_\infty$-categories $\bA, \bB$ throughout this section. The class of  $A_\infty$-functors from $\bA$ to $\bB$ forms itself an $A_\infty$-category which we denote by  $\Fun(\bA, \bB)$. The morphism space $\Fun(\bA,\bB)(F,G)$ between $A_\infty$-functors $F$ and $G$ consists of the prenatural transformations defined below. To facilitate better readability of formulas we will employ the following notation.  Given a family of $\Bbbk$-linear maps $\cF=\{f^d\}_{d \in \mathbb{N}}$, e.g.~$A_\infty$-functors or the structure maps of an $A_\infty$-structure, and $s \geq 1$, we write $\cF^{[s]}$ for the sum
\begin{displaymath}
\sum f^{t_1} \otimes \cdots \otimes f^{t_l},
\end{displaymath}
\noindent which ranges over all $l \geq 1$ and all $t_1, \dots, t_l$ such that $\sum t_j=s$. 

\begin{definition}\label{DefinitionFunctorCategory} Let $F, G: \bA \rightarrow \bB$ be $A_\infty$-functors. A \textbf{prenatural transformation} $\eta: F \rightarrow G$ of degree $m$ is a collection $(\eta^i)_{i \geq 0}$ of $\Bbbk$-linear maps 
	\begin{displaymath}
		\begin{tikzcd}
			\eta^i: \bA(X_{i-1}, X_i)[1] \otimes \cdots \otimes \bA(X_0, X_1)[1] \arrow{r} & \bB\left(F^0(X_0), G^0(X_d)\right)[1],
		\end{tikzcd}
	\end{displaymath}
	\noindent of degree $m$ for all $A_0, \dots, A_i \in \Ob{\bA}$ and all $i \geq 1$ and a map $\eta^0$, which associates to every $A \in \Ob{\bA}$ an element  $\eta^0(A) \in \bB\big(F^0(A), G^0(A)\big)$ of degree $m$. A prenatural transformation $\eta$ is \textbf{strict} if $\eta^d=0$ for all $d > 0$.
	\end{definition}
\noindent In particular, we obtain a canonical bijection of graded vector spaces
\begin{equation}\label{EquationBijectionMorphismsFunctorCategoryCoderivations}
	\Fun(\bA, \bB)(F,G) \cong \coder(F, G).
\end{equation}
\noindent The $A_\infty$-structure on the functor category is defined as follows.
\begin{definition}
 For $F, G \in \Ob{\Fun(\bA, \bB)}$ and $\eta: F \rightarrow G$ a prenatural transformation, set \begin{equation}\label{EquationDifferentialPrenaturalTransformation}
	\begin{aligned}
		{\mu_{\Fun(\bA,\bB)}^1(\eta)} & \coloneqq \sum  \mu_{\bB}^r\Big(G^{[s]} \otimes \mathbf{1} \otimes F^{[t]} \Big)(\eta) - (-1)^{|\eta|} \cdot \eta \star \mu_{\bA},
	\end{aligned}
\end{equation}
\noindent where the sum ranges over all $r \geq 1$ and all $s, t$ such that $s+t=r-1$.  
Similarly, for all $i \geq 2$ and functors $F_0, \dots, F_i \in \Ob{\Fun(\bA, \bB)}$, one defines
\begin{equation}\label{EquationDefinitionHigherMultiplicationFunctorCategory}
	\mu_{\Fun(\bA,\bB)}^i \coloneqq \sum \mu_{\bB}^r\Big(F_i^{[s_i]} \otimes \mathbf{1}_{i-1, i} \otimes F_{i-1}^{[s_{i-1}]} \otimes \cdots \otimes \mathbf{1}_{0,1} \otimes F_0^{[s_0]} \Big),
\end{equation}
\noindent where the sum ranges over all $r \geq 1$ and all $s_1, \dots, s_i \geq 0$ such that $\sum s_j=r-i$ and $\mathbf{1}_{j,j+1}$ denotes the identity map of $\Hom_{\Fun(\bA,\bB)}(F_j, F_{j+1})$. 
\begin{itemize}
	\item A \textbf{natural transformation} of degree $m$ is a cocycle of degree $m$.
	\item If $\bB$ is cohomologically unital, then a natural transformation $\eta: F \rightarrow G$ of degree $0$ is a \textbf{natural quasi-isomorphism} if $\eta^0(A)$ is a quasi-isomorphism for all $A \in \Ob{\bA}$.
\end{itemize} 
\end{definition}
\noindent With the given $A_\infty$-structure on $\Fun(\bA, \bB)$, the bijection \eqref{EquationBijectionMorphismsFunctorCategoryCoderivations} becomes an isomorphism of cochain complexes.   We note that if $\bB$ is strictly (resp.~cohomologically unital), then so is $\Fun(\bA, \bB)$, see \cite[Section 2b]{SeidelBook}.\medskip

\noindent Equation \eqref{EquationDifferentialPrenaturalTransformation} implies that for every natural transformation $\eta: F \rightarrow G$ with $F,G \in \Fun(\bA, \bB)$ and all $X, Y \in \Ob{\bA}$, one has the following homotopy commutative square in the category of cochain complexes:

\begin{equation}\label{EquationNaturalTransformationSquare}
	\begin{tikzcd}[row sep=3em]
		\bA(X,Y) \arrow{rrr}{F^1} \arrow{d}[swap]{G^1}&&& \bB(F^0(X), F^0(Y)) \arrow{d}{\mu^2_{\bB}\big(\eta^0(Y), -\big)} \\
		\bB(G^0(X), G^0(Y)) \arrow{rrr}[swap]{\mu^2_{\bB}\big(-, \eta^0(X)\big)} &&& \bB(F^0(X), G^0(Y)). 
	\end{tikzcd}
\end{equation}
\noindent Moreover, the homotopy class of the chain maps induced by $\eta^0(X)$ and $\eta^0(Y)$, only depend on homotopy class of $\eta$.

 Like $A_\infty$-categories and their functors, the $A_\infty$-structure of $\Fun(\bA, \bB)$ can be expressed in a natural way through tensor cocategories as described in \cite{KellerFunctorCategories, BespalovLyubashenkoManzyuk}. Functor categories are part of a much richer structure of a \textit{closed multicategory} on the category of small $A_\infty$-categories as established in \cite{BespalovLyubashenkoManzyuk}. As such, the underlying closed category comes with ``postcomposition'' $A_\infty$-functors
\begin{equation}\label{EquationPostcompositionFunctors}
	\begin{tikzcd}
\Fun(\bV, \bW) \arrow{r} & \Fun\big(\!\Fun(\bU, \bV), \Fun(\bU, \bW)\big),
\end{tikzcd}
\end{equation}
\noindent as well as ``precomposition'' $A_\infty$-functors
\begin{equation}\label{EquationPrecompositionFunctors}
	\begin{tikzcd}
		\Fun(\bU, \bV) \arrow{r} & \Fun\big(\!\Fun(\bV, \bW), \Fun(\bU, \bW)\big).
	\end{tikzcd}
\end{equation}
\noindent An explicit definition via cocategories can be found in \cite{KellerFunctorCategories}. On objects, they are defined by post- and precomposition of $A_\infty$-functors and their first Taylor coefficients agree with the maps \eqref{EquationNaturalityCoderivations} under the identification of prenatural transformations with coderivations.

\subsubsection{Equivalence relations for $A_\infty$-functors}\ \medskip

\noindent The $A_\infty$-structure of $\Fun(\bA, \bB)$ gives rise to two equivalence relations between $A_\infty$-functors.

\begin{definition}Let $F, G \in \Ob{\Fun(\bA, \bB)}$.
	\begin{enumerate}
		\item If $F^0=G^0$, then $F$ and $G$ are \textbf{homotopic}, denoted by $F \sim G$, if there exists a prenatural transformation $h: F \rightarrow G$ of degree $-1$  (called \textbf{homotopy}) with $h^0(A)=0$ for all $A \in \Ob{\bA}$ such that 
		\begin{displaymath}
	{\mu_{\Fun(\bA,\bB)}^1(h)}^d=G^d-F^d,
		\end{displaymath}
	 \noindent for all $d \geq 1$.
		\item If $\bB$ is cohomologically unital,  then $F$ and $G$ are \textbf{weakly equivalent}, denoted by $F\approx G$, if $F \cong G$ in $\HH^0\big(\!\Fun(\bA, \bB)\big)$.	
	\end{enumerate}
\end{definition}
\noindent  The notion of homotopy correspond the notion of homotopy for cofunctors between the resulting tensor dg cocategories and $\sim$ and $\approx$ are both equivalence relations. We record a few properties of these two notions. 
\begin{prp}\label{PropositionCompositionCompatibelWithEquivalence}
Homotopy is preserved under pre- and postcomposition with any $A_\infty$-functor. Weak equivalence is preserved under pre- and postcomposition with any cohomologically unital $A_\infty$-functor. Thus, composition descends to a well-defined operations on homotopy and weak equivalence classes.
\end{prp}
\begin{proof}
This is the content of \cite[Section 1h]{SeidelBook} and \cite[Lemma 2.4]{SeidelBook}.
\end{proof}
\noindent In case both notions are defined, homotopy implies weak equivalence.
\begin{lem}[{\cite[Lemma 2.5]{SeidelBook}}]\label{LemmaHomotopicFunctorsWeaklyEquivalent}If $\bB$ is cohomologically unital, then homotopic functors in $\Fun(\bA, \bB)$ are weakly equivalent.
\end{lem}
\noindent  The proof in the strictly unital case roughly goes as follows. Using that higher products of tensors which contain a identity morphism vanish, one sees that every homotopy $h$ between functors $F$ and $G$ determines an invertible natural transformation $\eta: F \rightarrow G$ defined by $\eta^0(A)=\operatorname{Id}_A$ for all $A \in \Ob{\bA}$ and  $\eta^d\coloneqq \mu{\Fun(\bA, \bB)}(h)^d$ for all $d \neq 1$.\medskip
  
\noindent We call functors \textbf{homotopy} (resp.~\textbf{weakly}) \textbf{invertible} if they admit an inverse up to homotopy (resp.~weak equivalence). Examples abound as the following proposition shows.
\begin{prp}[Inverse Function Theorems]\label{PropositionInverseFunctionTheorem}
	Every quasi-equivalence between $A_\infty$-categories is weakly invertible and every quasi-isomorphism is homotopy invertible.
\end{prp}
\begin{proof}
See \cite[Corollary 1.14, Theorem 2.9]{SeidelBook}.
\end{proof}

\begin{definition}
Let $\bA$ be an $A_\infty$-category.
\begin{enumerate}
	\item We denote by $\Aut^{\infty,h}(\bA)$ the group of homotopy classes of homotopy invertible $A_\infty$-endofunctors of $\bA$.
	\item If $\bA$ is cohomologically unital, we denote by $\Aut^{\infty}(\bA)$ the group of weak equivalence classes of weakly invertible  $A_\infty$-endofunctors of $\bA$. 
\end{enumerate}
\end{definition}
\noindent Of course, every homotopy invertible functor $F: \bA \rightarrow \bB$ induces an isomorphism $\Aut^{\infty, h}(\bA) \cong \Aut^{\infty, h}(\bB)$ of groups given by conjugation with $F$ and its homotopy inverse. In the same way, every weakly invertible $F$ induces a group isomorphism $\Aut^{\infty}(\bA)\cong \Aut^{\infty}(\bB)$. If $\bA$ is cohomologically unital, then \Cref{LemmaHomotopicFunctorsWeaklyEquivalent} implies that there is a canonical surjective homomorphism of groups
\begin{displaymath}
	\begin{tikzcd}
		\Aut^{\infty, h}(\bA) \arrow[twoheadrightarrow]{r} & \Aut^{\infty}(\bA).
	\end{tikzcd}
\end{displaymath}
\noindent In general, there seems to be no reason to expect the map to be injective. However, we will see later that it is true under certain conditions and for certain subsets of both groups.

\subsection{The Hochschild $A_\infty$-algebra of an $A_\infty$-category}\label{SectionHochschildAlgebraAInfinityCategory}

\subsubsection{$A_\infty$-structure and Hochschild cohomology}\ \medskip

\noindent The Hochschild space $C=C(\bA)$ of an $A_\infty$-category $\bA$ is by definition the Hochschild space of its underlying $\Bbbk$-graph. Via its brace algebra structure it inherits the structure of a (shifted) $A_\infty$-algebra via \eqref{EquationAInfinityStructureHochscchildComplex}.
\noindent We note  that if $\mu_\bA^i$ vanishes for all $i \geq 3$, then $\bA[-1]$ is a unshifted dg category and $C[-1]$ is an unshifted dg algebra. The Lie bracket of the pre-Lie algebra is known as the \textbf{Gerstenhaber bracket} and descends to $\HH^{\bullet}(C)$. The product $m_C^2: C \otimes C \rightarrow C$ is referred to as the \textbf{cup product}.\medskip

\noindent  The chain complex $\big(C(\bA), m_{C(\bA)}^1\big)$ generalizes the Hochschild complex of an ordinary associative $\Bbbk$-algebra.

\begin{definition}
	Let $\bA$ be a cohomologically unital $A_\infty$-category and $i \in \mathbb{Z}$. The $i$-th  \textbf{Hochschild cohomology} of $\bA$ is the cohomology 
	\begin{displaymath}
		\HHH^i(\bA, \bA) \coloneqq \HH^{i-1}\big(C(\bA)\big),
	\end{displaymath}
	\noindent with respect to the differential on $C(\bA)$.
\end{definition}
\noindent From \Cref{LemmaHochschildSpaceBraceRespectFiltration} and $\mu_{\bA} \in W_1C(\bA)$ it follows that the weight filtration on $C(\bA)$ descends to a filtration
\begin{displaymath}
	\HHH^{\bullet}(\bA,\bA)=W_0\HHH^{\bullet}(\bA,\bA) \supseteq W_1\HHH^{\bullet}(\bA,\bA) \supseteq \cdots 
\end{displaymath} on Hochschild cohomology. We refer to it as the \textbf{cohomological weight filtration} or simply weight filtration by abuse of terminology. Of particular interest with regards to integration of Hochschild classes will be the second filtration step
\begin{displaymath}
\HHH^{\bullet}_+(\bA, \bA) \coloneqq W_2\HHH^{\bullet}(\bA, \bA).
\end{displaymath}
\noindent The reason for the symbol ``+'' is related to the fact that $W_2C(\bA)$ is the \textit{positive} part of the index shifted weight filtration which is the correct filtration for the pre-Lie algebra structure on $C(\bA)$, cf.~\Cref{LemmaHochschildSpaceBraceRespectFiltration} and \Cref{SectionInfinityIsotopiesIntegration}.

\subsubsection{The normalized Hochschild algebra}

\begin{definition}
	Let $\bA$ be a strictly unital $A_\infty$-category. Its \textbf{normalized Hochschild space} is the subspace $\overline{C}(\bA) \subseteq C(\bA)$ consisting of the elements whose elementary components vanish on all elements of the form $f_n \otimes \cdots \otimes  f_1$ with $f_i=\operatorname{Id}_A$ for some $i \in [1,n]$ and some $A \in \Ob{\bA}$.
\end{definition}
\noindent Similar to $C(\bA)$, the normalized complex and its cohomology are filtered by weight. The assumption that $\bA$ is strictly unital translates into the condition $\mu_{\bA} \in \overline{C}(\bA)$. In fact, $\overline{C}(\bA)$ is a brace subalgebra of $C(\bA)$ as can be seen by direct inspection of the defining formulas and it is well-known that the inclusion $\overline{C}(\bA) \subseteq C(\bA)$ is a quasi-isomorphism of complexes. In summary, we record the following.
\begin{prp}
	The inclusion $\overline{C}(\bA) \hookrightarrow C(\bA)$ is a strict, strictly unital quasi-isomorphism of $A_\infty$-algebras. 
\end{prp}

\subsection{Modules and derived categories}\ \medskip

\noindent For any $A_\infty$-category $\bA$, we denote by $\bA^{\operatorname{op}}$ its \textbf{opposite category}. Explicitly, $\bA$ and its opposite have the same objects and the opposite $\Bbbk$-graph $\bA^{\operatorname{op}}(A,B) \coloneqq \bA(B, A)$ for all $A, B \in \Ob{\bA}$. Its $A_\infty$-structure is defined via
\begin{displaymath}
\mu_{\bA^{\operatorname{od}}}^i(f_i \otimes \cdots \otimes f_1) \coloneqq \sigma \cdot \mu_{\bA}^i(f_1 \otimes \cdots f_i),
\end{displaymath}
\noindent where $\sigma$ denotes the Koszul sign of the permutation $f_i \otimes \cdots \otimes f_1 \mapsto f_1 \otimes \cdots \otimes f_i$. We denote by $\cC(\Bbbk)$ the dg category of cochain complexes over $\Bbbk$.
\begin{definition}
Let $\bA$ be an $A_\infty$-category. A (non-unital)  \textbf{right $\bA$-module} is an $A_\infty$-functor $\bA^{\operatorname{op}} \rightarrow \cC(\Bbbk)$.
\end{definition}
\noindent We write $\operatorname{mod}-\bA \coloneqq \Fun(\bA^{\operatorname{op}}, \cC(\Bbbk))$ for the dg category of right $A_\infty$-modules over $\bA$. If $\bA$ is strictly unital (resp.~$c$-unital), one can also define dg subcategories $\operatorname{mod}_u-\bA$ (resp.~$\operatorname{mod}_c-\bA$) of strictly unital (resp.~$c$-unital) modules by restricting to strictly unital (resp.~$c$-unital) functors and prenatural transformations. In particular, one obtains natural notions of  homotopy equivalence and quasi-isomorphism of $A_\infty$-modules.

\begin{definition}
	Let $\bA$ be a strictly unital (resp.~cohomologically unital) $A_\infty$-category. Its \textbf{derived category} $\cD(\bA)$ is the localization of $\HH^0(\operatorname{mod}_u-\bA)$ (resp.~$\HH^0(\operatorname{mod}_c-\bA)$) at the class of natural quasi-isomorphisms. 
\end{definition}
\noindent As expected, if $\bA$ is strictly unital, then both definitions -- via strictly and cohomologically unital  $A_\infty$-modules -- yield equivalent triangulated categories. The derived category is a triangulated category and as in the case of dg categories, there exists embedding $\cY:\HH^0(\bA) \rightarrow \cD(\bA)$ which generalizes the Yoneda embedding, cf.~\cite[Section (1l)]{SeidelBook}. A cohomologically unital $A_\infty$-category $\bA$ is \textbf{pretriangulated} (resp.~\textbf{perfect}) if the essential image of $\cY$ in $\cD(\bA)$ is a triangulated (resp.~thick) subcategory. Every $A_\infty$-category admits a pre-triangulated hull as well as a perfect hull $\Perf(\bA)$ which generalizes the category of perfect complexes. It is uniquely defined up to quasi-equivalence, see \Cref{PropositionUniversalPropertiesHQE}.

\subsection{The category of $A_\infty$-isotopies}\label{SectionAInfinityIsotopies}\label{SectionHomotopyCategoryDGCategories}\ \medskip

\noindent The following special case of $A_\infty$-functors will be of particular importance to us. 

\begin{definition}\label{DefinitionAInfinityIsotopies}
	Let $\bA$ be an $A_\infty$-category. An \textbf{$A_\infty$-isotopy} of $\bA$ is an $A_\infty$-functor $F: \bA \rightarrow \bA$ such that $F^0=\operatorname{Id}_{\Ob{\bA}}$ and such that for all $A,A' \in \Ob{\bA}$, the $F^1: \bA(A,A') \rightarrow \bA(A,A')$ is the identity. 
\end{definition}
\noindent Due to its definition, every $A_\infty$-isotopy of a cohomologically unital $A_\infty$-category is cohomologically unital and the set of all $A_\infty$-isotopies form a group under composition. We denote by $\Aut^{\infty, h}_+(\bA) \subseteq \Aut^{\infty,h}(\bA)$ and $\Aut^{\infty}_+(\bA) \subseteq \Aut^{\infty}(\bA)$, the subgroups associated with the homotopy and weak equivalence classes of $A_\infty$-isotopies.\medskip

\noindent The set of $A_\infty$-isotopies of a small $A_\infty$-category $\bA$ form a full $A_\infty$-subcategory $\Iso(\bA) \subseteq \Fun(\bA,\bA)$ whose $A_\infty$-structure admits a reformulation in terms of the braces on the Hochschild complex. To do so, we observe first that any isotopy $F \in \Iso(\bA)$ is equivalent to the datum of an element $F_+ \in W_2C^0(\bA)$ consisting of the product of its higher Taylor coefficients $F^j$, $j \geq 2$. The $A_\infty$-functor equation \eqref{EquationAInfinityFunctors} for an element $F \in C^0(\bA)$ arising from an element $F_+ \in W_2C^0(\bA)$ in this way can be rewritten as the condition
\begin{equation}\label{EquationSimpliedFunctorEquationIsotopy}
	\sum_{r \geq 0} \mu_{\bA} \{F_+, \dots, F_+\}_r = F_+ \star \mu_{\bA},
\end{equation}
\noindent where $\mu_{\bA}\{\}_0=\mu_{\bA}$. Similarly, the differentials on $\Iso(\bA)$ can be expressed as
\begin{equation}\label{EquationDifferentialIsotopyCategory}
\mu_{\Iso(\bA)}^1(\eta)= \sum \mu_{\bA}\big\{ G_+, \dots, G_+, \eta, F_+, \dots, F_+ \big\} - (-1)^{|\eta|}\cdot  \eta \star \mu_{\bA},
\end{equation}
\noindent where the sum ranges over the numbers $i, j \geq 0$ of entries of the form $G_+$ and $F_+$. The higher products can be expressed by analogous formulas. 
For later reference we also want to record the following immediate observation.

\begin{lem}\label{LemmaIsotopyStrictlyUnitalNormalizedHochschildComplex}
	Suppose $\bA$ is a strictly unital $A_\infty$-category and $F$ an $A_\infty$-isotopy of $\bA$. Then, $F$ is strictly unital if and only if $F_+ \in \overline{C}(\bA)$ if and only if  $F \in \overline{C}(\bA)$. 
\end{lem}

\section{\texorpdfstring{Derived Picard groups and subgroups of $A_\infty$-functors}{Derived Picard groups and subgroups of A-infinity functors}}\label{SectionDerivedPicardGroups}

\noindent We remind ourselves of a few basic facts about derived Picard groups and discuss subgroups such as outer automorphism groups and their identity components. We introduce certain locally algebraic groups which conjecturally take the role of the identity component in the derived Picard group.

\subsection{The homotopy category of dg categories and derived Picard groups}\ \medskip

\noindent The category $\dgcat$ of small dg categories and dg functors admits a Dwyer-Kan-type model structure whose weak equivalences are the quasi-equivalences \cite{TabuadaDwykerKanModelStructure}. Its homotopy category will be denoted $\Hqe$. The derived tensor product of dg categories turns $\Hqe$ into a symmetric monoidal category which is \textit{closed} as shown by To\"{e}n \cite{ToenDerivedMoritaTheory}. Its internal hom objects and the morphisms in $\Hqe$ admit several equivalent descriptions which we recall below. Throughout the section, we fix an arbitrary field $\Bbbk$. We assume a certain familiarity of the reader with the theory of dg categories, such as tensor products, bimodules and their derived categories, e.g.~see \cite{KellerICMTalkDifferentialGradedCategories}.

\begin{definition}
	Let $\bA, \bB$ be quasi-essentially\footnote{A dg category is \textbf{quasi-essentially small} if its homotopy category is essentially small.} small dg categories. A bimodule $M: \bA \otimes \bB^{\op} \rightarrow  \cC(\Bbbk)$ is \textbf{right quasi-representable} if for all $A \in \bA$, the right dg $\bB$-module $M(-, A)$ is quasi-isomorphic to a representable module.
\end{definition}

\noindent Elements in the full subcategory $\qrep(\bA, \bB) \subseteq \cD(\bA \otimes \bB^{\op})$ spanned by the right quasi-representable bimodules are also called \textbf{quasi-functors}. 
\begin{thm}[{\cite{ToenDerivedMoritaTheory}}]
	\begin{enumerate}
		\item The internal hom object $[\bA, \bB] \in \Hqe$ is given by the dg category of $\h$-projective\footnote{A dg module is \textbf{$\h$-projective} if its internal hom in the category of dg modules preserves acyclicity of dg modules.} right quasi-representable  $\bA-\bB$-bimodules. 
		\item There is a bijection between $\Hom_{\Hqe}(\bA, \bB)$ and the set of isomorphism classes in $H^0([\bA, \bB])\simeq \qrep(\bA, \bB)$.
	\end{enumerate}
\end{thm}
\noindent A different description of $[\bA, \bB]$ was suggested by Kontsevich and subsequently proved by Canonaco-Ornaghi-Stellari (and in the special case of augmented $A_\infty$-categories by Faonte \cite{Faonte}, cf.~page $4$ in \cite{CanonacoOrnaghiStellari}).
\begin{thm}[{\cite{CanonacoOrnaghiStellari}}]\label{TheoremInternalHomThroughAInfinityFunctors}
\begin{enumerate}
	\item The dg category $[\bA, \bB]$ is quasi-equivalent to the dg category $\Fun(\bA, \bB)$ from \Cref{DefinitionFunctorCategory}.
	
	\item The inclusion of $\dgcat \subseteq \Acatc$ induces equivalences 
	
	\begin{displaymath}
	 \HAcat \simeq \Hqe \simeq \quotient{\Acatc}{\approx},
	\end{displaymath}
	\noindent where $\quotient{\Acatc}{\approx}$ denotes the quotient of $\Acatc$ by weak equivalences of functors and $\HAcat$ denotes the homotopy category of $\Acatc$, that is, its localisation at the class of quasi-equivalences. 
\end{enumerate}	 
\end{thm} 
\noindent In particular, $\Hom_{\Hqe}(\bA, \bB)$ is in bijection with the set of weak equivalence classes of $A_\infty$-functors $\bA \rightarrow \bB$. The perfect derived category of an $A_\infty$-category has a natural universal property in $\Hqe$.
\begin{prp}[{\cite[Theorem 7.2]{ToenDerivedMoritaTheory}}]\label{PropositionUniversalPropertiesHQE} The functor $\bA \mapsto \Perf \bA$ descends to a left adjoint to the inclusion of the full subcategory of $\Hqe$ consisting of all perfect $A_\infty$-categories.
\end{prp}
\noindent Although \cite[Theorem 7.2]{ToenDerivedMoritaTheory} only treats the case of dg categories, the general case simply follows from the fact that every $c$-unital $A_\infty$-category is canonically isomorphic in $\Hqe$ to a dg category by \Cref{PropositionYonedaRectification}. One defines the derived Picard group of an $A_\infty$-category as follows.
\begin{definition}\label{DefinitionDerivedPicardGroup}
	Let $\bA$ be a small cohomologically unital $A_\infty$-category. Its \textbf{derived Picard group} is the group $\DPic(\bA) \coloneqq \Aut^{\infty}\big(\Perf(\bA)\big)$.
\end{definition}
\noindent Equivalently, $\DPic(\bA)$ is the automorphism group of $\Perf(\bA)$ in $\Hqe$. If $\bA$ is a dg category, then the description via quasi-functors implies that \Cref{DefinitionDerivedPicardGroup} is equivalent to the more familiar description via invertible dg bimodules\footnote{Invertible dg bimodules of a perfect dg category are automatically right quasi-representable.} in \cite{KellerDerivedPicardGroup}. This was observed in \cite{ToenDerivedMoritaTheory}. By the universal property of the inclusion $\bA \rightarrow \Perf(\bA)$, every $F\in \Aut^{\infty}(\bA)$ extends uniquely up to weak equivalence to an $A_\infty$-endofunctor of $\Perf(\bA)$ which is seen to be a quasi-equivalence by the usual generation arguments in triangulated categories in combination with the five lemma. Hence there is a natural embedding of groups
\begin{displaymath}
\begin{tikzcd}
\Aut^{\infty}(\bA) \arrow[hookrightarrow]{r} & \DPic(\bA).
\end{tikzcd}
\end{displaymath}

\subsection{Outer automorphisms and identity components of derived Picard groups}\ \medskip

\noindent Throughout this section, we assume that $\cC$ is a graded $\Bbbk$-linear category. We denote by $\Out(\cC)$ its graded outer autoequivalence group, that is, the quotient of the group $\Aut(\cC)$ of graded autoequivalences of $\cC$ by natural isomorphisms. In case $\cC$ is a ungraded $\Bbbk$-algebra, this reduces to the outer automorphism group.
The group $\Aut(\cC)$ can be identified with the subgroup of strict quasi-equivalences of $\cC$. Moreover, it is not difficult to see 
 that functors $F, G \in \Aut(\cC)$ are weakly equivalent if and only if $F=G$ in $\Out(\cC)$. If the category algebra $\Bbbk[\cC]$ is finite-dimensional and $\Bbbk$ is algebraically closed, then, as in the ungraded case, $\Out(\cC)$ naturally admits the structure of an algebraic group whose identity component is a normal subgroup which we denote by $\OutO(\cC)$. In summary, one obtains the following series of embeddings
\begin{displaymath}
	\begin{tikzcd}
		\OutO(\cC) \arrow[hookrightarrow]{r} &  \Out(\cC) \arrow[hookrightarrow]{r} & \Aut^{\infty}(\cC) \arrow[hookrightarrow]{r} & \DPic(\cC).
	\end{tikzcd}
\end{displaymath}
For an ordinary algebra $A$ it was shown in \cite{Yekutieli} that $\DPic(A)$ is a locally algebraic group with identity component $\OutO(A)$.
\begin{definition}
	A \textbf{locally algebraic group} is a group $G$ together with a normal subgroup $N \subseteq G$ which is an algebraic group with the property that every coset of $N$ in $G$ is an algebraic variety and such that for each pair of cosets $U, V$ the inversion map $U \rightarrow U^{-1}$ and the multiplication map $U \times V \rightarrow UV$ are morphisms of varieties.
\end{definition}
\noindent The reader can find a definition of locally algebraic groups in terms of schemes in \cite{Yekutieli}. Unlike in the ungraded setting, the group $\OutO(\cC)$ is no longer a derived invariant (see \Cref{ExampleNotDerivedInvariant}) and hence is not the correct generalisation of the ``identity component'' of $\DPic(\cC)$. A more natural candidate based on $A_\infty$-isotopies is discussed in the next section.

\begin{exa}\label{ExampleNotDerivedInvariant}Let $K$ denote the Kronecker algebra. It is well-known (e.g.~\cite[Theorem 4.3]{MiyachiYekutieli}) that $\OutO(K)\cong \PGL_2(\Bbbk)$. On the other hand, we have an equivalence $\cD^b(\Bbbk K) \simeq \Dfd{B}$, where $B$ is the graded gentle algebra of the graded quiver 
	\begin{displaymath}
		\begin{tikzcd}
			\bullet \arrow[bend left]{r}{\alpha} & \arrow[bend left]{l}{\beta} \bullet
		\end{tikzcd}
	\end{displaymath}
	\noindent where $|\alpha|+1=|\beta|=1$ and the only relation $\beta \alpha=0$. One easily verifies that $\Aut(B)$ and hence $\OutO(B)$ are trivial.
\end{exa}

\subsection{Locally algebraic groups and identity components of the derived Picard groups of graded $\Bbbk$-linear categories}\label{SectionLocallyAlgebraicGroups}

\begin{definition}\label{DefinitionGroups} Let $\bA$ be an $A_\infty$-category and let $N \subseteq \Out(\HH^{\bullet}(\bA))$ (resp.~$N \subseteq \Aut(\bA)$) be a normal subgroup. We define $\Aut^{\infty}_N(\bA) \subseteq \Aut^{\infty}(\bA)$ (resp.~$\Aut^{\infty, h}_N(\bA) \subseteq \Aut^{\infty, h}(\bA)$) as the subset of weak equivalence (resp.~homotopy) classes of all $A_\infty$-functors $F$ such that $HH^{\bullet}(F) \in N$.
\end{definition}
\noindent  If If $N=\OutO(\bA)$, we will use the notation $\Aut^{\infty}_{\circ}(\bA)$ and $\Aut^{\infty, h}_{\circ}(\bA)$. If $\mathbf{1}$ denotes the trivial group and using the notation from \Cref{SectionAInfinityIsotopies}, we have $\Aut^{\infty, h}_+(\bA)=\Aut^{\infty, h}_{\mathbf{1}}(\bA)$ but only an inclusion $\Aut^{\infty}_+(\bA) \subseteq \Aut^{\infty}_{\mathbf{1}}(\bA)$ in general\footnote{Equality of the inclusion says that every $A_\infty$-functor whose induced functor on homotopy categories is isomorphic to the identity admits a weakly equivalent representative whose first Taylor coefficient \textit{is} the identity.}.

\begin{prp}
Let $\cC$ be a graded $\Bbbk$-linear category. Then $\Aut^{\infty}_+(\cC)= \Aut^{\infty}_{\mathbf{1}}(\cC)$.
\end{prp}
\begin{proof}
Given $F: \cC \rightarrow \cC$ and a natural isomorphism $\eta: F^1 \rightarrow \operatorname{Id}_{\cC}$ of $\Bbbk$-linear functors to the identity functor of $\cC=\HH^{\bullet}(\cC)$, $\eta$ determines, by formality of $\cC$, a strict natural transformation of $A_\infty$-functors between $F$ and an $A_\infty$-isotopy $G$ obtained from $F$ by base change. More precisely, $G^i: \cC(C_{i-1}, C_i) \otimes \cdots \otimes \cC(C_0, C_1) \rightarrow \cC(C_0,C_i)$ is the unique linear map which makes the canonical diagram
\begin{displaymath}
\begin{tikzcd}
\cC(C_{i-1}, C_i) \otimes \cdots \otimes \cC(C_0, C_1) \arrow{r}& \cC(C_0,C_i) \\
\cC(F^0(C_{i-1}), F^0(C_i)) \otimes \cdots \otimes \cC(F^0(C_0), F^0(C_1)) \arrow{r}  \arrow{u}{\sim}  & \cC(F^0(C_0), F^0(C_i))  \arrow{u}{\sim},
\end{tikzcd}
\end{displaymath}
induced by $\eta$ commute.
\end{proof}

\noindent If $\cC$ is a graded $\Bbbk$-linear category, we think of the group $\Aut_{\circ}^{\infty}(\cC)$ as a generalisation of the connected component of the identity functor in $\DPic(\cC)$ in the case of an ungraded algebra. There are several seemingly sensible ways to turn this mere analogy into a precise statement, e.g.~by viewing $\DPic(\cC)$ as the set of points on a ``differential graded group scheme'' $\DPicScheme(\cC)$ via relative Picard groups, cf.\cite{KellerDerivedPicardGroup}. Assuming that this is true in some sense, the following conjecture would be reasonable consequence.
\begin{conjecture}\label{Conjecture}
Let $\cC$ be a graded $\Bbbk$-linear category such that $\Hom^i_{\cC}(C, D)=0$ for all $C, D \in \Ob{\cC}$ and all $i  \geq 1$.  Then $\Aut^{\infty}_{\circ}(\cC)$ is a Morita invariant of $\cC$.
\end{conjecture}
\noindent Weaker assumptions under which we expect the above conjecture to be true are discussed in \Cref{RemarkWeakerConditionsConjecture} as they require the introduction of additional notation. A supporting argument for \Cref{Conjecture} is the following. In characteristic zero and assuming \Cref{IntroTheoremA} at this point, every $F \in \Aut^{\infty}_+(\bA)$ is connected to the identity  by a $1$-parameter subgroup $\mathbb{G}_a \subseteq \Aut^{\infty}_+(\bA)$ defined via $\lambda \mapsto \exp_{\bA}(\lambda h)$, where $h \in \HHH^1_+(\bA, \bA)$ is the unique preimage of $F$ under the exponential. \medskip

\noindent Structurally we will show that $\Aut^{\infty}_{\circ}(\bA)$ is a locally algebraic group under additional assumptions.  For this, we  need the following statement whose proof follows from \eqref{EquationNaturalTransformationSquare} and which is omitted.

\begin{lem}\label{LemmaFunctorInvertibilityHomotopyInvariant}Let $F, G: \bA \rightarrow \bB$ be $A_\infty$-functors and assume that $\bA$ is minimal. Then the following are true.
	\begin{enumerate}
		\item If $F \sim G$, then $F^1=G^1$.
		\item\label{ItemUpToConjugation} Suppose further that $\bB$ is minimal and strictly unital. If $F \approx G$, then $F^1$ and $G^1$ are naturally isomorphic as $\Bbbk$-linear functors.
	\end{enumerate}
\end{lem}

\noindent Using that every $A_\infty$-category admits a quasi-isomorphic minimal model, this directly implies the following.
\begin{cor}Let $\bA$ be a cohomologically unital $A_\infty$-category. Then, the assignment $F \mapsto F^1$ induces the vertical maps in the following commutative square of group homomorphisms:
	\begin{displaymath}
		\begin{tikzcd}
			\Aut^{\infty, h}(\bA) \arrow[twoheadrightarrow]{r} \arrow{d} & \Aut^{\infty}(\bA) \arrow{d}\\
			\Aut\big(\HH^{\bullet}(\bA)\big) \arrow[twoheadrightarrow]{r} & \Out\big(\HH^{\bullet}(\bA)\big).
		\end{tikzcd}
	\end{displaymath}
\end{cor}

\noindent Next, we  derive the following easy structural result about $\Aut^{\infty}_N(\bA)$.
\begin{lem}\label{LemmaSemiDirectProductInfinityAutmorphisms}Let $\bA$ be a cohomologically unital $A_\infty$-category and $N \subseteq \Out(\HH^{\bullet}(\bA))$ a normal subgroup. The set $\Aut^{\infty}_N(\bA)$ is the kernel of the group homomorphism 
	\begin{displaymath}
		\begin{tikzcd}[row sep=0.75em]
			\Aut^{\infty}(\bA) \arrow{r} & \Out(\bA)/N,\\
			F \arrow[mapsto]{r} & F^1.
		\end{tikzcd}
	\end{displaymath} Thus, it is a normal subgroup of $\Aut^{\infty}(\bA)$. If $\bA$ is formal, then there exists a natural inclusion $N \hookrightarrow \Aut_N^{\infty}(\bA)$ which is the section in the following split exact sequence of groups:
	\begin{equation}\label{EquationShortExactSequenceSpecialSubgroups}
		\begin{tikzcd}[row sep=0.75em]
			\mathbf{1} \arrow{r} & \Aut^{\infty}_{\mathbf{1}}(\bA) \arrow{r} & \Aut^{\infty}_N(\bA) \arrow{rr}{F \mapsto F^1} &&  N \arrow{r} & \mathbf{1}.
		\end{tikzcd}
	\end{equation}
	Thus, $\Aut_N^{\infty}(\bA) \cong \Aut^{\infty}_+(\bA) \rtimes N$ in this case.
\end{lem}
\begin{proof}
	By \Cref{LemmaFunctorInvertibilityHomotopyInvariant}, the projection $F \mapsto F^1$ is well-defined and a group homomorphism. It is clear that \eqref{EquationShortExactSequenceSpecialSubgroups} is exact. If $\bA$ is formal, we may assume that it is a $\Bbbk$-linear graded category, in which case every graded automorphism of $\bA$ determines  a strict $A_\infty$-functor and hence there is a canonical inclusion $N \hookrightarrow \Aut_N^{\infty}(\bA)$ which is a section.
\end{proof}
\noindent Of course the are analogous statements for the group $\Aut^{\infty, h}_N(\bA)$  which are proved in the same way. We finish this section by showing that $\Aut_N^{\infty}(\bA)$ is locally algebraic under certain assumptions.

\begin{prp}\label{PropositionLocallyAlgebraicGroup}Suppose that $\cC$ is a $\Bbbk$-linear graded category. Assume that $\Bbbk[\cC]$ is finite-dimensional and that $\Bbbk$ is algebraically closed. Let $N \subseteq \Out(\cC)$ be a closed normal subgroup. Then $\Aut^{\infty}_{N}(\cC)\cong \Aut_{+}^{\infty}(\cC) \rtimes N$ is locally algebraic.
\end{prp}
\begin{proof}As before we regard $N \subseteq \Aut^{\infty}(\cC)$ as a subgroup of homotopy classes of strict functors. Because $N$ is normal in $\Out(\cC)$, it follows from \Cref{LemmaSemiDirectProductInfinityAutmorphisms} 
	that $N$ is normal in $\Aut^{\infty}_{N}(\cC)$ with $\Aut^{\infty}_{N}(\cC) \cong N \rtimes \Aut^{\infty}_+(\cC)$. Consequently, the cosets of $N$ in are in bijection with elements of $\Aut^{\infty}_+(\cC)$: every coset $U$ of $N$ inside  $\Aut^{\infty}_{N}(\cC)$ contains a unique element $f_U \in \Aut^{\infty}_+(\cC)$ and right multiplication by $f_U$ yields a \textit{canonical} bijection $N \rightarrow U$. This allows us to endow $U$ with a well-defined structure of an algebraic variety via transfer. Since $f_U f_V \in \Aut_+^{\infty}(\cC) \cap UV$ it follows $f_{UV}=f_Uf_V$ by uniqueness. This shows that the multiplication maps and inversion maps on cosets agree with the ones  obtained from such maps of on $N$ via transfer along the bijections $N \rightarrow U$, $N \rightarrow V$ above. Thus, they are maps of algebraic varieties which shows that $\Aut^{\infty}_N(\cC)$ is a locally algebraic group.
\end{proof}

\noindent The reason why $\OutO(\cC)$ is not a Morita invariant in \Cref{ExampleNotDerivedInvariant} is that parts of it are transferred to the group $\HHH^1_+(\cC,\cC)$ which is also not a Morita invariant. Indeed, for $A=\Bbbk K$ and $B$ as in \Cref{ExampleNotDerivedInvariant}, one has $\HHH^1_+(A,A)=0 \cong \OutO(B)$ and $\OutO(A)\cong \PGL_2(\Bbbk) \cong \HHH^1_+(B,B)$. This is a very extreme case of ``transfer'' of  parts from $\OutO(-)$ to $\HHH^1_+(-)$ and vice versa which a typical behavior for graded gentle algebras. The abstract reason is that the group $\HHH^1_+(-,-)$ shrinks when passing from $A$ to its perfect hull, whereas the group $\Out(\HH^0(-))$ grows. Thus if, $A$ and $B$ are derived equivalent algebras, this combination of shrinking and growing moves parts between these two groups.

\section{Integration theory for complete pre-Lie algebras}\label{IntegrationTheoryCompletePreLieAlgebras}

\noindent  We recall some results from \cite{DotsenkoShadrinVallettePreLie} which allow us to integrate elements in pre-Lie algebras via generalized exponential maps. Throughout the entire section we will work over a field of characteristic $0$.

\subsection{The exponential map of a complete pre-Lie algebra and pro-unipotent groups.}\ \medskip

\noindent  Throughout this section, let $(V, \star, \mathbf{1})$ be a pre-Lie algebra  which we assume to be \textbf{complete} in the sense that there exists a collection of subvector spaces $\filt{V}{n} \subseteq V$, $n \geq -1$, such that 
\begin{displaymath}
	V= \prod_{n=-1}^{\infty}\filt{V}{n},
\end{displaymath}
\noindent and such that $\filt{V}{r} \star \filt{V}{s} \subseteq \filt{V}{r+s}$ for all $r,s, \geq -1$ and hence $\filt{V}{-1} \star \filt{V}{-1}=0$. If $V^{(0)}=0=V^{(-1)}$, we call it \textbf{proper}. For every complete pre Lie algebra $V$ as above, we denote by
\begin{displaymath}
	V_+ \coloneqq \prod_{n=1}^{\infty}\filt{V}{n}.
\end{displaymath}
\noindent its complete proper pre-Lie subalgebra. The vector space $V$ is complete with respect to the norm topology induced by the descending filtration $F_nV\coloneqq \prod_{i=n}^{\infty}\filt{V}{i}$, $n \geq -1$.  Moreover, for each $n \geq 1$, $F_nV$ forms an Lie ideal of $V$ and $V_+$. The quotients $V_+/F_nV$, $n \geq 1$ are nilpotent and
\begin{displaymath}
	V_+ = \lim_{n \geq 1}{\quotient{V_+}{F_nV}}.
\end{displaymath}
\noindent In other words, $V_+=F_1V \subseteq V$ is a pro-nilpotent Lie subalgebra.

\begin{exa}\label{ExampleHomotopyDerivationsCompletePreLieAlgebra}
	Let $\bA$ be an $A_\infty$-category. Then $C(\bA)$ is a complete pre-Lie algebra with left unit $\operatorname{Id}_{\bA} \in W_1C(\bA)$ and $C_+(\bA)\coloneqq \big(C(\bA)\big)_+=W_2C(\bA)$.
\end{exa} 

\noindent Next, we recall the definition of the exponential map for complete pre-Lie algebras.
\begin{definition}
	Let $V$ be a complete pre-Lie algebra and let  $v \in V_+$. The \textbf{exponential} of $v$ is the limit
	\begin{displaymath}
		\exp(v) \coloneqq \sum_{r=0}^{\infty} \frac{1}{r!}v^r,
	\end{displaymath}
	\noindent where $v^r$ denotes the $r$-fold right $\star$-product of $v$ with itself which is defined inductively as $v^{j+1}\coloneqq v^j \star v$, $v^1=v$ and, by convention, $v^{0}=\mathbf{1}$.
\end{definition}
\noindent Since $v^r \in F_rV$, the partial sum in the previous definition converges with respect to the topology in $V$. The exponential map is natural under morphisms $f: V \rightarrow W$ of pre-Lie algebras such that $f(V_+) \subseteq W_+$ and yields a bijection
\begin{displaymath}
	\begin{tikzcd}
		\exp: V_+^0 \arrow{r}{\sim} & \mathbf{1}+V_+^0,
	\end{tikzcd}
\end{displaymath}
\noindent where $\mathbf{1}+V_+^0\coloneqq\{\mathbf{1}+v\, | \, v \in V_+^0\}$ denotes the subset of \textbf{group-like elements}. The inverse $\mathbf{1} + w \mapsto \log(\mathbf{1} + w)$ of $\exp$, the pre-Lie logarithm, is defined via the pre-Magnus expansion \cite{Ebrahimi-FardManchon}. Its terms consists of $\star$-products of $w$ by itself and for $w \in V_+^0$, its first few terms are given by
\begin{displaymath}
	\log(\mathbf{1}+w) = w - \frac{1}{2}w \star w + \frac{1}{4} (w \star w) \star w + \frac{1}{12} w \star (w \star w) + \cdots.
\end{displaymath}

\begin{prp}[{\cite{DotsenkoShadrinVallettePreLie}}]\label{PropositionExponentialLogarithmBijection} Let $w \in V_+^0$. Then $\exp\left(\log(\mathbf{1}+w)\right)=\mathbf{1}+w$ and $\log(\exp(w))=w$.
\end{prp}

\subsection{The group of group-like elements of a complete pre-Lie algebra}\label{SectionGroupLikeElements}
\subsubsection{An associative product on group-like elements}

\begin{definition}A (symmetric or non-symmetric) brace algebra $V$ is \textbf{complete} if its underlying pre-Lie algebra is complete such that the associated filtrations are compatible with all brace operation in the sense that for each $r \geq 1$, its $r$-ary brace operation restricts to a map
	\begin{displaymath}
		F_{i_1}V \otimes \cdots \otimes F_{i_{r+1}}V \rightarrow  F_{i_1 + \cdots + i_{r+1}}V,
	\end{displaymath}
	\noindent for all $i_1, \dots, i_r \geq -1$.
\end{definition}
\noindent The Hochschild complex of a $\Bbbk$-graph is an example of a complete brace algebra. For any complete symmetric brace algebra $V$, one defines an associative operation $\odot: V \otimes \big(\mathbf{1}+V_+\big) \rightarrow V$ as
\begin{equation}\label{EquationDefinitionCompositionProduct}
	u \odot (\mathbf{1}+v) \coloneqq \sum_{n \geq 0} \frac{1}{n!} u\langle v, \dots, v \rangle_n,
\end{equation}

\noindent where $v\in V_+$. If $V$ is the symmetrisation of a complete brace algebra structure and if $|v|$ is even, then  by \eqref{EquationPowerSymmetrisation}, formula \eqref{EquationDefinitionCompositionProduct} can be alternatively expressed as
\begin{equation}\label{EquationAlternativeCompositionProduct}
	u \odot (\mathbf{1} + v) = \sum_{n \geq 0} u\{v, \dots, v\}_n.
\end{equation}
\noindent This definition makes sense for any brace algebra regardless of the characteristic of the underlying field. 
\begin{prp}[{\cite[Proposition 4]{DotsenkoShadrinVallettePreLie}}]
The set of group-like elements $\mathbf{1}+V_+^0$ becomes a group under the operation $\odot$ with neutral element $\mathbf{1}$.   
\end{prp}
\noindent  An explicit formula for inverses with respect to $\odot$ in the symmetric and non-symmetric case can be found in \cite[Proposition 4]{DotsenkoShadrinVallettePreLie} and \cite[Lemma 2.14]{VerstraeteDividedPowers}. Because the brace operations respect the filtration, the set of group-like elements is naturally a complete topological group when equipped with the topology induced by the descending filtration by the subgroups $G^i=\mathbf{1}+F_iV^0$, $i \geq 1$. Via
\begin{displaymath}
\mathbf{1}+V_+^0 \cong \lim_{n \geq 1}{\quotient{\mathbf{1}+V_+^0}{\mathbf{1}+F_nV^0}},
\end{displaymath}
\noindent  $\mathbf{1}+V_+^0$ is a pro-unipotent group.

\subsubsection{The Baker-Campbell-Hausdorff formula}\ \medskip

\noindent The exponential enjoys an analogue of the classical \textit{Baker-Campbell-Hausdorff} formula with respect to the $\odot$-product. There are many equivalent but different incarnations of this formula. The following is due to Dynkin.
\begin{prp}[{\cite[Theorem 2]{DotsenkoShadrinVallettePreLie}}]\label{PropositionBCHformula}
	Let $u,v \in V_+^0$. Then $\exp(u) \odot \exp(v)=\exp(w)$, where $w=\BCH{u}{v}$ is given by the \textbf{Baker-Campbell-Hausdorff (BCH) product}:
	\begin{equation}
		\BCH{u}{v} \coloneqq  u+v + \sum_{i=1}^{\infty} {\frac{{(-1)}^{i-1}}{i} \sum_{\substack{\underline{m},\underline{n} \in \mathbb{N}^i: \\ \underline{m}+\underline{n} \in \mathbb{N}_{>0}^i}}{ \frac{1}{|\underline{m}+\underline{n}|} \frac{1}{\underline{m}! \underline{n}!} \cdot [u^{m_1}, v^{n_1}, u^{m_2}, \dots, u^{m_i}, v^{n_i}]}.}
	\end{equation}
	\noindent Here, we used the following notation:
	\begin{itemize}
		\item For $\underline{p} \in \mathbb{N}^i$, $\underline{p}!=p_1!p_2! \cdots p_i!$ and $|\underline{p}|=\sum_{j=1}^i p_j$;
		\item $[A_1^{r_1}, A_2^{r_2}, \dots, A_i^{r_i}]=\ad_{A_1}^{r_1} \circ \ad_{A_2}^{r_2} \cdots \circ \ad_{A_{i-1}}^{a_{i-1}} \circ \ad_{A_i}^{r_i-1}(A_i)$.
	\end{itemize} 
\end{prp}
\noindent  In particular, one recovers the classical 
\begin{equation}\label{EquationSpecialCaseBCH}
	\exp(u)\odot \exp(v)=\exp(u+v)
\end{equation}
in case $[u,v]=0$. This holds automatically in the  special case $u=v$ from which one derives that $\exp(-u)$ is the inverse to $\exp(u)$ in $\mathbf{1}+V_+^0$. In summary, we have the following.
\begin{thm}{\cite[Lemma 1, Theorem 2]{DotsenkoShadrinVallettePreLie}}\label{TheoremBijectionExpLog}
	The maps $\exp$ and $\log$ define a pair of mutually inverse group homomorphisms between $V_+^0$ equipped with the BCH product and $\mathbf{1}+V_+^0$ equipped with the $\odot$-product.
\end{thm}

\noindent Useful for us will also be the following result about the interaction between the pre-Lie algebra structure, the composition product and the exponential map.
\begin{prp}{\cite[Proposition 5]{DotsenkoShadrinVallettePreLie}}\label{PropositionDSVExponentialAd}
	Let $v \in V_+$ and $w \in V$. Then,
	\begin{displaymath}
		\exp(\ad_{v})(w)= \left(\exp(v) \star w \right) \odot \exp(-v).
	\end{displaymath} 
\end{prp}

\subsection{Infinity-isotopies and integration}\label{SectionInfinityIsotopiesIntegration}\ \medskip

\noindent As before $(V, \star, \mathbf{1})$ shall denote a complete pre-Lie algebra over a field of characteristic $0$. We further fix a Maurer-Cartan element $\mu \in F_0V^1$ of the associated Lie algebra, that is, $\mu \star \mu=0$ with its induced differential $d_{\mu}: V \rightarrow V$ given by the adjoint map $d_\mu(a)=\mu \star a - (-1)^{|a|} a \star \mu$. We denote by $\cV$ both the pair $(V, \mu)$ as well as the the resulting complex $(V, d_{\mu})$. The filtration  $V \supseteq V_+ \supseteq 0$ induces a natural three-step filtration on the cohomology of $\cV$ whose middle term we denote by $\HH_+^{\bullet}(\cV)$. The Lie bracket on $V$ descends to a graded Lie algebra structure on $\HH_+^{\bullet}(\cV)$. This agrees with $\HHH_+^{\bullet}(\bA, \bA)=W_2\HHH^{\bullet}(\bA, \bA)$ if $V$ is the Hochschild complex of an $A_\infty$-category $\bA$.

\subsubsection{$\infty$-isotopies in pre-Lie algebras and homotopy}\ \medskip

\noindent  We will show that the exponential map descends to a map on $\HH_+^0(\cV)$ whose target are certain equivalence classes of group-like elements of $V$. We start with a few definitions.
\begin{definition}
	A group-like element $f \in \mathbf{1}+V_+^0$ is an \textbf{$\infty$-isotopy} of $\cV=(V, \mu)$ if
	\begin{equation}\label{EquationExponentialInfinityMorphism}
		\mu \odot f = f \star \mu.
	\end{equation}
	\noindent The \textbf{composition} of a pair $(f,g)$ of $\infty$-isotopies of $\cV$ is defined as $f \odot g$.  One declares $\infty$-isotopies $f, g$ of $\cV$ as \textbf{homotopic}, denoted by $f \sim_{\infty} g$, if there exists $h \in V^{-1}$ with $d_\mu(h) \in V_+$ such that
	\begin{displaymath}
		f\odot g^{-1}=\exp(d_\mu(h)).
	\end{displaymath}
\end{definition}
\noindent The definition of $\infty$-isotopies mimics the notion of $\infty$-isotopies between homotopy algebras structures on a fixed vector space over a Koszul operad, cf.~\cite[Section 5]{DotsenkoShadrinVallettePreLie}.

\begin{lem}\label{LemmaHomotopyEquivalenceRelation}
	Homotopy defines an equivalence relation on the group of $\infty$-isotopies. If $f \sim_\infty f'$ and $g \sim_\infty g'$ are homotopic $\infty$-isotopies, then $f \odot g \sim_{\infty} f' \odot g'$.
\end{lem}
\begin{proof}
	By assumption, there exists $\alpha \in \Img d_\mu$ such that $f \odot g^{-1}=\exp(\alpha)$. 
	Thus, all assertions are easily derived from the BCH formula and the observation that $[u,v] \in \Img d_\mu$ as soon as both $u$ and $v$ are $d_\mu$-cocycles and  $u \in \Img d_\mu$ or $v \in \Img d_\mu$. This in turn implies that $\nBCH(u,v) \in \Img d_\mu$ for cocycles $u, v$ whenever $u \in \Img d_\mu$ or $v \in \Img d_\mu$. For example, if $f \odot g^{-1}= \exp(\alpha)$ and $g \odot h^{-1}=\exp(\beta)$, then by associativity of $\odot$,
	\begin{displaymath}
		\begin{aligned}
			f \odot h^{-1}  = (f \odot g^{-1}) \odot (g \odot h^{-1}) =\exp(\alpha) \odot \exp(\beta)= \exp\big(\!\BCH{\alpha}{\beta}\big).
		\end{aligned}
	\end{displaymath}
	\noindent This proves transitivity. We leave the remaining details to the reader.
\end{proof}
\begin{definition}
	Let $V$ and $\cV$ be as before. We denote by $\IsoInfty(\cV)$ the group of homotopy classes of $\infty$-isotopies of $\cV$.
\end{definition}

\noindent In the concrete case of the Hochschild complex, one obtains the following relationship between $A_\infty$-isotopies and $\infty$-isotopies which was already observed in a more general setting in \cite{DotsenkoShadrinVallettePreLie}.
\begin{prp}\label{PropositionComparisonOfIsotopies}
 For any $A_\infty$-category $\bA$ the set of $\infty$-isotopies in $C(\bA)$ agrees with the set of $A_\infty$-isotopies of $\bA$ when regarded as elements of $C^0(\bA)$. Under this bijection, the composition of $A_\infty$-functors corresponds to the $\odot$-product.
\end{prp}
\begin{proof}
Comparison of definitions, cf.~\Cref{SectionHomotopyCategoryDGCategories}.
\end{proof}
\noindent We will also show that the two notions of homotopy, for $\infty$-isotopies and for $A_\infty$-functors, agree as well. The proof however is rather indirect and occupies the majority of \Cref{SectionIntegrationHochschildClasses}.

\subsubsection{Integration of cohomology classes to $\infty$-isotopies}\ \medskip

\noindent We show that $d_\mu$-cocycles are exactly the elements which integrate to $\infty$-isotopies. We use the notation from the previous sections.

\begin{lem}\label{LemmaKernelInfinityEndomorphism} Let $v \in V_+^0$. Then $\exp(v)$ is an $\infty$-isotopy of $\cV$ if and only if $v \in \ker d_{\mu}$.
\end{lem}
\begin{proof}
	We apply \Cref{PropositionDSVExponentialAd} for $w=\mu$. Since $d_\mu=\ad_\mu$ and if $v \in \ker d_\mu$ we have $\mu \in \ker \ad_v$ by graded symmetry of $[-,-]$ and hence $\exp(\ad_v)(\mu)=\operatorname{Id}_V(\mu)=\mu$. Thus, we obtain
	\begin{displaymath}
		\mu=\left(\exp(v) \star \mu\right) \odot \exp(-v).
	\end{displaymath}
	\noindent Multiplying with $\exp(v)$ via $\odot$ on both sides from the right yields $\mu \odot \exp(v)=\exp(v) \star \mu$. Here, we used that $\exp(-v)$ is the inverse of $\exp(v)$. Thus $\exp(v)$ is an $\infty$-isotopy if $v \in \ker d_{\mu}$. Vice versa, let $f \in \IsoInfty(V)$ and $v\coloneqq \log(f) \in V_+^0$. Then by \eqref{EquationExponentialInfinityMorphism} and \Cref{PropositionBCHformula},
	\begin{displaymath}
		\exp(\ad_v)(\mu)=(f \star \mu) \odot f^{-1}=(\mu \odot f) \odot f^{-1}=\mu \odot (f \odot f^{-1})=\mu,
	\end{displaymath}
	\noindent by associativity of $\odot$. Let 
	\begin{displaymath}
U\coloneqq \prod_{i \geq 0} (V^{(i)} \otimes_{\Bbbk} \mathbb{Z}[t]),
	\end{displaymath}
	\noindent where $\mathbb{Z} \subset \Bbbk$ via the canonical unital ring homomorphism. Then $U$ naturally a complete pre-Lie algebra with $\star$-product
	\begin{displaymath}
	(v_1 \otimes t^r) \star (v_2 \otimes t^s) \coloneqq (v_1 \star v_2) \otimes t^{r+s}.
	\end{displaymath}
	\noindent For every $\lambda \in \Bbbk$, evaluation at $t=\lambda$ defines a  morphism of complete pre-Lie algebras which we denote by $\operatorname{ev}_{\lambda}: U \rightarrow V$.	We can regard the expression $\exp(\ad_{v \otimes t})(\mu)-\mu$ as an element $P \in U$ and claim that $P=0$. Indeed, because $|v|=0$, we have $[v,v]=0$ and hence $\BCH{mv}{nv}=(m+n)v$ for all $m,n \in \mathbb{Z}$. Since the set of $\infty$-isotopies of $(V, \mu)$ is a group with respect to $\odot$, it follows from \Cref{PropositionBCHformula}, that $\operatorname{ev}_{\lambda}(P)=0$ for all $\lambda \in \mathbb{Z} \subset \Bbbk$. We may decompose $P$ into its weight components $P^{(n)} \in V^{(n)} \otimes \mathbb{Z}[t]$, $n \geq 0$. Then $P_0=0$ and for all $n \geq 1$,
	\begin{displaymath}
		P^{(n)}=\sum_{i=1}^{\infty} {\left(\frac{1}{i!} \ad_v^i(\mu)\right)}^{(n)} \otimes t^i \in V^{(n)} \otimes \mathbb{Z}[t],
	\end{displaymath}
	\noindent where by ${(-)}^{(n)}$ we denote the $V^{(n)}$-component of an element. It is important to note that the above sum is in fact finite because $v \in V_+$ and $\mu \in F_0V$. Since $\operatorname{ev}_{\lambda}(P^{(n)})=0$ for all $\lambda \in \mathbb{Z} \subset \Bbbk$ and because $\{1, t, t^2, \dots\} \subseteq \Bbbk[t]$ is linearly independent, \Cref{LemmaCalculationVectorPolynomials} below implies that ${\big(\ad_v^i(\mu)\big)}^{(n)}=0$ for all $n \geq 0$ and all $i \geq 1$. Since ${\ad_i(\mu)}^{(-1)}=0$ due to the fact that $v, \mu \in F_0V$, we conclude that $\ad_v^i(\mu)=0$ for all $i \geq 1$ and hence $d_\mu(v)=-\ad_v(\mu)=0$.
\end{proof}
\noindent The following lemma was used in the proof of \Cref{LemmaKernelInfinityEndomorphism}. 

\begin{lem}\label{LemmaCalculationVectorPolynomials}
Let $W$ be a vector space over a field $\Bbbk$ and let $w_1, \dots, w_r \in W$. Assume that 
 $P_1, \dots, P_r \in \Bbbk[t]$ are  polynomials and that there exists an infinite set $I \subseteq \Bbbk$ such that for all $\lambda \in I$,
\begin{displaymath}
\sum_{i=1}^r P_i(\lambda) \cdot w_i = 0.
\end{displaymath}
\noindent If $\{w_1, \dots, w_s\}$ is a maximal linearly independent subset of $\{w_1, \dots, w_r\}$, then $P_1, \dots, P_s$ lie in the $\Bbbk$-linear span of $P_{s+1}, \dots, P_r$. In particular, if $\{P_1, \dots, P_r\}$ is linearly independent, then all $w_i$ must be zero.
\end{lem}
\begin{proof}
Suppose that for $s+1 \leq m \leq r$, $w_m= \sum_{i=1}^{s} \alpha_i^m \cdot w_i$ is the unique linear combination. Then with $P_j'(t)\coloneqq P_j(t) + \sum_{i=s+1}^{r} \alpha_j^i P_i(t)$ for $1 \leq j \leq s$, we have 
\begin{displaymath}
\sum_{j=1}^{s} P_j'(\lambda) \cdot w_j=0,
\end{displaymath}
\noindent for all $\lambda \in I$. Because $w_1, \dots, w_s$ are linearly independent, it follows that $P_j'=0$ for all $1 \leq j \leq s$ because $I$ is infinite and every polynomial over a field has only finitely many roots. This proves the first claim. Next, suppose that at least one $w_i$ and w.l.o.g.~$w_1$ is non-zero. Consequently, if $\{P_1, \dots, P_r\}$ is linearly independent, then $\{w_1, \dots, w_r\}$ must be a maximal linearly independent subset by the first assertion. Thus, $P_i(\lambda)=0$ for all $\lambda \in I$ which again shows that $P_i=0$ for all $1 \leq i \leq r$ in contradiction to linear independence.
\end{proof}

\noindent When applied to the (normalized) Hochschild complex of an $A_\infty$-category, \Cref{LemmaKernelInfinityEndomorphism} yields the following. 
\begin{cor}\label{CorollaryIsotopyIsAInfinityIsotopy}
	Let $h \in C_+^0(\bA)$. Then $\exp(h)$ is an $A_\infty$-isotopy if and only if $h$ is a Hochschild cocycle. If $\bA$ is strictly unital and $h \in \overline{C}(\bA) \subseteq C(\bA)$, then $\exp(h)$ is strictly unital.
\end{cor}
\begin{proof}
	Follows from \Cref{PropositionIntegrationHochschildClassesAInfinityAlgebra}, \Cref{PropositionComparisonOfIsotopies} as well as \Cref{LemmaIsotopyStrictlyUnitalNormalizedHochschildComplex} and the fact that $\overline{C}(\bA)$ is closed under the brace operations.
\end{proof}

\begin{rem}
If $(\bA, \mu_{\bA})$ is an $A_\infty$-category, and $V=C(\bA)$ with its composition product $\star$, then \Cref{LemmaKernelInfinityEndomorphism} can be interpreted as the statement that the $0$-cocycles on $C(\bA)$, or equivalently, Hochschild $1$-cocycles, are in bijection with the endomorphism ring of $(\bA, \mu_{\bA})$ in the Deligne groupoid of $A_\infty$-structures on the underlying $\Bbbk$-graph $\bA$, cf.~\cite[Theorem 3]{DotsenkoShadrinVallettePreLie}.
\end{rem}
\noindent Next, we want to understand the relationship between cohomology and equivalence classes of $\infty$-isotopies.
\begin{prp}\label{PropositionIntegrationHochschildClassesAInfinityAlgebra}The exponential map descends to an isomorphism of groups
	\begin{displaymath}
		\begin{tikzcd}
			\exp:\HH_+^0(\cV) \arrow{r}{\sim} & \IsoInfty(\cV),
		\end{tikzcd}
	\end{displaymath}
	\noindent where the former is equipped with the BCH product and the latter with the $\odot$-product.
\end{prp}
\begin{proof}
	By \Cref{TheoremBijectionExpLog} and \Cref{LemmaKernelInfinityEndomorphism}, $\exp$ induces a bijection between $0$-cocycles in $V_+^0$ and $\infty$-isotopies. By definition of the Baker-Campbell-Hausdorff product and for every cocycle  $c \in V_+^0$, the map $b \mapsto \BCH{b}{c}$ (resp.~$b \mapsto \BCH{c}{b}$) induces a map from the set of coboundaries in $V_+^0$ to the set of cocycles $c' \in V_+^0$ which are cohomologous to $c$. The map is a bijection with inverse $c' \mapsto \BCH{c'}{-c}$. Since $\exp$ is a group homomorphism, it follows that $\exp(v_1) \sim_{\infty} \exp(v_2)$ if and only if $\BCH{v_1}{v_2}$ is a coboundary if and only if $[v_1]=[v_2]$ in $\HH^0(\cV)$. For every class $h \in \HH_+^0(\cV)$, there exists a  cocycle representative $v \in V_+^0$ and one defines $\exp(h)$ as the homotopy class of $\exp(v)$. By our previous observations this is independent of the choice of $v$.
\end{proof}

\section{Integration for brace algebras and Hochschild cohomology}\label{SectionIntegrationHochschildClasses}
\noindent We apply the results of \Cref{IntegrationTheoryCompletePreLieAlgebras} to the Hochschild complex of an $A_\infty$-category. Our main result is a proof of \Cref{IntroTheoremA} and \Cref{IntroTheoremNaturality} via the strategy outlined in the introduction. Let us briefly recall it here again. An application of the integration results for pre-Lie algebras from \Cref{IntegrationTheoryCompletePreLieAlgebras} to the Hochschild complex yields a preliminary exponential
\begin{displaymath}
	\begin{tikzcd}
e: \HHH^1_+(\bA,\bA) \arrow{r}{\sim} & \IsoInfty\!\big(C(\bA)\big),
\end{tikzcd} 
\end{displaymath}
\noindent whose image consists of $\sim_{\infty}$-equivalence classes of $A_\infty$-isotopies. A priori, the  relation $\sim_{\infty}$ bares no resemblance to the notions of homotopy and weak equivalence for $A_\infty$-functors and so our main obstacle lies in the comparison between the three. By relating $A_\infty$-isotopies to Maurer-Cartan elements of the Hochschild complex (with its canonical $A_\infty$-structure), we compare $\sim_{\infty}$ with the notion of Quillen homotopy equivalence for the associated Maurer-Cartan elements. This equivalence relation is then compared to other notions of equivalence for Maurer-Cartan elements, two of which are related to the notions of homotopy and weak equivalence of $A_\infty$-functors in much more direct way. Eventually, this proves \Cref{IntroTheoremA} for sufficiently nice dg categories. In a second step, we study naturality of the resulting exponential as a consequence of the invariance of the brace algebra structure of the Hochschild complex under quasi-equivalences. Altogether this leads to a proof of both theorems.

\subsection{Infinity-isotopies as Maurer-Cartan elements}\ \medskip

\noindent Let $V$ be a complete brace algebra with left unit $\mathbf{1}$ and $\mu \in V^1$ a Maurer-Cartan element. In particular, $V$ becomes a $A_\infty$-algebra via the operations \eqref{EquationAInfinityStructureHochscchildComplex}. The structure is complete in the following sense.
\begin{definition}
	A \textbf{filtered $A_\infty$-algebra} is an $A_\infty$-algebra $A$ together with a descending filtration
	\begin{displaymath}
	A=F_0A \supseteq F_1A \supseteq \cdots,
	\end{displaymath}
	by subvector spaces $F_iA \subseteq A$ such that for each $i \geq 1$ and all $m_1, \dots, m_i \geq 1$, $\mu_A^i$ restricts to a map
	\begin{displaymath}
	\begin{tikzcd}
	\mu_A^i: F_{m_1}A \otimes \cdots \otimes F_{m_i}A \arrow{r} & F_{m_1 + \cdots +  m_i}A.
	\end{tikzcd}
	\end{displaymath}
	\noindent It is moreover \textbf{proper} if in addition $A=F_0A=F_1A$ and $A$ is \textbf{complete} if it is complete with respect to the usual metric topology induced by the above filtration. A \textbf{filtered $A_\infty$-functor} is an $A_\infty$-functor $G: A \rightarrow B$ between filtered $A_\infty$-algebras such that for all $m_1, \dots, m_i \geq 0$,
	\begin{displaymath}
	G^i\big(F_{m_1}A \otimes \cdots \otimes F_{m_i}A\big) \subseteq F_{m_1 + \cdots + m_i}B.
	\end{displaymath}
\end{definition}
\noindent A special case of complete $A_\infty$-algebras are \textbf{nilpotent} $A_\infty$-algebras , that is, complete $A_\infty$-algebras $A$ such that $F_nA=0$ for some $n \geq 0$. A family of examples are algebras with nilpotent radical. Note that for any complete $A_\infty$-algebra $A$ and any $n \geq 0$, $F_nA$ inherits the structure of a complete $A_\infty$-algebra and the quotient $A/F_{n}A$ inherits the structure of a nilpotent $A_\infty$-algebra. 
\begin{definition}
	Let $A$ be a complete $A_\infty$-algebra. An element $\zeta \in F_1A^0$ is a \textbf{Maurer-Cartan element} if its curvature
	\begin{displaymath}
	\kappa(\zeta) \coloneqq \sum_{i \geq 1}\mu_C^i(\zeta, \dots, \zeta),
	\end{displaymath} 
	vanishes. We denote by $\MC(A)$ the set of Maurer-Cartan elements in $A$.
\end{definition}
\noindent The infinite sum in the above definition converges due to completeness and the assumption on $\zeta$. Of course, the same definition makes sense for all dg algebras without the completeness assumption in which case, and after desuspension, $\zeta$ becomes an element of degree $1$ for which the curvature takes the familiar form 
\begin{displaymath}
\kappa(\zeta)=d_A(\zeta)+\zeta^2.
\end{displaymath}
\noindent The Maurer-Cartan set extends to a functor $\MC(-)$ from the category of complete $A_\infty$-algebras with filtered $A_\infty$-functors to the category of sets: every filtered $A_\infty$-functor $F: A \rightarrow B$ induces a map 
	\begin{displaymath}
	\begin{tikzcd}[row sep= 0.5em]
	\MC(A) \arrow{r} & \MC(B), \\
	\zeta \arrow[mapsto]{r} & \sum_{n \geq 1} F^n(\zeta^{\otimes n}).
	\end{tikzcd}
	\end{displaymath}

\begin{prp}\label{PropositionBijectionMaurerCartanIsotopies} Let $V$ be a complete brace algebra, $\mu \in F_0V^1$ a Maurer-Cartan element and let $\cV$ denote the associated complete $A_\infty$-algebra. The assignment
	\begin{displaymath}
	\begin{tikzcd}
	f \arrow[mapsto]{r} &  f_+\coloneqq f-\mathbf{1},
	\end{tikzcd}
	\end{displaymath}
	\noindent restricts to a bijection between the set of $A_\infty$-isotopies of $\cV$ and $\MC(F_2\cV)=\MC(V_+)$. 
\end{prp}
\begin{proof}Since $\mu_{\cV}^1(v)= \mu \star v - (-1)^{|v|} v \star \mu$ and $\mathbf{1} \star \mu=\mu$, we find that
	\begin{displaymath}
	\begin{aligned}
	\kappa(f_+) & = \sum_{i \geq 2} \mu_{\cV}^i(f_+, \dots, f_+) + \mu_{\cV}^1(f_+)  \\
	&  = \sum_{i \geq 2} \mu \{f_+, \dots, f_+\} +  \big(\mu \star f_+ - f_+ \star \mu\big) \\
	& = \big(\mu \odot f - \mu\big)  - f_+ \star \mu \\ & = \mu \odot f - f \star \mu.
	\end{aligned}
	\end{displaymath}   
\end{proof}
\noindent Another way to formulate \Cref{LemmaKernelInfinityEndomorphism} is thus the following.
\begin{cor}\label{CorollaryExponentialBijectionCocyclesMCElements}
Let $V$ be a complete brace algebra and $\mu \in F_0V^1$ a Maurer-Cartan element. Then the exponential induces a bijection
\begin{displaymath}
\begin{tikzcd}
\exp: Z^0V_+ \arrow{r}{\sim} & \MC(V_+),
\end{tikzcd}
\end{displaymath}
\noindent where $Z^0V_+$ denotes the set of $0$-cocycles in $V_+$ with respect to $[\mu, -]$.
\end{cor}
\noindent The fact that the exponential associates an $A_\infty$-automorphism to a Hochschild $1$-cocycle was already proved in \cite[Lemma 7.1, Lemma 7.2]{ChoLee} (and appeared implicitly in earlier work by Kontsevich-Soibelman). The map in \Cref{CorollaryExponentialBijectionCocyclesMCElements}  becomes an isomorphism of groups, where $Z^0V_+$ is equipped with the Baker-Campbell-Hausdorff product and $\MC(V_+)$ is equipped with the multiplication arising from the $\odot$-product of the associated $\infty$-isotopies.

\subsection{Twisted deformations along Maurer-Cartan elements}\ \medskip

\noindent  A Maurer-Cartan element $\zeta \in F_1A^0$ in a complete $A_\infty$-algebra or dg algebra $A$  (in the shifted sense) gives rise to deformations of $A$ as a right $A$-module and as an $A_\infty$-algebra.
\begin{definition}
	Let $A$ be a complete $A_\infty$-algebra or dg algebra and let $\zeta \in F_1A^0$ be a Maurer-Cartan element. Define $A^{[\zeta]}$ as the right $A$-module with underlying $\Bbbk$-module $A$ and structure maps
	\begin{displaymath}
		\mu_{A^{[\zeta]}}^i \coloneqq \sum_{j \geq 0} \mu_A^{i+j}\big(\zeta \otimes \cdots \otimes \zeta \otimes \operatorname{Id}_{A^{\otimes i}}\big).
	\end{displaymath}
	Similarly, define $A^{\zeta}$ as the $A_\infty$-algebra modeled on a graded $\Bbbk$-module $C$ with structure maps
	\begin{displaymath}
		\mu^i_{A^{\zeta}} \coloneqq \sum_{j \geq 0} \mu_A^{i+j}\bigg(\!\sh\!\Big(\zeta^{\otimes j}, \operatorname{Id}_{A^{\otimes i}}\Big)\!\bigg),
	\end{displaymath}
	\noindent where $\sh: T(A[1]) \otimes T(A[1])\rightarrow T(A[1])$ denotes the shuffle product. It is defined as
	\begin{displaymath}
		\sh(a_1 \otimes \cdots \otimes a_r, a_{r+1} \otimes \cdots a_{r+s}) \coloneqq \sum_{\sigma} \varepsilon(\sigma) \cdot a_{\sigma(1)} \otimes \cdots \otimes a_{\sigma(r+s)},
	\end{displaymath}
\noindent where $\varepsilon(-)$ denotes the Koszul sign of a permutation and $\sigma$ runs through the set of $(r,s)$-shuffles, that is, permutations $\tau$ of $\{1, \dots, r+s\}$ such that 
\begin{displaymath}
\begin{array}{ccc}
	\tau(1) < \tau(2) < \dots < \tau(r) & \text{ and } & \tau(r+1) < \tau(r+2) < \dots \tau(r+s).
\end{array}
\end{displaymath}
\end{definition}

\noindent We note that $A^{\zeta}=A$ as graded algebras if $A$ is an unshifted dg algebra and $d_{A^{\zeta}}=d_A+ [\zeta, - ]$, where $[-,-]$ denotes the commutator of the multiplication. 
\subsection{Equivalence relations on sets of Maurer-Cartan elements}\ \medskip

\noindent The literature on deformation theory contains a number of equivalence relations for Maurer-Cartan elements for dg and $A_\infty$-algebras. Below we concern ourselves with four particular ones. Here are the first two.
\begin{definition}\label{DefinitionDeformations}
Let $A$ be a complete $A_\infty$-algebra or a dg algebra. Maurer-Cartan elements $\zeta, \xi \in F_1A^0$ are \textbf{dg gauge equivalent} (resp.~\textbf{homotopy gauge equivalent}) if the right $A$-modules $A^{[\zeta]}$ and $A^{[\xi]}$ are dg isomorphic (resp.~homotopy equivalent). They are moreover \textbf{strictly homotopy gauge equivalent} if there exists a homotopy equivalence between $A^{[\zeta]}$ and $A^{[\xi]}$ which is strict (in the sense of natural transformations of $A_\infty$-functors).
\end{definition}
\noindent The first two notions define equivalence relations on the set of Maurer-Cartan elements, the last does so as well if $A$ is a dg algebra. In the dg case, strict homotopy gauge equivalence coincides with ``homotopy gauge equivalence'' in the sense of Chuang-Holstein-Lazarev \cite{ChuangHolsteinLazarev} and what we call homotopy gauge equivalence is a straightforward generalisation of their notion which allows non-strict natural transformations between $A$-modules.

 If $U$ is a (complete) unshifted dg algebra, then $U^{\zeta}$ is isomorphic to the endomorphism ring of $U^{[\zeta]}$ and two Maurer-Cartan elements $\zeta, \xi \in F_1U^1$ are strictly homotopy gauge equivalent if and only if there exist $c, d \in U^{0}$ as well as $u,v \in U^{-1}$ such that
\begin{equation}\label{EquationExplicitFormulaHomotopyGaugeEquivalence}
\arraycolsep=2em
\begin{array}{lr}
{	\begin{aligned}
	0  & = d_U(c) + \xi \cdot c - c \cdot \zeta, \\[1em]
	0 & =  d_U(d) + \zeta \cdot d - d \cdot \xi,
	\end{aligned} } & 
{\begin{aligned}
	1 + d_U(u) + [\zeta, u] & = d \cdot c,\\[1em]
1 + d_U(v) + [\zeta, v] & = c \cdot d,
	\end{aligned}}
\end{array}
\end{equation}  
\noindent where $[a,b]=a \cdot b - (-1)^{|a||b|} b \cdot a$. Dg gauge equivalence corresponds to the case when $u=0=v$, in which case $c=d^{-1}$ and the above equations reduce to the condition
\begin{displaymath}
\zeta= c \xi c^{-1} - d_U(c) c^{-1}.
\end{displaymath}
\noindent If moreover $U$ is complete, there is another classical definition of gauge equivalence via exponentials of the adjoint maps.
\begin{definition}Maurer-Cartan elements $\zeta, \xi \in F_1U^1$ are \textbf{gauge equivalent} if there exists $u \in F_1U^0$ such that
\begin{equation}\label{EquationGaugeActionDG}
\zeta=e^u.\xi \coloneqq \xi + \sum_{n=1}^{\infty}\frac{1}{n!}\ad_u^{n-1}\big(\ad_u(\xi)-d_U(u)\big),
\end{equation}
\end{definition}
\noindent where $\ad_a(b)=[a,b]$. One can show that the condition $e^u.\xi=\zeta$ is equivalent to the condition that $\zeta$ and $\xi$ are dg gauge equivalent via the invertible element
 \begin{equation}\label{EquationGaugeElement}
c=\sum_{n \geq 0}\frac{1}{n!} u^{n},
\end{equation}
\noindent where $u^n$ now denotes the $n$-fold power of $u$ with respect to the multiplication in $U$. In other words, gauge equivalences correspond to dg isomorphisms which are induced by group-like elements of $U$ which we consider as a complete pre-Lie algebra with $\star$-product given by multiplication. The reader can find a proof of this result in \cite[Proposition 1.1]{BuijsFelixManrilloTanreQuillenFunctor} by using the universal property of the universal enveloping algebra of a dg Lie algebra.  For another discussion of this, see also the paragraph after \cite[Remark 4.1.10]{GuanPhDThesis}.

\subsubsection{Homotopy equivalence}\label{SectionHomotopyEquivalene}\ \medskip

\noindent The last notion we consider in this paper is the one of \textit{Quillen homotopy} which is a generalisation of gauge equivalence. Let $\Omega_1$ denote the commutative dg algebra given by the symmetric algebra in generators $t$ and $dt$ with $|t|=0$ and $|dt|=1$ and differential $d(t)=dt$. In particular, $(dt)^2=0$. It is isomorphic to  the dg algebra of polynomial differential forms on the affine $1$-simplex given by the quotient
\begin{displaymath}
\frac{\Bbbk[t_1, t_2, dt_1, dt_2]}{(t_1+t_2=1, dt_1 + dt_2 =0)} \cong \Omega_1,
\end{displaymath}
\noindent cf.~\cite[Section 7.3]{ManettiBook}. For any $A_\infty$-algebra $A$ and any dg algebra $U$, the tensor product $A \otimes U$ is naturally an $A_\infty$-algebra with structure maps
\begin{displaymath}
\begin{aligned}
\mu_{A \otimes U}^1 & \coloneqq \mu_A^1 \otimes \operatorname{Id}_U +  \operatorname{Id}_A \otimes \mu_U^1, \\
\mu_{A \otimes U}^i & \coloneqq \mu_A^i \otimes \mu_U^{(i)}, \, i \geq 2,
\end{aligned}
\end{displaymath}
\noindent where $\mu_U^{(i)}(a_i, \dots, a_1)= \mu_U^2\big(a_i,\mu_U^2(a_{i-1},\mu_U^2(\dots, a_1)))$. If $A$ is nilpotent, then $A \otimes U$ is nilpotent too with respect to the filtration $F_n(A \otimes U)=F_nA \otimes U$. For $U=\Omega_1$, the tensor product $A \otimes U$ comes with restriction maps $\operatorname{ev}_0, \operatorname{ev}_1: A \otimes \Omega_1 \rightarrow A$ defined by evaluating $t$ at $0$ and $1$ respectively and sending $dt$ to zero.
 
\begin{definition}
Let $A$ be a complete $A_\infty$-algebra. Maurer-Cartan elements $\zeta, \xi \in F_1A^0$ are \textbf{homotopic} (or \textbf{Quillen homotopic}) if there exists a Maurer-Cartan element $h$ of $\varprojlim_n \big( A/F_nA \otimes \Omega_1\big)$ such that $\operatorname{ev}_0(h)=\zeta$ and $\operatorname{ev}_1(h)=\xi$. We denote by $\overline{\MC}(A)$ the set of homotopy classes of Maurer-Cartan elements.
\end{definition}
\noindent After unraveling the definitions, cf.~\cite[Section 6.2.3]{RobertNicoudPhDThesis}, one sees that $\zeta, \xi$ are homotopic if and only if there exist power series $x(t)  \in F_1A[\![t]\!]$ with coefficients in $F_1A$ and $\lambda(t) \in A[\![t]\!]$ such that $x(0)=\zeta$, $x(1)=\xi$ and $x(t)$ is the unique solution\footnote{Uniqueness and existence follows for example from the results in \cite[Appendix A.2]{Robert-NicoudValletteHigherLieTheory}.} to the (formal) differential equation
\begin{equation}\label{EquationGaugeEquivalence}
\frac{\partial}{\partial t}x(t) = \mu_{\bA}^1\big(\lambda(t)\big) + \sum_{i\geq 1}\mu_{\bA}^i\big(\!\sh\!\big({x(t)}^{\otimes (i-1)},\lambda(t)\big)\big).
\end{equation}
\noindent The equation guarantees that $x(s)$ is a Maurer-Cartan element for all $s \in \Bbbk$. Homotopy defines an equivalence relation which agrees with the notion of homotopy for the associated $L_\infty$-algebra obtained by anti-symmetrisation. Gauge equivalence corresponds to the case when $\lambda$ is constant in which case $x(t)=e^{\lambda t}.\xi$ from \eqref{EquationGaugeActionDG}. This is a special case of an expression for the solution $x(t)$ as a tree-indexed sum, cf.~\cite[Appendix C]{Robert-NicoudValletteHigherLieTheory}. The definition of gauge equivalence extends as follows to the complete case.
\begin{definition}
 Let $A$ be a complete $A_\infty$-algebra. Maurer-Cartan elements $\zeta, \xi \in F_1A^0$ are \textbf{gauge equivalent} if there exists a homotopy $(x,\lambda)$  with constant $\lambda$ such that $x(0)=\zeta$ and $x(1)=\xi$.
\end{definition}
\noindent The same proof as in \cite[Proposition 9]{DotsenkoPoncinThreeHomotopies} via the Dupont contraction shows that every homotopy can be ``rectified'' to a gauge equivalence showing that the two notions are actually equivalent, see also \cite[Theorem 10.1.1, Proposition 10.2.3]{RobertNicoudPhDThesis} as well as the discussion at the end of Section 10.2.1 in loc.cit. One can summarize the situation as follows.

$$
\text{homotopic} \Longleftrightarrow \text{gauge equiv.} \xRightarrow{\text{dg algebra}} \text{dg gauge equiv.} \Longrightarrow  \text{strictly homotopy gauge equiv.}
$$

\noindent The following lemma relates the latter notions to equivalence relations of $A_\infty$-functors.
\begin{lem}\label{LemmaIsotopiesVSMaurerCartanElements} Let $\bU$ be dg category and let $F, G \in \Iso(\bU)$. Then, the following are true.
	\begin{enumerate}
		\item  Then $F \sim G$ if and only if $F_+$ and $G_+$ are dg gauge equivalent.
		\item  $F \approx G$ if and only if $F_+$ and $G_+$ are strictly homotopy gauge equivalent.
	\end{enumerate}	
\end{lem} 

\begin{proof}The Hochschild complex $C=C(\bU)$ is a dg algebra. Now the assertion follows by direct comparison of \eqref{EquationExplicitFormulaHomotopyGaugeEquivalence} with \eqref{EquationDifferentialIsotopyCategory} and the formulas for the $A_\infty$-structure on $C$ via the brace operations, keeping in mind that $\mu^i$ and $m^i$ are related by a sign \eqref{EquationSignRelationDesuspendedOperations}.
\end{proof}

\noindent As a first step, we want to compare gauge equivalence and dg gauge equivalence. We say that a dg category if \textbf{semi-free} if its underlying graded $\Bbbk$-linear category is a retract of a tensor category $T^a(\cG)$ for some $\Bbbk$-graph $\cG$. We remark that all cofibrant\footnote{Cofibrant dg categories are exactly the retracts of cellular dg categories which are sometimes also called ``semi-free''. Semi-free dg categories in this sense are semi-free in our sense but not vice versa.}  dg categories in the Dwyer-Kan model structure  are semi-free in the above sense, e.g.~see \cite[Proposition 2.40]{KarabasLee}.
\begin{lem}\label{LemmaDGGaugeEquivalenceGaugleEquivalence}
	Suppose $\bA$ is a semi-free dg category. Then $F, G \in \Iso(\bA)$ are homotopic if and only if $F_+$ and $G_+$ are gauge equivalent.
\end{lem}
\begin{proof}
As discussed above, gauge equivalence implies dg gauge equivalence and by \Cref{LemmaIsotopiesVSMaurerCartanElements} it is sufficient to prove the converse. Write $C=C(\bA)$ and suppose $\bA$ is a retract of $T^a(\cG)$ for some $\Bbbk$-graph $\cG$. By assumption, the underlying \textit{ungraded} algebra of $C$ is $\mathbb{N}^2$-graded, namely, in the first component by the Hochschild weight grading and in the second component by the weight grading inherited from tensor category construction. In other words, an elementary function with $m$ inputs is homogeneous of bidegree $(m,n)$ if its image lies in the weight $n$ component of $\bA \subseteq T^a(\cG)$. We denote by $C_{(i,j)}$ the homogenenous component of bidegree $(i,j) \in \mathbb{N}^2$ and write $C_{(+, \ast)}\coloneqq\bigoplus_{i \geq 1}\bigoplus_{j \geq 0}C_{(i,j)}$ as well as $C_{(0, +)}\coloneqq\bigoplus_{j \geq 1}C_{(0,j)}$. We note that the unit of $C$ lies in $C_{(0,0)}=C_{(0,0)}^0$ which is in bijection with tuples $(\lambda_X \mathbf{1}_X)_{X \in \Ob{\cG}}$, where $\lambda_X \in \Bbbk$ and  $\mathbf{1}_X$ denotes the identity morphism associated to $X$. In particular, $C_{(0,0)}$ lies in the center of $C$. Suppose that $c \in C^0$ is invertible with inverse $d$. We decompose $c$ as a $c=c_{(0,0)} + c_{(0,+)} + c_{(+, \ast)}$, where $c_{(u,v)} \in C_{(u,v)}$ for $u, v \in \{0,+, \ast\}$. If $d$ and $cd$ are decomposed in the same way, we see that that $c_{(0,0)}+c_{(0, +)}$ is necessarily the inverse to $d_{(0,0)}+d_{(0,+)}$. But because $\bA$ lies in a tensor category, it follows that $c_{(0,+)}=0=d_{(0,+)}$. This is analogous to the proof that all invertible elements in a tensor algebra are scalars. It follows that $cd=(cd)_{(0,0)}=c_{(0,0)}d_{(0,0)}$ and hence $c_{(0,0)}$ and $d_{(0,0)}$ correspond to tuples $(c_X\cdot \mathbf{1}_X)_{X \in \Ob{\cG}}$ and $(d_X \cdot\mathbf{1}_X)_{X \in \Ob{\cG}}$ with $c_X=d_X^{-1} \in \Bbbk^{\times}$ for all $X \in \Ob{\cG}$. Now, if $c$ and $d$ define a dg gauge equivalence between Maurer-Cartan elements $\zeta$ and $\xi$, that is,
\begin{displaymath}
	\zeta= c\xi c^{-1} - d_C(c) c^{-1},
\end{displaymath}  
\noindent then because $c_{(0,0)}$ and $d_{(0,0)}$ lie in the center of $C$, \eqref{EquationExplicitFormulaHomotopyGaugeEquivalence} implies that $d_{C}(c_{(0,0)})=0=d_C(d_{(0,0)})$. Hence by the Jacobi rule, $c'\coloneqq c \cdot d_{(0,0)}$ also satisfies
\begin{displaymath}
	\zeta= c' \xi c^{\prime -1} - d_C(c') c^{\prime -1}.
\end{displaymath}
\noindent By construction, $c'$ is a group-like element of $C$ which finishes the proof.
\end{proof}

\subsection{Homotopy of Maurer-Cartan elements versus homotopy of $\infty$-isotopies}\label{SectionGaugeTrivialElements}\ \medskip

\noindent We show that homotopy of Maurer-Cartan elements agrees with homotopy of $\infty$-isotopies. 
\begin{thm}\label{TheoremWeakEquivalenceGroupoids}	
	Let $(V, \mathbf{1})$ be a complete brace algebra over a field of characteristic $0$ and let $\mu \in F_0V^1$ be a Maurer-Cartan element. The exponential descends to a group isomorphism
	\begin{displaymath}
	\begin{tikzcd}
	\big(\!\HH_+^0(V), \BCH{-}{-}\big) \arrow{rr}{\exp}[swap]{\sim} && \overline{\MC}(V_+, \mu_V). 
	\end{tikzcd}
	\end{displaymath}
\end{thm}
\begin{proof}
	Let $d\coloneqq d_{\mu}$ denote the differential of the dg Lie algebra $(V, [-,-])$ and $\mu\coloneqq \mu_V$ the structure map of the $A_\infty$-structure on $V$ (and $V_+$) defined by $\mu$ and the brace algebra structure. Let $g(t) \in   t V^{-1}[\![t]\!]$ such that $dg(t) \in V^0_+[\![t]\!]$. That is, $g$ is a power series with trivial constant term and coefficients in $V^{-1}$ whose differentials lie in $V^0_+$. Define power series
	\begin{displaymath}
	\arraycolsep=1.5em
	\begin{array}{ccc}
	f(t) \coloneqq d\big(g(t)\big), & x(t) \coloneqq \exp\big(f(t)\big)  - \mathbf{1}, & \lambda(t) \coloneqq  \exp(f(t)) \star g(t)=(\mathbf{1} + x(t)) \star g(t).
	\end{array}
	\end{displaymath}
	
	\noindent Note that by choosing $g(t)=tu$ for any $u \in V^{-1}$ such that $d(u) \in V_+^0$, we can force any coboundary as the value of $f$ at $1$.  Our first claim is that the pair $(x, \lambda)$ defines a homotopy from $x(0)=0$ to $x(1)=\exp(d(g(1))) - \mathbf{1}$.  This follows from a series of manipulations of power series. First, we note that
	\begin{displaymath}
	\begin{aligned}
	\frac{\partial}{\partial t} x  & = (\mathbf{1}+x) \star f \\  & = f +  x \star d(g) \\ 
	& = f +  d(x \star g) - d(x) \star g + \mu^2\big(\sh(x, g)\big) \\ 
	& = d(\lambda) - d(x) \star g + \mu^2\big(\sh(x, g)\big)  \\ & = d(\lambda) + \big(\sum_{i \geq 2}\mu^i(x, \dots, x)\big) \star g + \mu^2\big(\sh(x, g)\big) .
	\end{aligned}
	\end{displaymath}
	\noindent The first equality follows directly from the definition of $\exp(-)$, the second by definition of $f$ and the third is the consequence of the general formula
	\begin{equation}\label{EquationDifferentialStarProduct}
	a \star d(b) = (-1)^{|a|}\big(d(a \star b) - d(a) \star b +  \mu^2(\sh(a,b))\big).
	\end{equation}
	
	\noindent which ones deduces from the brace relations, cf.~\cite[Theorem 3]{GerstenhaberCohomologyStructure}. The fourth equality follows from the definition of $\lambda$ and the last from the fact that $x(t)$ is a Maurer-Cartan element. Via the axiom \eqref{EquationIdentifiesBraceAlgebra} of brace algebras and the fact that $|x|=0=|t|$, we further obtain
	\begin{displaymath}
	\begin{aligned}
	\mu^n\big(x, \dots, x\big) \star g & = \sum_{i=0}^{n-1}\mu^n(\dots, x, x \star g, x, \dots) +  \sum_{j=0}^{n}\mu^{n+1}(\dots, x, g, x, \dots) \\
	& 	= \mu^n\big(\!\sh(x^{\otimes (n-1)}, x \star g)\big) +  \mu^{n+1}\big(\!\sh(x^{\otimes n}, g)\big),
	\end{aligned}
	\end{displaymath}
	\noindent where $i$ and $j$ denote the number of copies of $x$ in the argument before $x \star g$ and $g$ respectively. Thus, using $\lambda=(\mathbf{1}+ x) \star g$,
	\begin{equation}\label{EquationCalculationStarProduct}
	\begin{aligned}
	\big(\sum_{i \geq 2}\mu^i(x, \dots, x) \big) \star g & = \mu^2\big(\!\sh(x,x \star g)\big) + \sum_{i \geq 3}\mu^i\big(\sh(x^{\otimes (i-1)}, \lambda)\big) \\
	& =  \sum_{i \geq 2}\mu^i\big(\sh(x^{\otimes (i-1)}, \lambda)\big) - \mu^2\big(\!\sh(x, g)\big).
	\end{aligned}
	\end{equation}
	\noindent In combination with the above, this shows that $(x, \lambda)$ is a homotopy. Next, for any homotopy pair $(x, \lambda)$ with $x(t) \in V_+^0[\![t]\!]$, set $f(t) \coloneqq \log(\mathbf{1} +x(t)) \in V_+^0[\![t]\!]$. We claim that there exists a unique power series $g \in V^{-1}[\![t]\!]$ such that $\lambda(t)=\exp(f(t)) \star g(t)$. Indeed, since $\exp(f(t))=\mathbf{1} + x(t)$ and if $\lambda_n(t), x_n(t), g_n(t) \in V^{(n)}$ denote the respective components of $\lambda(t)$, $x(t)$ and $g(t)$ in $V^{(n)}$, we necessarily have
	\begin{displaymath}
	g_n(t)= \lambda_n(t) - \sum_{j=2}^{\infty} x_{j}(t) \star g_{n-j}(t)=\lambda_n(t) - \sum_{j=2}^{n+1} x_{j}(t) \star g_{n-j}(t),  
	\end{displaymath}  
	\noindent for all $n \geq 0$. The same equations also give a recursive recipe for the construction of $g(t)$ which proves its existence and uniqueness.  Now, the fact that $(x, \lambda)$ is a homotopy implies the following equalities.
	\begin{displaymath}
	\begin{aligned}
	\big(\mathbf{1} + x \big) \star f = \frac{\partial}{\partial t}x &  = d\lambda + \sum_{i \geq 2}\mu^i\big(\!\sh(x^{\otimes (i-1)}, \lambda)\big) \\
	& = d\big( (\mathbf{1}+x ) \star g\big) + \sum_{i\geq 2} \mu^i\big(x^{\otimes i}\big) \star g + \mu^2\big(\!\sh(x,g)\big) \\ 
	& = dg + d(x \star g) -d(x) \star g  + \mu^2\big(\!\sh(x,g)\big) \\
	& = dg + x \star dg \\ 
	& = \big(\mathbf{1}+x\big) \star d(g).
	\end{aligned}
	\end{displaymath}
	\noindent Here we used that $x$ is a Maurer-Cartan element as well as \eqref{EquationDifferentialStarProduct} and \eqref{EquationCalculationStarProduct}. As before, uniqueness implies $f= dg$. It follows that $\exp(h)$ and $\operatorname{Id}_V$ are Quillen homotopic if and only if $h $ is a coboundary and hence by definition of $\sim_\infty$, $\exp(v_1)$ and $\exp(v_2)$ are Quillen homotopic if and only if $\exp(v_1) \sim_\infty \exp(v_2)$. Thus, the result follows from \Cref{PropositionIntegrationHochschildClassesAInfinityAlgebra}.
\end{proof}

\begin{rem}\label{RemarkMaurerCartanSpaces}
The de Rham algebra $\Omega_1$ from \Cref{SectionHomotopyEquivalene} is part of a simplicial dg algebra $\Omega_{\bullet}$. Using a similar formalism as the definition of (Quillen) homotopy of Maurer-Cartan elements, $\Omega_{\bullet}$ gives rise to a simplicial set $\MC_{\bullet}(A)$ for any complete $A_\infty$-algebra $A$ called the \textit{Maurer-Cartan space}. If $A$ is also proper, it is known to be a Kan complex. The points ($0$-simplices)  of $\MC_{\bullet}(A)$  are in bijection with Maurer-Cartan elements of $A$ and the set $\overline{\MC}(A)$ of homotopy classes of Maurer-Cartan elements agrees with $\pi_0(\MC_{\bullet}(A))$, its set of connected components. One can interpret the proof of \Cref{TheoremWeakEquivalenceGroupoids} as showing that the exponential induces a non-trivial bijection between the homotopies ($1$-simplices of the Maurer-Cartan space) to the trivial Maurer-Cartan element of the $A_\infty$-structure on the (non-proper!) Hochschild complex and certain $1$-simplices in the \textit{exponential group} of its underlying dg Lie algebra $L$ with the Gerstenhaber bracket.  Over a field of characteristic $0$, the exponential group is related to the Maurer-Cartan space of $L$ via the loop space construction. Bypassing the exponential, we wonder whether this phenomenon is the incarnation of a weak equivalence between the loop space of the Maurer-Cartan space of the Lie structure and the Maurer-Cartan space of $A_\infty$-structure of the Hochschild complex, or more generally, any $B_\infty$-algebra which satisfies similar completeness assumptions, cf.~\Cref{SectionBInfinityAlgebras}.
\end{rem}
\noindent  As a consequence of the theorem, we immediately obtain the following corollary.
\begin{cor}\label{CorollaryExponentialInjectiveHomotopyClassesFunctors}Let $\bU$ be a semi-free dg category over a field of characteristic $0$. Then the exponential map induces an isomorphism of groups
	\begin{displaymath}
		\begin{tikzcd}
			\exp: \HHH^1_+(\bU,\bU) \arrow{r}{\sim} & \Aut^{\infty, h}_+(\bU). 
		\end{tikzcd}
	\end{displaymath}
\end{cor}
\begin{proof}Follows from \Cref{LemmaIsotopiesVSMaurerCartanElements}, \Cref{LemmaDGGaugeEquivalenceGaugleEquivalence}, the equivalence of homotopy  equivalence and gauge equivalence as well as \Cref{TheoremWeakEquivalenceGroupoids}.
\end{proof}

\subsection{Comparison of homotopy and weak equivalence via $A_\infty$-centers}\ \medskip

\noindent We investigate situations in which the notions of homotopy and weak equivalence for $A_\infty$-isotopies coincide. An important tool is the $A_\infty$-center of an $A_\infty$-category introduced in \cite{BriggsGelinas}. The projection $\eta \mapsto \eta^0$, $\eta \in C(\bA)=\Fun(\bA, \bA)(\operatorname{Id}_{\bA}, \operatorname{Id}_{\bA})$, descends to a map
\begin{displaymath}
\begin{tikzcd}[row sep=0.5em]
\Pi^{\bullet}:\HHH^{\bullet}(\bA, \bA) \arrow{r} & \sfZ_{\gr}^{\bullet}\big(\HH^{\bullet}(\bA)\big),
\end{tikzcd}
\end{displaymath}
\noindent where the latter denotes the graded center of the graded homotopy category, cf.~\cite[Section 2]{KrauseYe}. Elements of $Z^n_{\gr}(\HH^{\bullet}(\bA))$ are collections of morphisms $(\phi_A)_{A \in \Ob{\bA}}$, $\phi_A \in \Hom_{\HH^{\bullet}(\bA)}(A,A)$ of degree $n$ which commute with a morphism of degree $m$ in $\HH^{\bullet}(\bA)$ up to the sign $(-1)^{mn}$. The map is called the \textbf{characteristic morphism} and is natural under quasi-isomorphisms between $A_\infty$-algebras. If $\bA$ is minimal, the image of $\Pi$ in $\sfZ_{\gr}^{\bullet}(\bA)$ is called the \textbf{$A_\infty$-center}, a notion of due to Briggs-Gelinas \cite{BriggsGelinas}. 

\begin{prp}\label{PropositionComparisonHomotopyWeakEquivalence}
Let $\bA$ be a  cohomologically unital $A_\infty$-category. If $\Img \Pi^0 \subseteq \sfZ^0_{\gr}(\HH^0(\bA))$ contains all invertible elements of the center, then $A_\infty$-isotopies $F, G$ of $\bA$ are homotopic if and only if they are weakly equivalent.
\end{prp}
\begin{proof}By composing $F$ with the inverse of $G$, we can reduce the problem to the case $F=\operatorname{Id}_{\bA}$. Because all statements and assumptions are invariant under quasi-isomorphisms, we may assume that $\bA$ is minimal and hence strictly unital. By \Cref{LemmaHomotopicFunctorsWeaklyEquivalent}, $F \sim G$ implies $F \approx G$. To show the converse, let $\eta: \operatorname{Id}_{\bA} \rightarrow G$ and $\varepsilon: G \rightarrow \operatorname{Id}_{\bA}$ be homotopy inverse natural transformations of degree $0$. It is sufficient to replace $\eta$ by a homotopy, that is, by a natural transformation $\eta_2$ such that $\eta_2^{0}(A)=\operatorname{Id}_A \in \bA(A,A)$ for all $A \in \Ob{\bA}$. The weight $0$ component of $\mu^2_{\Fun(\bA, \bA)}(\varepsilon, \eta)$ at $A \in \Ob{\bA}$ is $\mu^2_{\bA}\big(\varepsilon^0(A), \eta^0(A)\big) \in \bA(A,A)^0$. By assumption and minimality, this agrees with the unit in $C(\bA)$. The fact that $\operatorname{Id}_{\bA}=G^1$ and \eqref{EquationNaturalTransformationSquare} show that $\eta^0, \varepsilon^0 \in \sfZ(\HH^0(\bA))=\sfZ^0_{\gr}(\bA)$ and are mutually inverses of each other. By assumption, $\varepsilon^0$ can be completed to a natural transformation $\varepsilon' \in \Fun(\bA, \bA)(\operatorname{Id}_{\bA}, \operatorname{Id}_{\bA}) \cong C(\bA)$ and by construction $\eta_2\coloneqq \mu^2_{\Fun(\bA, \bA)}(\eta, \varepsilon') \in \Fun(\bA, \bA)(\operatorname{Id}_{\bA}, G)$ is a homotopy.
\end{proof}
\noindent \Cref{PropositionComparisonHomotopyWeakEquivalence} applies under any of the following assumptions:
\begin{itemize}
	\item $\HH^0(\bA)$ has the ``smallest possible center'', that is, it consists only of functions $\eta^0$ such that $\eta^0(A) \in \Hom_{\HH^0(\bA)}(A,A)$ is a multiple of the identity for all $A \in \Ob{\bA}$. This is the case if all objects of $\bA$ are \textit{cohomologically schurian} in the sense that $\Hom_{\HH^0(\bA)}(A,A)\cong \Bbbk$ for all $A \in \Ob{\bA}$, e.g.~if the objects of $\bA$ correspond to a (possibly infinite) simple-minded collection in the sense of \cite{KoenigYangSilting} in which case the condition is automatically satisfied if $\Bbbk$ is algebraically closed and $\bA$ is \textbf{proper}, that is, all its morphism complexes have finite-dimensional total cohomology.  
		\item $\bA$ is formal;
	\item $\bA$ is an $E_{\infty}$-algebra, e.g.~any commutative, minimal $A_\infty$-algebra or $\bA=C^{\ast}(X, \Bbbk)$, the dg algebra of singular cochains on a topological space $X$;
	\item more generally, if $\bA$ is an $E_2$-algebra, e.g.~the Hochschild complex of any $A_\infty$-category.
\end{itemize}
\noindent In all the aforementioned cases, $\Pi^0$ is surjective onto $\sfZ^0_{\gr}(\HH^{\bullet}(\bA))$. This is straightforward in the first case and the other cases are discussed in the paragraphs after \cite[Corollary 3.13]{BriggsGelinas}. Finally, we remark that injectivity of the exponential is preserved when passing to the perfect or pretriangulated hull.
\begin{rem}\label{RemarkCharacteristicMorphismSplits}In the case that $\bA$ is formal, all maps $\Pi^i$, $i \geq 0$ split: up to quasi-isomorphism we may assume that $\bA$ is a graded $\Bbbk$-linear category and the corresponding map $Z^i_{\gr}(\bA) \rightarrow \HHH^i(\bA,\bA)$ assigns to an element $\phi=(\phi_A)_{A \in \Ob{\bA}}$ the strict natural transformation $\eta$ with $\eta^0(A)=\phi_A$ for all $A \in \Ob{\bA}$.
\end{rem} 
\begin{lem}\label{LemmaRestrictionInjeciveHochschild}
	Let $\bB$ be a perfect $A_\infty$-category and suppose that $\bA \subseteq \bB$ is a full subcategory which split generates $\bB$, that is, the smallest thick subcategory of $\cD(\bB)$ which contains $\HH^0(\bA)$ also contains $\HH^0(\bB)$. Then the restriction map $\HHH^1_+(\bB, \bB) \rightarrow \HHH^1_+(\bA, \bA)$ is injective and if $\exp_{\bA}$ is injective, then so is $\exp_{\bB}$.
\end{lem}
\begin{proof}
	By assumption, the inclusion $\bA \subseteq \bB$ is a Morita equivalence and hence the restriction map $C(\bB, \bB) \rightarrow C(\bA, \bA)$ is a quasi-isomorphism of braces algebras which preserves the cohomological weight filtration. In particular, the restriction to $\HHH^1_+(\bB, \bB)$ is injective and has image in $\HHH^1_+(\bA, \bA)$.
	
	For the second assertion, the universal property of $\Perf \bA \simeq \bB$ of \Cref{PropositionUniversalPropertiesHQE}) implies that every $F \in \Aut^{\infty}(\bA)$ admits an extension to an $A_\infty$-functor $\widehat{F} \in \Aut^{\infty}(\bB)$ and this extension is unique up to weak equivalence. Moreover, every $A_\infty$-isotopy of $\bB$ restricts to an $A_\infty$-isotopy on any of its subcategories as do natural transformations between them. Hence, there is a natural restriction group homomorphism $\Aut^{\infty}_+(\bB) \rightarrow \Aut^{\infty}_+(\bA)$ which is injective by the uniqueness of extensions. Because the restriction morphism $C(\bB) \rightarrow C(\bA)$ is one of brace algebras, the maps fit into a commutative diagram
	\begin{displaymath}
		\begin{tikzcd}[row sep=3em]
			\HHH^1_+(\bB, \bB) \arrow{rr}{\exp_{\bB}}[swap]{\sim} \arrow[hookrightarrow]{d} && \Aut^{\infty, h}_+(\bB) \arrow[twoheadrightarrow]{rr} \arrow{d}  && \Aut^{\infty}_+(\bB) \arrow[hookrightarrow]{d} \\
			\HHH^1_+(\bA, \bA) \arrow{rr}{\sim}[swap]{\exp_{\bA}} && \Aut^{\infty,h}_+(\bA) \arrow[twoheadrightarrow]{rr} &&  \Aut^{\infty}_+(\bA).
		\end{tikzcd}
	\end{displaymath}
	\noindent In particular, the middle vertical map is injective. Thus, if the lower right horizontal map is injective, then so is the upper right horizontal map.
\end{proof}
\noindent An analogous proof also shows that injectivity is preserved when passing from an $A_\infty$-category to its pretriangulated hull.

\begin{rem}\label{RemarkConnectingMorphism}
For any $A_\infty$-algebra $A$, the characteristic morphism $\Pi^0$ arises from the long exact sequence associated to the short exact sequence of complexes
\begin{displaymath}
\begin{tikzcd}
0 \arrow{r} & W_1C(A) \arrow{r} & C(A) \arrow{r} & A \arrow{r} & 0.
\end{tikzcd} 
\end{displaymath}
\noindent The corresponding connection morphisms $\HH^i(\bA) \rightarrow \HH^{i+1}\big(W_1C(A)\big)$ admit lifts to the chain level as discussed in the paragraph subsequent to \cite[Proposition 3.6]{BriggsGelinas}. In particular, the injectivity condition in \Cref{IntroTheoremA} can also be verified by showing that all invertible elements of $Z(\HH^0(\bA))$ lie in the kernel of the connecting morphism.
\end{rem}

\begin{rem}\label{RemarkWeakerConditionsConjecture} Suppose that $\cC$ is a graded $\Bbbk$-linear category such that the homomorphism $\Pi^1: \HHH^1(\cC, \cC) \rightarrow Z^1_{\gr}(\HH^{\bullet}(\cC))$ vanishes identically, for example any category which satisfies the conditions of \Cref{Conjecture}. Note that  every element $\phi=(\phi_C)_{C \in \Ob{\cC}}$ in the image of $\Pi^1$ has the property that $\phi_C^2=- \phi_C^2$ and hence $\phi_C^2=0$. Because $\cC$ has no higher $A_\infty$-multiplications this means that that $\phi_C$ is a Maurer-Cartan element of the algebra $\Hom^{\bullet}_{\cC}(C,C)$ and hence whenever $\phi_C \neq 0$, the free dg module associated to $C$ admits a deformation (cf.~\Cref{DefinitionDeformations}). The idea is that elements of the identity component of $\DPic(\cC)$ map objects of $\cC$ to their deformations in $\cD(\cC)$ and so in order to be able to expect that elements in this component are representable by elements in $\Aut_{\circ}^{\infty}(\cC)$ and hence by $A_\infty$-endofunctors of $\cC$, one needs to ensure that all objects of $\cC$ are 'rigid' and not deformed under the infinitesimal transformations represented by the Hochschild classes. The vanishing of $\Pi^1$ is a way to ensure this and a more general form of \Cref{Conjecture} only assumes that $\cC$ is formal and $\Pi^1=0$.
\end{rem}

\subsection{Quasi-invariance of $B_\infty$-structure and the weight filtration}\ \medskip

\noindent We recall some parts of Keller's proof \cite{KellerInvarianceHigherStructures} that the Hochschild complex of a dg category is a Morita invariant in the homotopy category $\operatorname{Ho}(B_\infty)$ of  of $B_\infty$-algebras. We will show that his isomorphisms in $\operatorname{Ho}(B_\infty)$ preserve the cohomological weight filtration if induced by a quasi-equivalence between $A_\infty$-categories. A similar statement was already proved by Briggs and Gelinas  \cite{BriggsGelinas} who showed that every quasi-isomorphism $F: A \rightarrow B$ between $A_\infty$-algebras induces a canonical isomorphism $C(A) \rightarrow C(B)$ in the homotopy category of $A_\infty$-algebras which induces an isomorphism between the cohomological weight filtration. However, it is not clear to us whether they are also isomorphisms in $\operatorname{Ho}(B_\infty)$.
\subsubsection{Overview on $B_\infty$-algebras}\label{SectionBInfinityAlgebras}\ \medskip

\begin{definition}
A \textbf{$B_\infty$-algebra} is an $A_\infty$-algebra $B$ together with a multiplication map $T(B[1]) \otimes T(B[1]) \rightarrow T(B[1])$ which turns the dg coalgebra $T(B[1])$ into a dg bialgebra\footnote{The tensor coalgebra is then automatically a dg Hopf algebra.}.
\end{definition}
\noindent The multiplication map amounts to a family of ``multi-braces``
\begin{displaymath}
\begin{tikzcd}
b_{r,s}: B^{\otimes r} \otimes B^{\otimes s} \arrow{r} & B, & r,s \geq 0,
\end{tikzcd}
\end{displaymath}
\noindent satisfying certain associativity conditions. A brace algebra is a $B_\infty$-algebra with $\mu_B=0$ and $b_{r,\ast}=0$ for all $r \geq 2$, $\ast \in \mathbb{N}$. The remaining ones are exactly the brace operations. If $V$ is a brace algebra with a Maurer-Cartan element $\mu \in V^1$, then $V$ becomes a $B_\infty$-algebra with the induced $A_\infty$-structure.
In particular, the Hochschild complex of any $A_\infty$-algebra is an example of a $B_\infty$-algebra. A morphism of $B_\infty$-algebras is an $A_\infty$-functor whose associated dg cofunctor is a morphism of dg bialgebras and a quasi-isomorphism is defined as for $A_\infty$-algebras. We will consider the homotopy category $\operatorname{Ho}(B_\infty)$ of $B_\infty$-algebras obtained by formally inverting all quasi-isomorphisms.

\subsubsection{Gluing of $A_\infty$-categories along functors}\label{SectionGluingAInfinityCategories}\ \medskip

\noindent We recall the ``upper triangular'' gluing construction of $A_\infty$-categories along an $A_\infty$-functor which is the natural generalization of its differential graded counterpart as discussed for example in \cite[Section 3.2]{OrlovNoncommutativeSchemes}. Let $\bA, \bB$ be $A_\infty$-categories and let $F: \bA \rightarrow \bB$ be an $A_\infty$-functor.\medskip

\noindent We define a new $A_\infty$-category $\bD=\bA \cup_{F} \bB$ with objects $\Ob{\bA} \sqcup \Ob{\bB}$ and morphisms
\begin{displaymath}
	\begin{aligned}
		\bD(B, A) & \coloneqq \bB\big(B, F^0(A)\big), \\
		\bD(A, B) &  \coloneqq  0, \\
		\bD(A, A') & \coloneqq \bA\big(A, A'\big), \\
		\bD(B, B') & \coloneqq \bB\big(B, B'\big).
	\end{aligned}
\end{displaymath} 
\noindent for all objects $A, A' \in \Ob{\bA}$, $B, B' \in \Ob{\bB}$. The $A_\infty$-structure on $\bD$ is the unique structure such that the fully faithful inclusions 
\begin{displaymath}
	\begin{array}{ccc}
		\iota_{\bA}:\bA \hookrightarrow \bD, & \text{ and } & \iota_{\bB}: \bB \hookrightarrow \bD,
	\end{array}
\end{displaymath}
\noindent  are strict $A_\infty$-functors and such that all other non-trivial multiplications are attained on elements of the form 
\begin{displaymath}
	\begin{aligned}
		f \in & \, \bB(B_{r-1}, B_r) \otimes \cdots \otimes \bB(B_0, B_1) \otimes  \bD(B_0, A_s) \otimes  \\ & \bA(A_{s-1}, A_s) \otimes \cdots \otimes \bA(A_0, A_1),
	\end{aligned}
\end{displaymath}
\noindent for some $r, s \geq 0$, for which one defines
\begin{displaymath}
	\mu_{\bD}(f) \coloneqq \sum\mu_{\bB}\Big(\operatorname{Id}^{\otimes (r+1)} \otimes F^{j_1} \otimes \cdots \otimes F^{j_{m}}\Big)(f).
\end{displaymath}
\noindent The sum ranges over all $j_1, \dots, j_m \geq 1$, $m \geq 0$ such that $\sum_i j_i=s$.  It is not difficult but somewhat tedious to verify that $\bD$ is an $A_\infty$-category. We note that if $F$ is a quasi-equivalence, then so are $\iota_{\bA}$ and $\iota_{\bB}$. Moreover, for every $A_\infty$-functor $H: \bB \rightarrow \bE$, there exists an $A_\infty$-functor $\widehat{H}: \bA \cup_F \bB \rightarrow \bE$ such that $\widehat{H} \circ \iota_{\bB}=H$ and $\widehat{H} \circ \iota_{\bA}=H \circ F$. This is analogous to \cite[Proposition 3.11]{OrlovNoncommutativeSchemes} in the dg case.

 \subsubsection{From quasi-equivalences to maps between Hochschild spaces and Keller's approach}\ \medskip
 
 \noindent  By restricted naturality  of the Hochschild complex, any quasi-equivalence $F:\bA \rightarrow \bB$ between $A_\infty$-categories induces a canonical span $C(F): C(\bA) \rightarrow C(\bB)$ in the category of $B_\infty$-algebras defined as
\begin{displaymath}
	\begin{tikzcd}	
	\text{}	& C(\bA \cup_F \bB) \arrow{dl}[swap]{\res(\iota_{\bA})} \arrow{dr}{\res(\iota_{\bB})} \\
	C(\bA) & & C(\bB),
	\end{tikzcd}
\end{displaymath}
\noindent for which we also use the formal expression $C(F)={\res(\iota_{\bB})} \circ \res(\iota_{\bA})^{-1}$. It is worth emphasizing that $C(F)$ is even a span of brace algebras morphisms which will be the underlying reason for the naturality of the exponential map from  \Cref{IntroTheoremA}. Because $\iota_{\bA}$ and $\iota_{\bB}$ are fully faithful quasi-equivalences, $\res(\iota_{\bA})$ and $\res(\iota_{\bB})$ are quasi-isomorphisms and hence $C(F)$ represents an isomorphism in $\operatorname{Ho}(B_\infty)$. The proof is analogous to the case of dg categories. Alternatively it follows from the results of \Cref{SectionGelinasBriggs} below. If $G: \bB \rightarrow \bD$ is another $A_\infty$-functor, then $C(G \circ F)=C(G) \circ C(F)$ in $\operatorname{Ho}(B_\infty)$. Moreover, as a map in $\cD(\Bbbk)$, $C(F)$ only depends on the weak equivalence class of $F$. The first claim can be proved with the same strategy from \cite[Section 4.6. Theorem d)]{KellerInvarianceHigherStructures}. Although this is likely also true for the second claim  we give a more direct argument for both claims below which uses a comparison to similar maps between Hochschild complexes described by Briggs-Gelinas.

\subsubsection{Isomorphisms of weight filtrations after Briggs-Gelinas}\label{SectionGelinasBriggs}\ \medskip

\noindent We recall the construction from \cite{BriggsGelinas} which is adapted to the case of $A_\infty$-categories and the Hochschild complex\footnote{The original construction in \cite{BriggsGelinas} worked with augmented $A_\infty$-categories and functors and a version of the normalized Hochschild complex.}. Although similar in spirit, our arguments differ somewhat from the proofs in \cite{BriggsGelinas}. In particular, we extend the construction to quasi-equivalences and show that it only depends on the weak equivalence class of a functor. 

In what follows, we fix $A_\infty$-categories $\bA$ and $\bB$.  Given an $A_\infty$-functor $\Phi:\bA \rightarrow \bB$, let $C(\bA, \bB, \Phi)\coloneqq \Fun(\bA, \bB)(\Phi, \Phi)$ which we refer to as the \textit{relative Hochschild complex}. We recall from  \eqref{EquationBijectionMorphismsFunctorCategoryCoderivations} its bijection with spaces of coderivations.
For all $A_\infty$-functors $U: \bA' \rightarrow \bA$, $V: \bB \rightarrow \bB'$ the post- and precomposition $A_\infty$-functors \eqref{EquationPostcompositionFunctors} and \eqref{EquationPrecompositionFunctors} yield cochain maps
\begin{displaymath}
	\begin{tikzcd}
	 & C(\bA, \bB, \Phi) \arrow{dl}[swap]{U^{\ast}} \arrow{dr}{V_{\ast}} \\
	C(\bA', \bB, \Phi \circ U) && C\big(\bA, \bB', V \circ \Phi\big).
	\end{tikzcd}
\end{displaymath}
\noindent By  \Cref{CorollaryNaturalMapsDerivationPreserveWeightFiltrations},  $U^{\ast}$ and $V_{\ast}$ are also compatible with the respective weight filtrations. By construction, the assignments $F \mapsto F_{\ast}$ and $F \mapsto F^{\ast}$ are functorial and commute in the sense that $G_{\ast} \circ F^{\ast}=F^{\ast} \circ G_{\ast}$ for any pair  of functors $F: \bA' \rightarrow \bA$ and $G: \bB \rightarrow \bB'$. Since homotopic functors correspond to homotopic dg cofunctors, it follows that that the homotopy classes of $V_\ast$ and $U^{\ast}$ only depend on the homotopy classes of $U$ and $V$.
 Thus, if $F$ is a homotopy invertible functor, then $F_{\ast}$ and $F^{\ast}$ are homotopy equivalences. It follows immediately, that every homotopy invertible $A_\infty$-functor $F: \bA \rightarrow \bB$ induces an isomorphism  
\begin{displaymath}
\widetilde{C}(F) \coloneqq F^{\ast -1} \circ F_{\ast}: C(\bA) \rightarrow C(\bB),
\end{displaymath}
\noindent in $\cD(\Bbbk)$. We note that the map $\widetilde{C}(F)$ is analogous to the inverse of the map $C(F)$ from \cite{BriggsGelinas} (which differs from the map which we denote by $C(F)$). For weakly invertible functors, we can still regard $\widetilde{C}(F)$ as a span in $\cD(\Bbbk)$. Among other things, we show below that it represents an isomorphism in $\cD(\Bbbk)$.

\begin{prp}\label{PropositionMapsAgree}Let $F: \bA \rightarrow \bB$ be an $A_\infty$-functor. The following are true.
	
	\begin{enumerate}
		\item[i)] \label{Enumerate1} If $\bA$ and $\bB$ are cohomologically unital and $F$ is weakly invertible, then the span $\widetilde{C}(F)$ defines an isomorphism in $\cD(\Bbbk)$. Moreover, it only depends on the weak equivalence class of $F$.
		\item[ii)] \label{Enumerate2} For every quasi-equivalence $F: \bA \rightarrow \bB$, 
		\begin{displaymath}
		C(F)=\widetilde{C}(\iota_{\bB})^{-1} \circ \widetilde{C}(\iota_{\bA})
		\end{displaymath}
		\noindent  in $\cD(\Bbbk)$, where $\iota_{\bA}: \bA \hookrightarrow \bA \cup_F \bB$ and $\iota_{\bB}: \bB \hookrightarrow \bA \cup_F \bB$ denote the inclusions.
			\item[iii)] \label{Enumerate3} The assignment $F \mapsto \widetilde{C}(F)$ is functorial and maps homotopy invertible $A_\infty$-functors to quasi-isomorphisms which induce isomorphisms between the cohomological weight filtrations. If $\bA$ and $\bB$ are cohomologically unital, the same is true for all weakly invertible functors.
	\end{enumerate}
\end{prp}
\begin{proof}Functoriality of $\widetilde{C}$ is a direct consequence of the same property of $(-)_{\ast}$ and $(-)^{\ast}$ and the fact that they commute. Let $F, G: \bA \rightarrow \bB$ be $A_\infty$-functors and let $\eta: F \rightarrow G$  be a natural transformation. By \eqref{EquationNaturalTransformationSquare}, the post- and pre-composition functors \eqref{EquationPostcompositionFunctors} and \eqref{EquationPrecompositionFunctors} yield a homotopy commutative diagram in the category of cochain complexes
\begin{displaymath}
\begin{tikzcd}[row sep=3em]
	& \Fun(\bA, \bB)(F,F) \arrow{d}{\eta_{\ast}} \\
\Fun(\bA, \bA)\big(\operatorname{Id}_{\bA}, \operatorname{Id}_{\bA}\big) \arrow{ur}{F_{\ast}} \arrow{dr}[swap]{G_{\ast}} & \Fun(\bA, \bB)(F, G) & \Fun(\bB, \bB)\big(\operatorname{Id}_{\bB}, \operatorname{Id}_{\bB}\big) \arrow{ul}[swap]{F^{\ast}} \arrow{dl}{G^{\ast}}\\
  & \Fun(\bA, \bB)(G, G) \arrow{u}[swap]{\eta^{\ast}}
\end{tikzcd}
\end{displaymath}
\noindent  If $\bA$ and  $\bB$ are cohomologically unital and $\eta$ is invertible in $\HH^0\big(\Fun(\bA, \bB)\big)$, then $\eta^{\ast}$ and $\eta_{\ast}$ are homotopy equivalences as they are functorial in $\eta$ and only depend on the homotopy class of $\eta$ which follows from \eqref{EquationNaturalTransformationSquare} and its subsequent paragraph. Consequently, $\widetilde{F}$ and $\widetilde{G}$ are equivalent spans. Hence, if $U: \bU \rightarrow \bV$, $V: \bV \rightarrow \bU$ are mutually weakly inverse $A_\infty$-functors, then $\widetilde{C}(V) \circ \widetilde{C}(U) \simeq \widetilde{C}(V \circ U)=\widetilde{C}(\operatorname{Id}_{\bU})$ as spans in $\cD(\Bbbk)$. The latter represents an isomorphism and hence so do $\widetilde{C}(U)$ and $\widetilde{C}(V)$. This proves i). If $\iota$ is any strict inclusion between $A_\infty$-categories, one  easily sees that 
\begin{equation}\label{EquationRelationRestriction}
{\iota}_{\ast} \circ \res(\iota)=\iota^{\ast}
\end{equation}
\noindent on the level of complexes and hence $\res(\iota)=\iota_{\ast}^{-1} \circ \iota^{\ast}$ in $\cD(\Bbbk)$ if $\iota$ is a quasi-equivalence. An application of this to $\iota \in \{\iota_{\bA}, \iota_{\bB}\}$ shows ii). It remains to show that $\widetilde{C}(F)$ induces isomorphisms between the cohomological weight filtrations. Because its inverse in $\cD(\Bbbk)$ is of the same form (we assume that $F$ is homotopy or weakly invertible), it suffices to show that it respects the weight filtrations which follows from \Cref{CorollaryNaturalMapsDerivationPreserveWeightFiltrations}.
\end{proof}
\begin{cor}\label{CorollaryQuasiEquivalenceFilteredQuasiIso}
For every quasi-equivalence $F: \bA \rightarrow \bB$ between cohomologically unital $A_\infty$-categories, $C(F): C(\bA) \rightarrow C(\bB)$ is an isomorphism in $\operatorname{Ho}(B_\infty)$ and induces an isomorphism between the cohomological weight filtrations.
\end{cor}

\subsection{Naturality and proofs of \Cref{IntroTheoremA} and \Cref{IntroTheoremNaturality}}\ \medskip

\noindent We show that the exponential map is natural with respect to quasi-equivalences and prove the general case of \Cref{CorollaryExponentialInjectiveHomotopyClassesFunctors}. If $F:\bA \rightarrow \bB$ is a homotopy or weakly invertible $A_\infty$-functor between cohomologically unital categories, we denote by $F_{\ast}$ the induced isomorphisms $\Aut^{\infty, h}(\bA) \rightarrow \Aut^{\infty, h}(\bB)$ and $\Aut^{\infty}(\bA) \rightarrow \Aut^{\infty}(\bB)$ respectively.
\begin{thm}\label{TheoremNaturalityExponentialQuasiIsos}Let $\bA, \bB$ be cohomologically unital $A_\infty$-categories over a field of characteristic $0$.
	\begin{enumerate}
		\item The exponential descends to maps $\exp_{\bA}^h:\HHH^1_+(\bA,\bA) \rightarrow \Aut^{\infty, h}(\bA)$ and $\exp_{\bA}:\HHH^1_+(\bA,\bA) \rightarrow \Aut^{\infty}(\bA)$ which are related by the projection $\Aut^{\infty, h}(\bA) \twoheadrightarrow \Aut^{\infty}(\bA)$.
		\item  Every quasi-equivalence $F: \bA \rightarrow \bB$ induces a commutative diagram
		\begin{displaymath}
			\begin{tikzcd}[row sep=3em]
				\HHH^1_+(\bA,\bA)  \arrow[twoheadrightarrow]{rr}{\exp_{\bA}}[swap]{} \arrow{d}{\sim}[swap]{\HH^0(C(F))} &&  \Aut^{\infty}_+(\bA) \arrow{d}{F_\ast}[swap]{\sim} \\
				\HHH^1_+(\bB,\bB) \arrow[twoheadrightarrow]{rr}{\exp_{\bB}}[swap]{} &&  \Aut^{\infty}_+(\bB). 
			\end{tikzcd}
		\end{displaymath}
		\noindent consisting of surjective group homomorphisms with respect to composition of functors and the Baker-Campbell-Hausdorff product.
		
			\item  Every quasi-isomorphism $F: \bA \rightarrow \bB$ induces a commutative diagram
		\begin{displaymath}
		\begin{tikzcd}[row sep=3em]
		\HHH^1_+(\bA,\bA)  \arrow{rr}{\exp_{\bA}^h}[swap]{\sim} \arrow{d}{\sim}[swap]{\HH^0(C(F))} &&  \Aut^{\infty, h}_+(\bA) \arrow{d}{F_\ast}[swap]{\sim} \\
		\HHH^1_+(\bB,\bB) \arrow{rr}{\exp_{\bB}^h}[swap]{\sim} &&  \Aut^{\infty, h}_+(\bB). 
		\end{tikzcd}
		\end{displaymath}
		\noindent consisting of group isomorphisms with respect to the same group structures as in (2).
	\end{enumerate}  
\end{thm}
\begin{proof}
First, we observe that both vertical maps in the first commutative diagram only depend on the weak equivalence class of $F$. Let $G$ be a weak inverse of $F$ and let $\bD \coloneqq \bA \cup_F \bB$ with its strict inclusions $\iota_{\bA}: \bA \hookrightarrow \bD$ and $\iota_{\bB}: \bB \hookrightarrow \bD$. From  \Cref{SectionGluingAInfinityCategories}, we recall the definition of an  $A_\infty$-functor $\widehat{H}: \bD \rightarrow \bD$ associated to a functor $H: \bB \rightarrow \bD$. 
By \Cref{PropositionMapsAgree}, $C(F)$ is a span $\res(\iota_{\bB})\circ \res(\iota_{\bA})^{-1}$ of quasi-isomorphisms which induce isomorphisms between the cohomological weight filtrations. For any $A_\infty$-category $\bE$, let us write $\Theta(\bE)$ for the group of group-like elements $\operatorname{Id}_{\bE}+C_+^0(\bE)$ of $W_1C(\bE)$ equipped with the $\odot$-product. There exists a natural inclusion $\Theta(\bD) \hookrightarrow C^0(\bD)$ and restriction maps $\Theta(\bD) \rightarrow \Theta(\bA)$ and $\Theta(\bD) \rightarrow \Theta(\bB)$ given by the restriction of $\res(\iota_{\bA})$ and $\res(\iota_{\bB})$ along these inclusions. The assignment $H \mapsto \widehat{G} \circ H \circ  \iota_{\bA}$ induces an isomorphism of multiplicative monoids
\begin{equation}\label{EquationMapEndomorphisms}
\begin{tikzcd}
\End^{\infty}(\bD) \arrow{r}{\sim} & \End^{\infty}(\bA),
\end{tikzcd}
\end{equation}
\noindent between the monoids of weak equivalence classes of $A_\infty$-endofunctors of $\bD$ and $\bA$, that is, $\End^{\infty}(\bD)=\Hom_{\Hqe}(\bD, \bD)$. Here, we used the following observations. First, $\widehat{G} \circ \iota_{\bA}=G \circ F \approx \operatorname{Id}_{\bA}$. Because $\iota_{\bA}$ is weakly invertible (being a quasi-equivalence)  and weak inverses are unique up to weak equivalence, it follows that $\iota_{\bA} \circ \widehat{G} \approx \operatorname{Id}_{\bD}$. The two equalities then also imply that the  inverse of \eqref{EquationMapEndomorphisms} is realized by the assignment $H \mapsto \iota_{\bA} \circ H \circ \widehat{G}$. In the same way, by composing with $\iota_{\bB}$ and $\widehat{\operatorname{Id}_{\bB}}$ instead, we obtain an isomorphism of monoids $\End^{\infty}(\bD) \rightarrow \End^{\infty}(\bB)$. The resulting composition $\End^{\infty}(\bA) \rightarrow \End^{\infty}(\bD) \rightarrow \End^{\infty}(\bB)$ is realized by $H \mapsto \big(\widehat{\operatorname{Id}_{\bB}} \circ \iota_{\bA}\big) \circ H \circ \big( \widehat{G} \circ \iota_{\bB}\big)= F \circ H \circ G$. Similar to \eqref{EquationRelationRestriction}, we see that $\End^{\infty}(\bD) \rightarrow \End^{\infty}(\bA)$ also coincides with the span of isomorphisms
\begin{displaymath}
\begin{tikzcd}
\End^{\infty}(\bD) \arrow{rr}{{\iota_{\bA}}^{\ast}} && \Hom_{\Hqe}(\bA, \bD) && \arrow{ll}[swap]{{(\iota_{\bA})}_{\ast}}   \End^{\infty}(\bA),
\end{tikzcd}
\end{displaymath}
\noindent and similar for the other monoid map. Therefore, keeping in mind that $\res(\iota_{\bA})$ and $\res(\iota_{\bB})$ are morphisms of brace algebras, it is now straightforward to verify that the following diagram commutes:
\begin{displaymath}
	\begin{tikzcd}[row sep=3em]
		Z^0\big(C_+(\bA)\big)  \arrow{rr}{\exp}[swap]{\sim}  && \Theta(\bA) \arrow{rr} &&  \End^{\infty}(\bA)  \arrow[xshift=+0.0em, bend left=7em]{dd}{F_{\ast}=G_{\ast}^{-1}}   \\
			Z^0\big(C_+(\bD)\big)  \arrow{rr}{\exp}[swap]{\sim} \arrow[swap]{d}{\res(\iota_{\bB})} \arrow{u}{\res(\iota_{\bA})} &&  \Theta(\bD) \arrow[swap]{u}{\res} \arrow{d}{\res} \arrow{rr} && \End^{\infty}(\bD) \arrow{u} \arrow{d} \\
			Z^0\big(C_+(\bB)\big) \arrow{rr}{\exp}[swap]{\sim}  && \Theta (\bB) \arrow[]{rr} &&  \End^{\infty}(\bB). 
	\end{tikzcd}
\end{displaymath}
\noindent All vertical and horizontal maps between objects in the two leftmost columns are group homomorphisms with respect to the $\odot$-product and the Baker-Campbell-Hausdorff product. Moreover, all other horizontal and vertical maps are morphisms of monoids with respect to the $\odot$-product and composition of functors.
We now observe that if the composite map $Z^0(C_+(\bB)) \rightarrow \End^{\infty}(\bB)$ in the lower row of the diagram descends to a map $\HHH^1_+(\bB, \bB) \rightarrow  \End^{\infty}(\bB)$ on cohomology, then so does the composite map of the upper row. To see this, let $a \in Z^0(C_+(\bA))$ be a cocycle. We choose a cocycle $b \in C^0_+(\bB)$ whose cohomology class $[b]$ corresponds to $[a]$ under the isomorphism $\HHH^1_+(\bA, \bA) \cong \HHH^1_+(\bB, \bB)$ induced by $\res(\iota_{\bA})$ and $\res(\iota_{\bB})$. By assumption, $\exp(b) \sim \operatorname{Id}_{\bB}$ if $[b]=0$. This happens if and only if $[a]=0$ and by commutativity of the above diagram, the weak equivalence class of $\exp(a)$ coincides with $F_{\ast}\big(\exp([b])\big)$. Because $F_\ast$ is an isomorphism of groups, the latter is trivial if  $[b]=0$ which happens if and only if $[a]=0$. The same arguments also show that the resulting map $\HHH^1_+(\bA, \bA) \rightarrow \Aut^{\infty}(\bA)$ is injective if and only if this is the case for $\HHH^1_+(\bB, \bB) \rightarrow \Aut^{\infty}(\bB)$. If $F$ is a quasi-isomorphism, then by using similar arguments as before we also obtain an analogous commutative diagram where every instance of $\End^{\infty}(-)$ is replaced by $\End^{\infty, h}(-)$, the monoid of homotopy classes of $A_\infty$-endofunctors. The only difference is that $G$ now has to be a homotopy inverse of $F$, which exists by assumption.

The final step is to establish the existence and injectivity of the map $\HHH^1_+(\bA, \bA) \rightarrow \Aut^{\infty, h}(\bA)$. By \Cref{PropositionYonedaRectification}, $\bA$ is canonically quasi-isomorphic to a dg category. The category $\mathsf{dgcat}_S$ of dg categories with a fixed set $S$ of objects is a model category with weak equivalences given by quasi-isomorphisms and cofibrant objects given by the retracts of semi-free dg categories with objects $S$. The proof of the latter is analogous to the proof of \cite[Proposition 2.40]{KarabasLee} for general dg categories and follows from the description of generating cofibrations  in $\mathsf{dgcat}_S$ via tensor categories as found in \cite[Section 3]{MuroDwyerKanModelStructure}.  By taking a cofibrant replacement in $\mathsf{dgcat}_S$, it follows that there exists a quasi-isomorphism $\bA \rightarrow \bA'$ to a semi-free dg category $\bA'$. By \Cref{CorollaryExponentialInjectiveHomotopyClassesFunctors}, the exponential descends to an injective map $\HHH^1_+(\bA', \bA') \hookrightarrow \Aut^{\infty,h}(\bA')$. This finishes the proof.
\end{proof}

\begin{cor}\label{CorollaryExponentialInjectiveMainTheorem}
Let $\bA$ be a cohomologically unital $A_\infty$-category over a field of characteristic $0$. If the image of $\Pi^0: \HHH^{0}(\bA, \bA) \rightarrow \sfZ^0_{\gr}(\HH^0(\bA))$ contains all invertible elements, then $\exp: \HHH^1_+(\bA,\bA) \rightarrow \DPic(\bA)$ is an embedding of groups with image $\Aut^{\infty}_+(\bA)$.
\end{cor}

\section{Necessary conditions for uniqueness of lifts}\label{SectionObstructionUniquenessLifting}

\noindent We show that the vanishing of $\HHH^1_+(\bA, \bA)$ is a necessary condition of the uniqueness of lifts of equivalences between homotopy categories to $A_\infty$-functors. Throughout this section, all $A_\infty$-categories and $A_\infty$-functors will be cohomologically unital and over a field of characteristic zero.
\begin{definition}
Let $\bA, \bB$ be $A_\infty$-categories. A \textbf{lift} of a $\Bbbk$-linear functor $H: \HH^{0}(\bA) \rightarrow \HH^{0}(\bB)$ is an $A_\infty$-functor $F: \bA \rightarrow \bB$ such that $\HH^{0}(F)$ and $H$ are naturally isomorphic.
\end{definition}

\begin{thm}[{\Cref{IntroTheoremC}}]\label{TheoremUniquenessLiftingVanishing}
Let $f: \HH^0(\bA) \rightarrow \HH^0(\bB)$ be an equivalence which has a lift. If $\exp_{\bA}$ is injective and if any two lifts of $f$ are weakly equivalent, then $\HHH^1_+(\bA, \bA)=0=\HHH^1_+(\bB,\bB)$.
\end{thm}
\begin{proof}
	Suppose $h \in \HHH^1_+(\bA, \bA)$ is non-zero. Then by injectivity of $\exp_{\bA}$, $\exp_{\bA}(h) \in \Aut_+^{\infty}(\bA)$ is not weakly equivalent to $\operatorname{Id}_{\bA}$. Let $F$ be any lift of $f$. In particular, because $F$ is weakly invertible, $F' \coloneqq F \circ \exp_{\bA}(h)$ cannot be weakly equivalent to $F$. But by construction we have $\HH^0(F')=\HH^0(F) \circ \HH^0(\exp(h))=\HH^0(F)$ in contradiction to our assumptions.
\end{proof}
\noindent The previous theorem is line with the sufficient conditions for uniqueness by Genovese \cite[Theorem 1.4]{GenoveseUniquenessLifting}. He proves uniqueness of lifts for certain functors which includes all equivalences whose domain $\bB$ is the pretriangulated hull of a full $\Bbbk$-linear subcategory $\bA \subset \bB$, that is, $\HH^0(\bB)$ is generated as a triangulated category by its full subcategory $\HH^0(\bA)$. It is not difficult to see that $\HHH^1_+(\bE, \bE)=0$ for all $\Bbbk$-linear categories $\bE$ and hence \Cref{LemmaRestrictionInjeciveHochschild} implies $\HHH^1_+(\bB, \bB)=0$. 

\begin{rem}\label{RemarkGeneralizationNecessaryCondition}
	We expect an analogue of \Cref{TheoremUniquenessLiftingVanishing} for general liftable functors $\HH^0(\bA) \rightarrow \HH^0(\bB)$ which states that in order for any liftable functor $f:\HH^0(\bA) \rightarrow \HH^0(\bB)$ to have unique lifts, $\HHH^1_+(\bA, \bA)$ needs to vanish. We did not verify the details but believe that the proof should look as follows. Given a lift $F: \bA \rightarrow \bB$ of $f$, we consider its relative Hochschild complex $C_F=\Fun(\bA, \bB)(F,F)$. This admits the structure of a right brace module over the brace algebra $C(\bA)=\Fun(\bA, \bA)(\operatorname{Id}_{\bA}, \operatorname{Id}_{\bA})$. Then, analogous to \cite[Proposition 3]{DotsenkoShadrinVallettePreLie}, one should be able to prove that 
	\begin{displaymath}
		\exp(r_{h})(F)= F \odot \exp(h),
	\end{displaymath}
	\noindent where $r_h: C_F \rightarrow C_F$ is the endomorphism determined by the right action of an element $h \in C(\bA)$ on $C_F$ and where $F$ is considered as an element of $C_F$ via its Taylor coefficients. The operation $\odot$ is defined analogously by means of the right action. If $h$ is a Hochschild $1$-cocycle, then a comparison of definitions should exhibit $F \odot \exp(h)$ as $F \circ \exp(h)$ and a similar calculation as in \Cref{PropositionBijectionMaurerCartanIsotopies} would show that $F \circ \exp(h) - F$ is a Maurer-Cartan element of the $A_\infty$-algebra $C_F$. We expect that via a similar route as in the proof of \Cref{TheoremWeakEquivalenceGroupoids}, one can show that the assignment $h \mapsto F \circ \exp(h)$ induces an embedding of $\HHH^1_+(\bA, \bA)$ into the set of homotopy classes of functors $\bA \rightarrow \bB$. As in the proof of \Cref{TheoremUniquenessLiftingVanishing} and by definition we have $\HH^0(F \circ \exp(h))=\HH^0(F)=f$. Under injectivity assumptions on $\exp$ this then again necessitates the vanishing of $\HHH^1_+(\bA, \bA)$ in order for $F$ to be the unique lift of $f$ up to weak equivalence.
\end{rem}

\section{Applications to derived Picard groups of Fukaya categories}\label{SectionApplicationsFukayaCategories}
\subsection{Fukaya categories of cotangent bundles and plumbings}
\subsubsection{Wrapped and compact Fukaya categories and Koszul duality} \ \medskip

\noindent  Throughout this section, we fix a field $\Bbbk$ of characteristic zero. Let $M$ be a connected, compact smooth manifold and let $T^{\ast}M$ denote its cotangent bundle. The latter is a Liouville domain, a special type of symplectic manifold, and we denote by $\cW$ its  wrapped Fukaya $A_\infty$-category in the sense of \cite{AbouzaidSeidelOpenStringAnalogue, GanatraPardonShendeCovariantlyFunctorialWrappedFukayaCategory}. To be precise, the definition of $\cW$ implicitly depends on the choice of a background class $b \in \HH^2(T^{\ast}M, \mathbb{Z}_2)$ for which ones takes the canonical choice, namely the pullback of the second Stiefel-Whitney class of $M$. The objects of $\cW$ are certain Lagrangian submanifolds of $T^{\ast}M$ which are possibly non-compact and which behave well near the non-compact ends of $T^{\ast}M$.

 If $M$ is closed, results by Abouzaid \cite{AbouzaidGenerationFukayaCategory} show that $\Perf \cW$ is generated by a single object, the cotangent fiber $F=T^{\ast}_qM \in \cW$ for any fixed $q \in M$. He showed further \cite{AbouzaidWrappedFukayaCategoryBasedLoops} that $\Hom_{\cW}(F,F)$ is quasi-equivalent to the dg algebra of singular chains
\begin{displaymath}
C_{-\ast}\coloneqq C_{-\ast}(\Omega_q M, \Bbbk),
\end{displaymath}
\noindent on the based (Moore) loop space $\Omega_q M$ of $M$ endowed with the Pontryagin product. Both results were later generalized to the non-closed case by  Chantraine-Dimitroglou Rizell-Ghiggini-Golovko \cite{ChantraineDimitroglouRizellGhigginiGolovko} and also by Ganatra-Pardon-Shende, cf.~\cite[Corollary 6.1]{GanatraPardonShendeMicrolocalMorseTheory} and \cite[Example 1.15]{GanatraPardonShendeSectorialDescent}. When endowed with the natural augmentation, the classical Eilenberg-Moore equivalence shows that the dg algebra of singular cochains $C^{\ast}\coloneqq C^{\ast}(M ,\Bbbk)$ on $M$ is determined by $C_{-\ast}$ through the quasi-isomorphism
\begin{displaymath}
C^{\ast} \simeq \RHom_{C_{-\ast}}(\Bbbk, \Bbbk),
\end{displaymath}
\noindent If $M$ is simply connected, one can switch the roles of $C^{\ast}$ and $C_{-\ast}$ in the above equivalence and $C^{\ast}$ and $C_{-\ast}$ are Koszul dual in this case. However, in general $C_{-\ast}$ is a stronger invariant than $C^{\ast}$ as illustrated by the example of lens spaces whose homotopy types can be told apart by $C_{-\ast}$ but not by $C^{\ast}$.  Geometrically, Koszul duality is realized by the passage from the cotangent fiber $F$ to the zero section $M \cong Z \subseteq T^{\ast}M$ as observed in \cite{EtguLekiliKoszulDualityPatterns, EkholmLekiliDuality}. If $M$ is closed and simply-connected, then the zero section generates the \textit{compact Fukaya category} $\cF$ of $T^{\ast}M$ which is by definition, the full subcategory of $\cW$ consisting of closed Lagrangian submanifolds of $T^{\ast}M$. This follows from \cite{FukayaSeidelSmithCategoricalViewpoint} where it was shown that all objects in $\cF$ are isomorphic to $Z$ in $\cD(\cF)$. Under the previous assumptions the duality also gives rise to quasi-equivalences
 \begin{displaymath}
 \begin{array}{ccc}
 	\Perf \cW \simeq \Dfd{\cF} & \text{and} & \Dfd{\cW} \simeq \Perf \cF,
 \end{array}
 \end{displaymath} where $\Dfd{-}$ denotes the subcategory of derived category consisting of all dg modules whose underlying cochain complex has finite dimensional total cohomology.

 \subsubsection{Derived Picard groups for cotangent bundles}\ \medskip

\noindent  The dg algebra $C_{-\ast}$ from the previous section satisfies $\HH^0(C_{-\ast})\cong \Bbbk$ and is concentrated in non-positive degrees showing that $F \subseteq \Perf \cW$ is a silting object and the triple $(\Perf \cF, \Perf \cW, F)$ is a $(\dim M +1)$-Calabi-Yau triple in the sense of \cite[Section 5.1]{IyamaYangSiltingReduction} as shown in \cite[Theorem 4.4]{BaeJeongKimCluster}. On the other hand $C^{\ast}$ is an $E_\infty$-algebra and \Cref{IntroTheoremA} therefore yields the following.
\begin{cor}\label{CorollaryTheoremE}Let $M$ be a compact smooth manifold. Then $\exp_{\cW}$ induces an embedding of groups
	\begin{displaymath}
		\begin{tikzcd}
			\HHH^1_+(C_{-\ast}, C_{-\ast}) \cong \HHH^1_+(\cW, \cW) \arrow[hookrightarrow]{r} & \DPic\big(\cW),
		\end{tikzcd}
	\end{displaymath}
\noindent and if $M$ is further closed and simply-connected, then $\exp_{\cF}$ also induces an embedding
	\begin{displaymath}
	\begin{tikzcd}
		\HHH^1_+(C^{\ast}, C^{\ast}) \cong \HHH^1_+(\cF,\cF)\arrow[hookrightarrow]{r} & \DPic\big(\cF).
	\end{tikzcd}
\end{displaymath}
\end{cor}
\noindent Because $C^{\ast}$ and $C_{-\ast}$ are Koszul dual if $M$ is simply-connected, there exists an isomorphism of graded Lie algebras \begin{displaymath}
\HHH^{\bullet}(C^{\ast}, C^{\ast}) \cong \HHH^{\bullet}(C_{-\ast}, C_{-\ast}),
\end{displaymath}
\noindent in this case which interchanges the weight filtration on either side with the so-called \textit{shearing filtration} on the opposite side of the isomorphism, cf.~\cite[Theorem 3.18]{BriggsGelinas}. In particular, if $M$ is \textit{formal}, that is, $C^{\ast}$ is formal, we obtain an embedding of a subset of $\HHH^1(\HH^{\bullet}(M), \HH^{\bullet}(M))$ into $\DPic(\cF)$ and \Cref{Conjecture} states that $\Aut_{\circ}^\infty(\HH^{\bullet}(C^{\star}))$ is a Morita invariant which describes the identity component of $\DPic(\cF)$.
 As $\Bbbk$ is of characteristic zero, a famous result by Deligne-Griffiths-Morgan-Sullivan \cite{DeligneGriffithsMorganSullivan} states that all compact K\"{a}hler manifolds are formal. Other examples are rationally elliptic spaces with positive Euler characteristic, e.g.~Grassmanians and other homogeneous spaces as well as all closed, simply-connected manifolds of dimension at most $6$, see \cite[Example 2.2]{BerglundStoll} for these and other examples. There is a dual story for $\DPic(\cW)$ if $M$ is \textit{coformal}, which is equivalent \cite{Saleh} to $C_{-\ast}$ being formal. Examples are spheres as well as any combination of their connected sums and products.

\subsubsection{Derived Picard groups for plumbings of cotangent bundles}\ \medskip

\noindent As version of \Cref{CorollaryTheoremE} is true for plumbings of cotangent bundles of spheres along a tree. For a tree $Q$ and $d \geq 2$, let $\cW_{d, Q}$ (res.~$\cF_{d, Q}$) denote the wrapped (resp.~compact) Fukaya category of the plumbing of cotangent bundles of the $d+1$-dimensional sphere along $Q$ in the sense of \cite{EtguLekiliKoszulDualityPatterns, BaeJeongKimCluster}. Then according to Theorem \cite[Theorem 4.6.]{BaeJeongKimCluster}, $\Perf \cW_{d, Q}$ and  $\Perf \cF_{d, Q}$ belong to a $(d+1)$-Calabi-Yau-triple and $\Perf \cW_{d, Q}$ is generated by a silting object whose indecomposable direct summands all have endomorphism ring isomorphic to $\Bbbk$. Moreover, $\Perf \cF_{d,Q}$ is generated by the associated simple-minded collection. In particular, by Theorem \Cref{IntroTheoremA}, we obtain embeddings of the respective subspaces of Hochschild cohomology of this silting object and the simple-minded collection into the derived Picard groups of $\cW_{d,Q}$ and $\cF_{d,Q}$. For plumbings of $2$-spheres, the endomorphism algebra of the silting object is quasi-isomorphic to the $3$-Calabi-Yau Ginzburg dg algebra $\Gamma_{Q}$ associated with $Q$, cf.~\cite{EtguLekiliKoszulDualityPatterns}. If $Q$ is non-Dynkin, $\HH^{\bullet}(\Gamma_{Q})$ is concentrated in degree $0$ and formal by \cite{Hermes}, see also \cite[Corollary 21]{EtguLekiliKoszulDualityPatterns} for a different proof. Moreover  its Koszul dual is given by the \textit{zigzag algebra} of $Q$, cf.~\cite{HuerfanoKhovanov}.
\begin{cor}\label{CorollaryPlumbings}
Let $Q$ be a tree. Then the exponential map induces an embedding of groups
\begin{displaymath}
	\begin{tikzcd}
		\HHH^1_+(\Gamma_Q, \Gamma_Q) \arrow[hookrightarrow]{r} & \DPic(\cW_{2, Q}).
\end{tikzcd}
\end{displaymath}
\end{cor}
\noindent The Hochschild cohomology of $\Gamma_{Q}$ and its Lie algebra structure, for Dynkin and non-Dynkin trees was computed in \cite{CrawleyBoeveyEtingofGinzburg, EtguLekiliKoszulDualityPatterns, LekiliUeda}. As before,\Cref{Conjecture} says in the non-Dynkin case that $\Aut_{+}^{\infty}(\HH^0(\Gamma_Q)) \rtimes \OutO(\HH^0(\Gamma_Q))$ should describe the identity component of the derived Picard group of $\cW_{2, Q}$. With the results of the very recent work by Karabas-Lee \cite{KarabasLeeWrappedFukayaPlumbings} one might be able to prove similar results for a much wider range of wrapped Fukaya categories of plumbings.

\subsection{Partially wrapped Fukaya categories after Haiden-Katzarkov-Kontsevich}\label{SectionPartiallyWrapped}\ \medskip

\noindent We give an outlook on the applications of \Cref{IntroTheoremA} to the computation of derived Picard groups of partially wrapped Fukaya categories of surfaces by the author \cite{OpperGradedGentle}.

\noindent  Let $\Sigma$ be a oriented smooth compact surface and let $\emptyset \neq \cM \subsetneq \partial\Sigma$ be a compact subset of its boundary such that $B \cap \cM \neq \emptyset$ for all components $B \subseteq \partial \Sigma$. Components $B$ such that $B \subset \cM$ will be called \textit{fully marked}, the others \textit{partially marked}. We also choose a line field $\eta$ on $\Sigma$, that is, a section $\eta: \Sigma \rightarrow \mathbb{P}(T\Sigma)$  of the projectivized tangent bundle, which amounts to a smoothly varying choice of $1$-dimensional subspaces $\eta_x \subseteq T_x\Sigma$ at each point $x \in \Sigma$. For example, the subspaces spanned by the vectors of any nowhere vanishing vector field form a line field. Not every line field arises in this way but much like vector fields, $\eta$ allows us to associate a \textit{winding number} $\omega_{\eta}(B) \in \mathbb{Z}$ to every boundary component $B$ of $\Sigma$.

To a triple $(\Sigma, \cM, \eta)$, Haiden-Katzarkov-Kontsevich \cite{HaidenKatzarkovKontsevich} associated a Morita class of triangulated $A_\infty$-categories called the \textbf{partially wrapped Fukaya category} $\Fuk(\Sigma, \cM, \eta)$.
These categories admit generators with formal endomorphism algebra which is  homologically smooth\footnote{A dg algebra $C$ is \textbf{homologically smooth} if its diagonal bimodule is perfect, that is, $C \in \Perf(C \otimes C^{\operatorname{op}})$.} but possibly infinite-dimensional and which belongs to the class of \textit{graded gentle algebras}. In the finite-dimensional ungraded case, that is, when the algebra is concentrated in degree $0$, these type of algebras first appeared in work by Assem-Skowronski \cite{AssemSkowronski}. They can be defined as quotients of path algebras which are obtained combinatorially from an \textit{arc system}, a collection of pairwise disjoint embedded paths on $\Sigma$ with endpoints in $\cM$. Every graded homologically smooth gentle algebra $A$ arises as a formal generator of some $\Fuk(\Sigma_A, \cM_A, \eta_A)$ as shown by Lekili-Polishchuk \cite{LekiliPolishchukGentle} (see also \cite{OpperPlamondonSchroll}).

One often distinguishes between two different cases of graded gentle algebras: finite-dimensional (``proper'') ones which are not necessarily homologically smooth and homologically smooth ones which are possibly infinite-dimensional.  The two cases are related by Koszul duality in the following way. Given a finite-dimensional graded gentle algebra $A$, there exist functors
\begin{displaymath}
\begin{array}{ccc}
	 \Perf(A) \hookrightarrow \Dfd{A^{\invex}} & \text{and} & \Perf(A^{\invex}) \rightarrow \Dfd{A},
\end{array}
\end{displaymath}
\noindent  where $A^{\invex}=\Omega A^{\vee}$ denotes the ``left'' Koszul dual over the semi-simple quotient $A/\rad(A)$ which is defined by the cobar construction of the $A/\rad(A)$-linear dual $A^{\vee}$ of $A$.  The algebra $A^{\invex}$ is quasi-isomorphic to an explicit homologically smooth graded gentle algebra. The first functor is a fully-faithful embedding whereas the second essentially surjective and both functors are equivalences if and only if $\Sigma_A$ contains no fully marked boundary component with vanishing winding number. In particular, this is always the case if $A$ is concentrated in degree $0$. Likewise, starting with a homologically smooth graded gentle algebra $C$, the (usual) Koszul dual $C^!$ is quasi-isomorphic to an explicit finite-dimensional graded gentle algebra.

 The derived Picard group of the bounded derived category of any ungraded finite-dimensional gentle algebra has been studied by the author in \cite{OpperDerivedEquivalences}. For such an algebra $A$, the central tool in loc.cit.~for the study of $\DPic(A)$ is the construction of a surjective group homomorphism 
\begin{displaymath}
	\begin{tikzcd}
		\Psi: \Aut\big(\cD^b(A)\big) \arrow[twoheadrightarrow]{r} & \MCG(\Sigma_A, \cM_A, \eta_A),
	\end{tikzcd}
\end{displaymath}
\noindent from the group of triangulated autoequivalences of $\cD^b(A)$ to the \textit{mapping class group} of a triple $(\Sigma_A, \cM_A, \eta_A)$ which admits an explicit construction\footnote{The setup in \cite{OpperPlamondonSchroll, OpperDerivedEquivalences} is slightly different in that every connected component $C$ of $\cM$ is replaced by a point in a finite set  $\cM' \subseteq \Sigma$ of distinguished points. The points which correspond to components $C \cong S^1$ then become interior points (``punctures'').} from $A$, cf.~\cite{OpperPlamondonSchroll}. Here, the mapping class group refers to the group of isotopy classes of diffeomorphisms $f: \Sigma \rightarrow \Sigma$ which restrict to a permutation of $\cM$ and such that the homotopy class of $\eta$ is invariant under pullback along $f$. A similar technique for derived equivalences also lead to a derived equivalence classification in the ungraded case in \cite{OpperDerivedEquivalences}, see also \cite{AmiotPlamondonSchroll, JinSchrollWang} for related classification results via different techniques.  The surjectivity of $\Psi$, exhibits $\Aut\big(\cD^b(A)\big)$ as an extension of the mapping class group by the kernel of $\Psi$. It seems difficult to determine this kernel in general but the problem becomes much more tractable if one restricts to its intersection with the derived Picard group, that is, triangulated functors which admit an $A_\infty$-lift. Under the assumption that $\Bbbk$ is algebraically closed and $\operatorname{char} \Bbbk \neq 2$, it was shown in \cite{OpperDerivedEquivalences} that
\begin{equation}\label{EquationKernel}
\ker \Psi \cap \DPic\big(\cD^b(A)\big) \cong \mathbb{Z} \times \big(\HH^1(\Sigma, \Bbbk^{\times}) \ltimes \mathbb{G}_a^{b}\big),
\end{equation}
\noindent where $\mathbb{Z}$ acts through the shift functor, $\mathbb{G}_a=(\Bbbk, +)$ denotes the additive group and $b$ denotes the number of boundary components $B$ on $\Sigma_A$ with vanishing winding number such that $B \cap \cM \cong [0,1]$ topologically. The subgroup $\HH^1(\Sigma, \Bbbk^{\times}) \ltimes \mathbb{G}_a^{b}$ in \eqref{EquationKernel} is realized as an explicit subgroup $\cK \subseteq \Out_{\circ}(A)$. The subgroup $\HH^1(\Sigma, \Bbbk^{\times})$ acts through its identification with local systems which are used to multiply arrows in the underlying quiver of $A$ by scalars. The automorphisms of $A$ which correspond to a copy of $\mathbb{G}_a$ have the form of an exponential $\exp(\lambda f)=\operatorname{Id}_A + \lambda f$, where $\lambda \in \Bbbk$ and $f \in \End_{\Bbbk}(A)$ is an nilpotent endomorphism with $f^2=0$.

The structure of $\cK$ is closely connected to the structure of $\HHH^1(A,A)$ (of which the author was unaware at the time of writing \cite{OpperDerivedEquivalences}). Ignoring certain exceptions\footnote{The exceptions are graded gentle algebras $A$ such that $\Dfd{A}\simeq \cD^b(\mathbb{A}^1_{\Bbbk})$ or  $\Dfd{A}\simeq \cD^b(\mathbb{P}^1_{\Bbbk})$. Equivalently, these are all graded gentle algebras whose associated surface is a cylinder with $\cM \cong S^1 \sqcup [0,1]$ or $\cM \cong [0,1] \sqcup [0,1]$ and whose boundary components have vanishing winding number.} with known derived Picard groups, one deduces that there is an isomorphism
\begin{displaymath}
\HHH^1(A,A) \cong  \HH_1(\Sigma, \Bbbk) \ltimes \Bbbk^b
\end{displaymath}
\noindent where $\HH_1(\Sigma, \Bbbk)$ and $\Bbbk$ are considered as abelian Lie algebras. Because $A$ is concentrated in degree $0$, one has $\HHH^1_+(A,A)=0$ and the exponential gives no further information. The structure of $\HHH^1(A,A)$ remains largely unchanged in the graded case and the elements of $\cK$ admit natural generalisations to (not necessarily strict) $A_\infty$-functors yielding an analogous subgroup of $\DPic(\Dfd{A})$ in this case, even in positive characteristic. A new phenomenon in the graded case are fully marked boundary components with winding number $0$, something which is impossible in the ungraded case. For example, over a field of characteristic zero, the contribution to $\HHH^1_+(A,A)$ of a fully marked component with vanishing winding number is a copy of the positive part of the completed \textit{Witt algebra} $W$. This is the Lie algebra $W=\prod_{n > 0}\Bbbk e_i$ with the unique continuous Lie bracket determined by the relation 
\begin{displaymath}
[e_m, e_n]= (m-n) e_{m+n}.
\end{displaymath}
\noindent Over $\Bbbk=\mathbb{C}$, its uncompleted counterpart  $\oplus_{n > 0}\Bbbk e_i$ is isomorphic to a Lie algebra of holomorphic vector fields on the complex plane with $e_m$ corresponding to the vector field $z \mapsto -z^{m+1}\frac{\partial}{\partial z}$. Ignoring technical details and the same exceptions as before, \Cref{IntroTheoremA} can be used to show that each copy of $W$ contributes a subgroup to $\DPic(\Dfd{A})$. By extending the homomorphism $\Phi$ to the graded setting and keeping track of shifts this leads to a description of the derived Picard group of graded gentle algebras which are homologically smooth or proper. For example, for any proper graded gentle algebra $A$, one obtains
\begin{displaymath}
	\DPic\big(\Dfd{A}\big) \cong \Aut^{\infty}_{\circ}(A) \rtimes \MCG_{\operatorname{gr}}(\Sigma_A, \cM_A, \eta_A),
\end{displaymath}
\noindent where $\Aut^{\infty}_{\circ}(A)$ denotes the ``identity component'' from \Cref{Conjecture} and $\MCG_{\operatorname{gr}}(\Sigma_A, \cM_A, \eta_A)$ denotes the \textit{graded mapping class group}: a central extension of  $\MCG(\Sigma_A, \cM_A, \eta_A)$ which incorporates a notion of shift. Thanks to the exponential map from \Cref{IntroTheoremA}, $\Aut^{\infty}_{\circ}(A)$ admits an explicit description in terms of the Lie algebra $\HHH^1(A,A)$ over a field of characteristic $0$ and for fields with positive odd characteristic if $\partial \Sigma_A$ contains no fully marked components.
\bibliography{Bibliography}{}

@article
{AmiotPlamondonSchroll,
   author={Amiot, C. and Plamondon, {P-G}. and Schroll, S.},
   title={A complete derived invariant for gentle algebras via
winding numbers and {A}rf invariants},
   journal={arXiv:1904.02555 [math.RT]},
   year=2019
}

@article
{AssemSkowronski,
   author={I.~Assem and A.~Skowro\'{n}ski},
   title={ Iterated tilted algebras of type {$\tilde{A}_n$}},
   journal={Mathematische Zeitschrift},
   volume={195},
   year=1987,
   pages={269--290}
}

@article {HaidenKatzarkovKontsevich,
    AUTHOR = {Haiden, F. and Katzarkov, L. and Kontsevich, M.},
     TITLE = {Flat surfaces and stability structures},
   JOURNAL = {Publ. Math. Inst. Hautes \'{E}tudes Sci.},
  FJOURNAL = {Publications Math\'{e}matiques. Institut de Hautes \'{E}tudes
              Scientifiques},
    VOLUME = {126},
      YEAR = {2017},
     PAGES = {247--318},
      ISSN = {0073-8301,1618-1913},
   MRCLASS = {14D23 (14F05 18E30 32Q26 53D37 81T30)},
  MRNUMBER = {3735868},
MRREVIEWER = {Mee\ Seong\ Im},
       DOI = {10.1007/s10240-017-0095-y},
       URL = {https://doi.org/10.1007/s10240-017-0095-y},
}

@article
{LekiliPolishchukGentle,
   author={Y.~Lekili and A.~Polishchuk},
   title={ Derived equivalences of gentle algebras via {F}ukaya categories},
 journal = {Mathematische Annalen},
 volume={376},
 pages={187--225},
   year={2020}
}

@article
{OpperPlamondonSchroll,
   author={S.~Opper and {P-G}.~Plamondon and S.~Schroll},
   title={ A geometric model for the derived category of gentle algebras},
 journal = {arXiv:1801.09659 [math.RT]},
   year={2018}
}

@article
{OpperDerivedEquivalences,
   author={S.~Opper},
   title={On auto-equivalences and complete derived invariants of gentle algebras},
 journal = {arXiv:1904.04859 [math.RT]},
   year={2019}
}

@article {Yekutieli,
    AUTHOR = {Yekutieli, A.},
     TITLE = {The derived {P}icard group is a locally algebraic group},
   JOURNAL = {Algebr. Represent. Theory},
  FJOURNAL = {Algebras and Representation Theory},
    VOLUME = {7},
      YEAR = {2004},
    NUMBER = {1},
     PAGES = {53--57},
      ISSN = {1386-923X},
   MRCLASS = {16D90 (18E30 20G15)},
  MRNUMBER = {2046954},
MRREVIEWER = {Rapha\"{e}l Rouquier},
       DOI = {10.1023/B:ALGE.0000019383.78214.31},
       URL = {https://doi.org/10.1023/B:ALGE.0000019383.78214.31},
}

@article {CanonacoOrnaghiStellari,
    AUTHOR = {Canonaco, A. and Ornaghi, M. and Stellari, P.},
     TITLE = {Localizations of the category of {$A_\infty$} categories and
              internal {H}oms},
   JOURNAL = {Doc. Math.},
  FJOURNAL = {Documenta Mathematica},
    VOLUME = {24},
      YEAR = {2019},
     PAGES = {2463--2492},
      ISSN = {1431-0635,1431-0643},
   MRCLASS = {18G80 (18D20 18G35 18G70 57T30)},
  MRNUMBER = {4061058},
MRREVIEWER = {J\'{e}r\^{o}me\ Scherer},
}

@article {ToenDerivedMoritaTheory,
    AUTHOR = {To{\"{e}}n, B.},
     TITLE = {The homotopy theory of {$dg$}-categories and derived {M}orita
              theory},
   JOURNAL = {Invent. Math.},
  FJOURNAL = {Inventiones Mathematicae},
    VOLUME = {167},
      YEAR = {2007},
    NUMBER = {3},
     PAGES = {615--667},
      ISSN = {0020-9910,1432-1297},
   MRCLASS = {18D05 (18E30 18G55 19D55)},
  MRNUMBER = {2276263},
MRREVIEWER = {Mark\ Hovey},
       DOI = {10.1007/s00222-006-0025-y},
       URL = {https://doi.org/10.1007/s00222-006-0025-y},
}

@inproceedings
{GetzlerCartanHomotopy,
   author={E.~Getzler},
   title={Cartan homotopy formulas and the {Gauss}-{Manin} connection in cyclic homology},
   booktitle={Quantum deformations of algebras and their representations (Ramat-Gan, 1991/92; Rehovot, 1991/92), Israel Conference Proceedings, vol. 7, Bar-Ilan University, Ramat-Gan},
   year=1993,
   pages={65--78}
}

@article {LadaMarkl,
    AUTHOR = {T.~Lada and M.~Markl},
     TITLE = {Symmetric Brace Algebras},
   JOURNAL = {Appl.Categor.Struct.},
    VOLUME = {13},
      YEAR = {2005},
     PAGES = {351--370}
}

@article {OudomGuin,
    AUTHOR = {J.~Oudom and D.~Guin},
     TITLE = {On the {Lie} envelopping algebra of a pre-{Lie} algebra},
   JOURNAL = {J.K-Theory},
FJOURNAL={Journal of
K-theory: K-theory and its Applications to Algebra, Geometry, and Topology},
    VOLUME = {2},
      YEAR = {2008},
     PAGES = {147--167}
}

@article{GerstenhaberVoronov,
    author = {Gerstenhaber, M. and Voronov, A.},
    title = "{Homotopy {G}-algebras and moduli space operad}",
    Fjournal = {International Mathematics Research Notices},
journal = {Int. Math. Res. Not.},
    volume = {1995},
    number = {3},
    pages = {141-153},
    year = {1995}
}

@Article{CanonacoStellariTourExistenceUniqueness,
 Author = {Canonaco, A. and Stellari, P.},
 Title = {A tour about existence and uniqueness of dg enhancements and lifts},
 FJournal = {Journal of Geometry and Physics},
 Journal = {J. Geom. Phys.},
 ISSN = {0393-0440},
 Volume = {122},
 Pages = {28--52},
 Year = {2017},
 Language = {English},
 DOI = {10.1016/j.geomphys.2016.11.030},
 zbMATH = {6803446},
 Zbl = {1399.14007}
}

@Article{GenoveseUniquenessLifting,
 Author = {Genovese, F.},
 Title = {The uniqueness problem of dg-lifts and {Fourier}-{Mukai} kernels},
 FJournal = {Journal of the London Mathematical Society. Second Series},
 Journal = {J. Lond. Math. Soc., II. Ser.},
 Volume = {94},
 Number = {2},
 Pages = {617--638},
 Year = {2016}
}

@Article{CohenJones,
 Author = {Cohen, R. and Jones, J.},
 Title = {A homotopy theoretic realization of string topology},
 FJournal = {Mathematische Annalen},
 Journal = {Math. Ann.},
 ISSN = {0025-5831},
 Volume = {324},
 Number = {4},
 Pages = {773--798},
 Year = {2002},
 Language = {English},
 DOI = {10.1007/s00208-002-0362-0},
 zbMATH = {1889186},
 Zbl = {1025.55005}
}

@InCollection{KellerAInfinityAlgebrasInRepresentationTheory,
 Author = {Keller, B.},
 Title = {{{\(A\)}}-infinity algebras in representation theory},
 BookTitle = {Representations of algebras. Vol. I, II. Proceedings of the 9th international conference, Beijing, China, August 21--September 1, 2000.},
 ISBN = {7-303-06113-4},
 Pages = {74--86},
 Year = {2002},
 Publisher = {Beijing: Beijing Normal University Press},
 Language = {English},
 zbMATH = {5015755},
 Zbl = {1086.18007}
}

@Article{GetzlerJones,
 Author = {Getzler, E. and Jones, J.},
 Title = {A{{\({}_{\infty}\)}}-algebras and the cyclic bar complex},
 FJournal = {Illinois Journal of Mathematics},
 Journal = {Ill. J. Math.},
 ISSN = {0019-2082},
 Volume = {34},
 Number = {2},
 Pages = {256--283},
 Year = {1990},
 Language = {English},
 zbMATH = {4149147},
 Zbl = {0701.55009}
}

@Book{SeidelBook,
 Author = {Seidel, P.},
 Title = {Fukaya categories and {Picard}-{Lefschetz} theory},
 FSeries = {Zurich Lectures in Advanced Mathematics},
 Series = {Zur. Lect. Adv. Math.},
 ISBN = {978-3-03719-063-0},
 Year = {2008},
 Publisher = {Z{\"u}rich: European Mathematical Society (EMS)},
 Language = {English},
 DOI = {10.4171/063},
 zbMATH = {5294596},
 Zbl = {1159.53001}
}

@Book{BespalovLyubashenkoManzyuk,
 Author = {Bespalov, Y. and Lyubashenko, V. and Manzyuk, O.},
 Title = {Pretriangulated {{\(A_\infty\)}}-categories},
 ISBN = {978-966-02-4861-8},
 Year = {2008},
 Publisher = {Ky{\"{\i}}v: Instytut Matematyky NAN Ukra{\"{\i}}ny},
 Language = {English},
 zbMATH = {5520225},
 Zbl = {1199.18001}
}

@InCollection{KellerFunctorCategories,
 Author = {Keller, B.},
 Title = {{{\(A\)}}-infinity algebras, modules and functor categories},
 BookTitle = {Trends in representation theory of algebras and related topics. Workshop on representations of algebras and related topics, Quer\'etaro, M\'exico, August 11--14, 2004.},
 ISBN = {0-8218-3818-0},
 Pages = {67--93},
 Year = {2006},
 Publisher = {Providence, RI: American Mathematical Society (AMS)},
 Language = {English},
 zbMATH = {5083069},
 Zbl = {1121.18008}
}

@Article{KellerDerivedPicardGroup,
 Author = {Keller, B.},
 Title = {Hochschild cohomology and derived {Picard} groups.},
 FJournal = {Journal of Pure and Applied Algebra},
 Journal = {J. Pure Appl. Algebra},
 ISSN = {0022-4049},
 Volume = {190},
 Number = {1-3},
 Pages = {177--196},
 Year = {2004},
 Language = {English},
 DOI = {10.1016/j.jpaa.2003.10.030},
 zbMATH = {2093455},
 Zbl = {1060.16010}
}

@Article{DotsenkoShadrinVallettePreLie,
 Author = {Dotsenko, V. and Shadrin, S. and Vallette, B.},
 Title = {Pre-{Lie} deformation theory},
 FJournal = {Moscow Mathematical Journal},
 Journal = {Mosc. Math. J.},
 ISSN = {1609-3321},
 Volume = {16},
 Number = {3},
 Pages = {505--543},
 Year = {2016},
 Language = {English},
 URL = {www.mathjournals.org/mmj/2016-016-003/2016-016-003-003.html},
 zbMATH = {6783281},
 Zbl = {1386.18054}
}

@Article{Ebrahimi-FardManchon,
 Author = {Ebrahimi-Fard, K. and Manchon, D.},
 Title = {A {Magnus}- and {Fer}-type formula in dendriform algebras.},
 FJournal = {Foundations of Computational Mathematics},
 Journal = {Found. Comput. Math.},
 ISSN = {1615-3375},
 Volume = {9},
 Number = {3},
 Pages = {295--316},
 Year = {2009},
 Language = {English},
 DOI = {10.1007/s10208-008-9023-3},
 zbMATH = {5577220},
 Zbl = {1173.17002}
}

@article{BriggsGelinas,
      title={The {A}-infinity centre of the {Y}oneda algebra and the characteristic action of {H}ochschild cohomology on the derived category}, 
      author={Briggs, B. and Gelinas, V.},
      year={2017},
      eprint={1702.00721},
      archivePrefix={arXiv},
      primaryClass={math.RT},
journal={arXiv:1702.00721 [math.RT]}
}

@article {ChuangHolsteinLazarev,
    AUTHOR = {Chuang, J. and Holstein, J. and Lazarev, A.},
     TITLE = {Maurer-{C}artan moduli and theorems of {R}iemann-{H}ilbert
              type},
   JOURNAL = {Appl. Categ. Structures},
  FJOURNAL = {Applied Categorical Structures. A Journal Devoted to
              Applications of Categorical Methods in Algebra, Analysis,
              Order, Topology and Computer Science},
    VOLUME = {29},
      YEAR = {2021},
    NUMBER = {4},
     PAGES = {685--728},
      ISSN = {0927-2852,1572-9095},
   MRCLASS = {18F99 (18G80 55U35)},
  MRNUMBER = {4282047},
MRREVIEWER = {Maosong\ Xiang},
       DOI = {10.1007/s10485-021-09631-3},
       URL = {https://doi.org/10.1007/s10485-021-09631-3}
}

@article {BuijsFelixManrilloTanreQuillenFunctor,
    AUTHOR = {Buijs, U. and F\'{e}lix, Y. and Murillo, A. and
              Tanr\'{e}, D.},
     TITLE = {Lie models of simplicial sets and representability of the
              {Q}uillen functor},
   JOURNAL = {Israel J. Math.},
  FJOURNAL = {Israel Journal of Mathematics},
    VOLUME = {238},
      YEAR = {2020},
    NUMBER = {1},
     PAGES = {313--358},
      ISSN = {0021-2172,1565-8511},
   MRCLASS = {55P62 (17B70 55U10)},
  MRNUMBER = {4145802},
MRREVIEWER = {Antonio\ J.\ Garv\'{\i}n},
       DOI = {10.1007/s11856-020-2026-8},
       URL = {https://doi.org/10.1007/s11856-020-2026-8}
}

@article{DotsenkoPoncinThreeHomotopies,
    AUTHOR = {Dotsenko, V. and Poncin, N.},
     TITLE = {A tale of three homotopies},
   JOURNAL = {Appl. Categ. Structures},
  FJOURNAL = {Applied Categorical Structures. A Journal Devoted to
              Applications of Categorical Methods in Algebra, Analysis,
              Order, Topology and Computer Science},
    VOLUME = {24},
      YEAR = {2016},
    NUMBER = {6},
     PAGES = {845--873},
      ISSN = {0927-2852,1572-9095},
   MRCLASS = {18G55 (18D50)},
  MRNUMBER = {3572456},
MRREVIEWER = {Christopher\ L.\ Rogers},
       DOI = {10.1007/s10485-015-9407-x},
       URL = {https://doi.org/10.1007/s10485-015-9407-x}
}

@book{ManettiBook,
    AUTHOR = {Manetti, M.},
     TITLE = {Lie methods in deformation theory},
    SERIES = {Springer Monographs in Mathematics},
 PUBLISHER = {Springer, Singapore},
      YEAR = {2022},
     PAGES = {xii+574},
      ISBN = {978-981-19-1184-2; 978-981-19-1185-9},
   MRCLASS = {32Gxx (14D06 17Bxx 18G80 53D17)},
  MRNUMBER = {4485797},
MRREVIEWER = {Andrei\ D.\ Halanay},
       DOI = {10.1007/978-981-19-1185-9},
       URL = {https://doi.org/10.1007/978-981-19-1185-9}
}

@article {GerstenhaberCohomologyStructure,
    AUTHOR = {Gerstenhaber, M.},
     TITLE = {The cohomology structure of an associative ring},
   JOURNAL = {Ann. of Math. (2)},
  FJOURNAL = {Annals of Mathematics. Second Series},
    VOLUME = {78},
      YEAR = {1963},
     PAGES = {267--288},
      ISSN = {0003-486X},
   MRCLASS = {16.90 (18.00)},
  MRNUMBER = {161898},
MRREVIEWER = {M.\ Auslander},
       DOI = {10.2307/1970343},
       URL = {https://doi.org/10.2307/1970343},
}

@article{Robert-NicoudValletteHigherLieTheory,
      title={Higher Lie theory}, 
      author={Robert-Nicoud, D. and Vallette, B.},
      year={2020},
journal={arXiv:2010.10485 [math.AT]},
      eprint={2010.10485},
      archivePrefix={arXiv},
      primaryClass={math.AT}
}

@misc{KellerInvarianceHigherStructures,
      title={Derived invariance of higher structures on the {H}ochschild complex}, 
      author={Keller, B.},
      year={2003},
      note={preprint, available at the author’s home page},
}

@article {AbouzaidGenerationFukayaCategory,
    AUTHOR = {Abouzaid, M.},
     TITLE = {A cotangent fibre generates the {F}ukaya category},
   JOURNAL = {Adv. Math.},
  FJOURNAL = {Advances in Mathematics},
    VOLUME = {228},
      YEAR = {2011},
    NUMBER = {2},
     PAGES = {894--939},
      ISSN = {0001-8708,1090-2082},
   MRCLASS = {53D37},
  MRNUMBER = {2822213},
MRREVIEWER = {Michael\ J.\ Usher},
       DOI = {10.1016/j.aim.2011.06.007},
       URL = {https://doi.org/10.1016/j.aim.2011.06.007}
}

@article {EtguLekiliKoszulDualityPatterns,
    AUTHOR = {Etg\"{u}, T. and Lekili, Y.},
     TITLE = {Koszul duality patterns in {F}loer theory},
   JOURNAL = {Geom. Topol.},
  FJOURNAL = {Geometry \& Topology},
    VOLUME = {21},
      YEAR = {2017},
    NUMBER = {6},
     PAGES = {3313--3389},
      ISSN = {1465-3060,1364-0380},
   MRCLASS = {57R58 (16E45)},
  MRNUMBER = {3692968},
MRREVIEWER = {Matthew\ Stoffregen},
       DOI = {10.2140/gt.2017.21.3313},
       URL = {https://doi.org/10.2140/gt.2017.21.3313},
}

@article {AbouzaidFukayaCategoryPlumbings,
    AUTHOR = {Abouzaid, M.},
     TITLE = {A topological model for the {F}ukaya categories of plumbings},
   JOURNAL = {J. Differential Geom.},
  FJOURNAL = {Journal of Differential Geometry},
    VOLUME = {87},
      YEAR = {2011},
    NUMBER = {1},
     PAGES = {1--80},
      ISSN = {0022-040X,1945-743X},
   MRCLASS = {53D37 (53D40)},
  MRNUMBER = {2786590},
MRREVIEWER = {Timothy\ Perutz},
       URL = {http://projecteuclid.org/euclid.jdg/1303219772},
}

@article {DeligneGriffithsMorganSullivan,
    AUTHOR = {Deligne, P. and Griffiths, P. and Morgan, J. and
              Sullivan, D.},
     TITLE = {Real homotopy theory of {K}\"{a}hler manifolds},
   JOURNAL = {Invent. Math.},
  FJOURNAL = {Inventiones Mathematicae},
    VOLUME = {29},
      YEAR = {1975},
    NUMBER = {3},
     PAGES = {245--274},
      ISSN = {0020-9910,1432-1297},
   MRCLASS = {32C10 (57D15)},
  MRNUMBER = {382702},
MRREVIEWER = {M.\ F.\ Atiyah},
       DOI = {10.1007/BF01389853},
       URL = {https://doi.org/10.1007/BF01389853},
}

@article {KrauseYe,
    AUTHOR = {Krause, H. and Ye, Y.},
     TITLE = {On the centre of a triangulated category},
   JOURNAL = {Proc. Edinb. Math. Soc. (2)},
  FJOURNAL = {Proceedings of the Edinburgh Mathematical Society. Series II},
    VOLUME = {54},
      YEAR = {2011},
    NUMBER = {2},
     PAGES = {443--466},
      ISSN = {0013-0915,1464-3839},
   MRCLASS = {18E30 (16E35)},
  MRNUMBER = {2794666},
MRREVIEWER = {Du\v{s}ko\ Bogdani\'{c}},
       DOI = {10.1017/S0013091509001199},
       URL = {https://doi.org/10.1017/S0013091509001199},
}

@article {GanatraPardonShendeSectorialDescent,
    AUTHOR = {Ganatra, S. and Pardon, J. and Shende, V.},
     TITLE = {Sectorial descent for wrapped {F}ukaya categories},
   JOURNAL = {J. Amer. Math. Soc.},
  FJOURNAL = {Journal of the American Mathematical Society},
    VOLUME = {37},
      YEAR = {2024},
    NUMBER = {2},
     PAGES = {499--635},
      ISSN = {0894-0347,1088-6834},
   MRCLASS = {53D37 (53D40 57R17)},
  MRNUMBER = {4695507},
       DOI = {10.1090/jams/1035},
       URL = {https://doi.org/10.1090/jams/1035},
}

@article{GanatraPardonShendeMicrolocalMorseTheory,
      title={Microlocal {M}orse theory of wrapped {F}ukaya categories}, 
      author={Ganatra, S. and Pardon, J. and Shende, V.},
      year={2018},
      journal={arXiv:1809.08807 [math.SG]},
      archivePrefix={arXiv},
      primaryClass={math.SG}
}

@article {GanatraPardonShendeCovariantlyFunctorialWrappedFukayaCategory,
    AUTHOR = {Ganatra, Sheel and Pardon, John and Shende, Vivek},
     TITLE = {Covariantly functorial wrapped {F}loer theory on {L}iouville
              sectors},
   JOURNAL = {Publ. Math. Inst. Hautes \'{E}tudes Sci.},
  FJOURNAL = {Publications Math\'{e}matiques. Institut de Hautes \'{E}tudes
              Scientifiques},
    VOLUME = {131},
      YEAR = {2020},
     PAGES = {73--200},
      ISSN = {0073-8301,1618-1913},
   MRCLASS = {53D40},
  MRNUMBER = {4106794},
MRREVIEWER = {Alexander\ Fel\cprime shtyn},
       DOI = {10.1007/s10240-019-00112-x},
       URL = {https://doi.org/10.1007/s10240-019-00112-x},
}

@article {AbouzaidSeidelOpenStringAnalogue,
    AUTHOR = {Abouzaid, M. and Seidel, P.},
     TITLE = {An open string analogue of {V}iterbo functoriality},
   JOURNAL = {Geom. Topol.},
  FJOURNAL = {Geometry \& Topology},
    VOLUME = {14},
      YEAR = {2010},
    NUMBER = {2},
     PAGES = {627--718},
      ISSN = {1465-3060,1364-0380},
   MRCLASS = {53D40 (53D12 53D37)},
  MRNUMBER = {2602848},
MRREVIEWER = {Timothy\ Perutz},
       DOI = {10.2140/gt.2010.14.627},
       URL = {https://doi.org/10.2140/gt.2010.14.627},
}

@article {CrawleyBoeveyEtingofGinzburg,
    AUTHOR = {Crawley-Boevey, W. and Etingof, P. and Ginzburg,
              V.},
     TITLE = {Noncommutative geometry and quiver algebras},
   JOURNAL = {Adv. Math.},
  FJOURNAL = {Advances in Mathematics},
    VOLUME = {209},
      YEAR = {2007},
    NUMBER = {1},
     PAGES = {274--336},
      ISSN = {0001-8708,1090-2082},
   MRCLASS = {14A22 (16G20)},
  MRNUMBER = {2294224},
MRREVIEWER = {Jon\ Eivind\ Vatne},
       DOI = {10.1016/j.aim.2006.05.004},
       URL = {https://doi.org/10.1016/j.aim.2006.05.004},
}

@article {LekiliUeda,
    AUTHOR = {Lekili, Y. and Ueda, K.},
     TITLE = {Homological mirror symmetry for {M}ilnor fibers of simple
              singularities},
   JOURNAL = {Algebr. Geom.},
  FJOURNAL = {Algebraic Geometry},
    VOLUME = {8},
      YEAR = {2021},
    NUMBER = {5},
     PAGES = {562--586},
      ISSN = {2313-1691,2214-2584},
   MRCLASS = {53D37 (14F08 14J33)},
  MRNUMBER = {4371540},
MRREVIEWER = {David\ E.\ Hurtubise},
       DOI = {10.14231/ag-2021-017},
       URL = {https://doi.org/10.14231/ag-2021-017},
}

@article {MiyachiYekutieli,
    AUTHOR = {Miyachi, J. and Yekutieli, A.},
     TITLE = {Derived {P}icard groups of finite-dimensional hereditary
              algebras},
   JOURNAL = {Compositio Math.},
  FJOURNAL = {Compositio Mathematica},
    VOLUME = {129},
      YEAR = {2001},
    NUMBER = {3},
     PAGES = {341--368},
      ISSN = {0010-437X,1570-5846},
   MRCLASS = {18E30 (16G20 16G60)},
  MRNUMBER = {1868359},
MRREVIEWER = {Michael\ S.\ Barot},
       DOI = {10.1023/A:1012579131516},
       URL = {https://doi.org/10.1023/A:1012579131516},
}

@article {EkholmLekiliDuality,
    AUTHOR = {Ekholm, T. and Lekili, Y. },
     TITLE = {Duality between {L}agrangian and {L}egendrian invariants},
   JOURNAL = {Geom. Topol.},
  FJOURNAL = {Geometry \& Topology},
    VOLUME = {27},
      YEAR = {2023},
    NUMBER = {6},
     PAGES = {2049--2179},
      ISSN = {1465-3060,1364-0380},
   MRCLASS = {57R17},
  MRNUMBER = {4634745},
       DOI = {10.2140/gt.2023.27.2049},
       URL = {https://doi.org/10.2140/gt.2023.27.2049},
}

@article{ChantraineDimitroglouRizellGhigginiGolovko,
      title={Geometric generation of the wrapped {F}ukaya category of {W}einstein manifolds and sectors}, 
      author={Chantraine, B. and Dimitroglou Rizell, G. and Ghiggini, P. and Golovko, R.},
      year={2024},
     journal = {Ann. Sci. Éc. Norm. Supér.},
Fjournal={Annales Scientifiques de l'École Normale Supérieure, Serie 4},
volume ={57},
number ={1},
pages ={1--85},
}

@article {IyamaYangSiltingReduction,
    AUTHOR = {Iyama, O. and Yang, D.},
     TITLE = {Silting reduction and {C}alabi-{Y}au reduction of triangulated
              categories},
   JOURNAL = {Trans. Amer. Math. Soc.},
  FJOURNAL = {Transactions of the American Mathematical Society},
    VOLUME = {370},
      YEAR = {2018},
    NUMBER = {11},
     PAGES = {7861--7898},
      ISSN = {0002-9947,1088-6850},
   MRCLASS = {16E35 (13F60 16G99 18E30)},
  MRNUMBER = {3852451},
MRREVIEWER = {Silvana\ Bazzoni},
       DOI = {10.1090/tran/7213},
       URL = {https://doi.org/10.1090/tran/7213},
}

@article {AbouzaidWrappedFukayaCategoryBasedLoops,
    AUTHOR = {Abouzaid, M.},
     TITLE = {On the wrapped {F}ukaya category and based loops},
   JOURNAL = {J. Symplectic Geom.},
  FJOURNAL = {The Journal of Symplectic Geometry},
    VOLUME = {10},
      YEAR = {2012},
    NUMBER = {1},
     PAGES = {27--79},
      ISSN = {1527-5256,1540-2347},
   MRCLASS = {53D37 (53D12)},
  MRNUMBER = {2904032},
MRREVIEWER = {Janko\ Latschev},
       URL = {http://projecteuclid.org/euclid.jsg/1332853049},
}

@article{VerstraeteDividedPowers,
      title={Pre-{L}ie algebras with divided powers and the {D}eligne groupoid in positive characteristic}, 
      author={Verstraete, M.},
      year={2023},
      eprint={2310.20300},
      archivePrefix={arXiv},
      primaryClass={math.AT},
  journal={arXiv:2310.20300 [math.AT]},   
}

@phdthesis{GuanPhDThesis,
  title        = {Koszul Duality and Deformation Theory},
  author       = {Guan, A.},
  year         = 2021,
  month        = {July},
  address      = {Lancaster, UK},
  note         = {Available at \url{https://www.research.lancs.ac.uk/portal/en/publications/koszul-duality-and-deformation-theory(1e5a76e9-d0a6-42b4-947f-eccff8af3d93).html}},
  school = {Lancaster University},
  type         = {PhD thesis},
doi ={10.17635/lancaster/thesis/1362},
publisher = {Lancaster University},
}

@phdthesis{GanatraPhDThesis,
  title        = {Symplectic cohomology and duality for the wrapped {Fukaya} Category},
  author       = {Ganatra, S.},
  year         = {2012},
  month        = {June},
  address      = {Massachusetts, U.S.},
  note         = {Available at \url{http://hdl.handle.net/1721.1/73362}},
  school = {Massachusetts Institute of Technology},
  type         = {PhD thesis},
publisher = {Massachusetts Institute of Technology},
}

@phdthesis{RobertNicoudPhDThesis,
  title        = {Operads and {M}aurer-{C}artan spaces},
  author       = {Robert-Nicoud, D.},
  year         = 2018,
  month        = {June},
  address      = {Paris, France},
  note         = {Available at \url{https://arxiv.org/abs/1807.02129}},
  school = {Universit\'{e} Paris 13},
  type         = {PhD thesis},
publisher = {Universit\'{e} Paris 13},
}

@article {OrlovNoncommutativeSchemes,
    AUTHOR = {Orlov, D.},
     TITLE = {Smooth and proper noncommutative schemes and gluing of {DG}
              categories},
   JOURNAL = {Adv. Math.},
  FJOURNAL = {Advances in Mathematics},
    VOLUME = {302},
      YEAR = {2016},
     PAGES = {59--105},
      ISSN = {0001-8708,1090-2082},
   MRCLASS = {14F05 (16E45 18E30)},
  MRNUMBER = {3545926},
MRREVIEWER = {Shintarou\ Yanagida},
       DOI = {10.1016/j.aim.2016.07.014},
       URL = {https://doi.org/10.1016/j.aim.2016.07.014},
}

@incollection {FukayaSeidelSmithCategoricalViewpoint,
    AUTHOR = {Fukaya, K. and Seidel, P. and Smith, I.},
     TITLE = {The symplectic geometry of cotangent bundles from a
              categorical viewpoint},
 BOOKTITLE = {Homological mirror symmetry},
    SERIES = {Lecture Notes in Phys.},
    VOLUME = {757},
     PAGES = {1--26},
 PUBLISHER = {Springer, Berlin},
      YEAR = {2009},
      ISBN = {978-3-540-86374-8},
   MRCLASS = {53D37 (18G40 53D35 53D40 57R17)},
  MRNUMBER = {2596633},
MRREVIEWER = {Timothy\ Perutz},
       DOI = {10.1007/978-3-540-68030-7\_1},
       URL = {https://doi.org/10.1007/978-3-540-68030-7_1},
}

@article{BaeJeongKimCluster,
      title={Cluster categories from {F}ukaya categories}, 
      author={Bae, H. and Jeong, W. and Kim, J.},
      year={2023},
      eprint={2209.09442},
      archivePrefix={arXiv},
      primaryClass={math.SG},
	journal={arXiv:2209.09442 [math.SG]}
}

@article{ChasSullivanStringTopology,
      title={String Topology}, 
      author={Chas, M. and Sullivan, D.},
      year={1999},
      eprint={math/9911159},
      archivePrefix={arXiv},
      primaryClass={math.GT},
journal={arXiv:math/9911159 [math.GT]}
}

@article {KoenigYangSilting,
    AUTHOR = {Koenig, S. and Yang, D.},
     TITLE = {Silting objects, simple-minded collections, {$t$}-structures
              and co-{$t$}-structures for finite-dimensional algebras},
   JOURNAL = {Doc. Math.},
  FJOURNAL = {Documenta Mathematica},
    VOLUME = {19},
      YEAR = {2014},
     PAGES = {403--438},
      ISSN = {1431-0635,1431-0643},
   MRCLASS = {16E35 (18E30)},
  MRNUMBER = {3178243},
MRREVIEWER = {Octavio\ Mendoza Hern\'{a}ndez},
}

@article{KarabasLee,
      title={Homotopy Colimits of DG Categories and {F}ukaya Categories}, 
      author={Karabas, D. and Lee, S.},
      year={2022},
      eprint={2109.03411},
      archivePrefix={arXiv},
      primaryClass={math.SG},
journal={arXiv:2109.03411 [math.SG]}
}

@article {MuroDwyerKanModelStructure,
    AUTHOR = {Muro, F.},
     TITLE = {Dwyer-{K}an homotopy theory of enriched categories},
   JOURNAL = {J. Topol.},
  FJOURNAL = {Journal of Topology},
    VOLUME = {8},
      YEAR = {2015},
    NUMBER = {2},
     PAGES = {377--413},
      ISSN = {1753-8416,1753-8424},
   MRCLASS = {18D20 (18G55 55U35)},
  MRNUMBER = {3356766},
MRREVIEWER = {Philippe\ Gaucher},
       DOI = {10.1112/jtopol/jtu029},
       URL = {https://doi.org/10.1112/jtopol/jtu029},
}

@book {LodayValletteAlgebraicOperads,
    AUTHOR = {Loday, J. and Vallette, B.},
     TITLE = {Algebraic operads},
    SERIES = {Grundlehren der mathematischen Wissenschaften [Fundamental
              Principles of Mathematical Sciences]},
    VOLUME = {346},
 PUBLISHER = {Springer, Heidelberg},
      YEAR = {2012},
     PAGES = {xxiv+634},
      ISBN = {978-3-642-30361-6},
   MRCLASS = {18D50 (16E99)},
  MRNUMBER = {2954392},
MRREVIEWER = {Andrey\ Yu.\ Lazarev},
       DOI = {10.1007/978-3-642-30362-3},
       URL = {https://doi.org/10.1007/978-3-642-30362-3},
}

@article {HuerfanoKhovanov,
    AUTHOR = {Huerfano, R. and Khovanov, M.},
     TITLE = {A category for the adjoint representation},
   JOURNAL = {J. Algebra},
  FJOURNAL = {Journal of Algebra},
    VOLUME = {246},
      YEAR = {2001},
    NUMBER = {2},
     PAGES = {514--542},
      ISSN = {0021-8693,1090-266X},
   MRCLASS = {17B37 (16G10)},
  MRNUMBER = {1872113},
MRREVIEWER = {Iain\ G.\ Gordon},
       DOI = {10.1006/jabr.2001.8962},
       URL = {https://doi.org/10.1006/jabr.2001.8962},
}

@article{JinSchrollWang,
      title={A complete derived invariant and silting theory for graded gentle algebras}, 
      author={Jin, H. and Schroll, S. and Wang, Z.},
      year={2023},
      eprint={2303.17474},
journal={arXiv:2303.17474 [math.RT]},
      archivePrefix={arXiv},
      primaryClass={math.RT}
}

@article {FelixThomas,
    AUTHOR = {Fel\'{i}x, Y. and Thomas, J.},
     TITLE = {Rational {BV}-algebra in string topology.},
   JOURNAL = {Bull. Soc. Math. France},
    VOLUME = {136},
      YEAR = {2008},
    NUMBER = {2},
     PAGES = {311--327},
}

@article {PolishchukModuliofCurves,
    AUTHOR = {Polishchuk, A.},
     TITLE = {Moduli of curves as moduli of $A_{\infty}$-structures},
   JOURNAL = {Duke Math. J.},
    VOLUME = {166},
      YEAR = {2017},
    NUMBER = {15},
     PAGES = {2871--2924},
}

@article {TabuadaDwykerKanModelStructure,
    AUTHOR = {Tabuada, G.},
     TITLE = {Une structure de cat\'{e}gorie de mod\`eles de {Q}uillen sur
              la cat\'{e}gorie des dg-cat\'{e}gories},
   JOURNAL = {C. R. Math. Acad. Sci. Paris},
  FJOURNAL = {Comptes Rendus Math\'{e}matique. Acad\'{e}mie des Sciences.
              Paris},
    VOLUME = {340},
      YEAR = {2005},
    NUMBER = {1},
     PAGES = {15--19},
      ISSN = {1631-073X,1778-3569},
   MRCLASS = {18E99},
  MRNUMBER = {2112034},
       DOI = {10.1016/j.crma.2004.11.007},
       URL = {https://doi.org/10.1016/j.crma.2004.11.007},
}

@incollection {KellerICMTalkDifferentialGradedCategories,
    AUTHOR = {Keller, B.},
     TITLE = {On differential graded categories},
 BOOKTITLE = {International {C}ongress of {M}athematicians. {V}ol. {II}},
     PAGES = {151--190},
 PUBLISHER = {Eur. Math. Soc., Z\"{u}rich},
      YEAR = {2006},
      ISBN = {978-3-03719-022-7},
   MRCLASS = {18E30 (14A22 16D90)},
  MRNUMBER = {2275593},
MRREVIEWER = {Volodymyr\ V.\ Lyubashenko},
}

@article {Faonte,
    AUTHOR = {Faonte, G.},
     TITLE = {{$A_\infty$}-functors and homotopy theory of
              dg-categories},
   JOURNAL = {J. Noncommut. Geom.},
  FJOURNAL = {Journal of Noncommutative Geometry},
    VOLUME = {11},
      YEAR = {2017},
    NUMBER = {3},
     PAGES = {957--1000},
      ISSN = {1661-6952,1661-6960},
   MRCLASS = {18G50 (18G55)},
  MRNUMBER = {3713010},
MRREVIEWER = {Estanislao\ Herscovich},
       DOI = {10.4171/JNCG/11-3-6},
       URL = {https://doi.org/10.4171/JNCG/11-3-6},
}

@article{BerglundStoll,
      title={Higher structures in rational homotopy theory}, 
      author={Berglund, A. and Stoll, R.},
      year={2023},
journal={arXiv:2310.11824 [math.AT]},
      eprint={2310.11824},
      archivePrefix={arXiv},
      primaryClass={math.AT}
}

@article {Saleh,
    AUTHOR = {Saleh, B.},
     TITLE = {Noncommutative formality implies commutative and {L}ie
              formality},
   JOURNAL = {Algebr. Geom. Topol.},
  FJOURNAL = {Algebraic \& Geometric Topology},
    VOLUME = {17},
      YEAR = {2017},
    NUMBER = {4},
     PAGES = {2523--2542},
      ISSN = {1472-2747,1472-2739},
   MRCLASS = {55P62},
  MRNUMBER = {3686405},
MRREVIEWER = {Steffen\ Sagave},
       DOI = {10.2140/agt.2017.17.2523},
       URL = {https://doi.org/10.2140/agt.2017.17.2523},
}

@article {Hermes,
    AUTHOR = {Hermes, S.},
     TITLE = {Minimal model of {G}inzburg algebras},
   JOURNAL = {J. Algebra},
  FJOURNAL = {Journal of Algebra},
    VOLUME = {459},
      YEAR = {2016},
     PAGES = {389--436},
      ISSN = {0021-8693,1090-266X},
   MRCLASS = {16G20 (16E35 16E45)},
  MRNUMBER = {3503979},
MRREVIEWER = {Jian\ Min\ Chen},
       DOI = {10.1016/j.jalgebra.2016.03.041},
       URL = {https://doi.org/10.1016/j.jalgebra.2016.03.041},
}

@article{KarabasLeeWrappedFukayaPlumbings,
      title={Wrapped {F}ukaya category of plumbings}, 
      author={Karabas, D. and Lee, S.},
      year={2024},
      eprint={2405.10783},
journal={arXiv:2405.10783 [math.SG]},
      archivePrefix={arXiv},
      primaryClass={math.SG}
}

@article{OpperGradedGentle,
      title={Autoequivalences of Fukaya categories of surfaces and graded gentle algebras}, 
      author={Opper, S.},
      year={2025},
      journal={arXiv:2510.11543} 
}

@article {ChoLee,
    AUTHOR = {Cho, C.-H.- and Lee, S.},
     TITLE = {Notes on {K}ontsevich-{S}oibelman's theorem about cyclic
              {$A_\infty$}-algebras},
   JOURNAL = {Int. Math. Res. Not. IMRN},
  FJOURNAL = {International Mathematics Research Notices. IMRN},
      YEAR = {2011},
    NUMBER = {14},
     PAGES = {3095--3140},
      ISSN = {1073-7928,1687-0247},
   MRCLASS = {53D37 (16E40 16E45 18G50)},
  MRNUMBER = {2817675},
MRREVIEWER = {Andrey\ Yu.\ Lazarev},
       DOI = {10.1093/imrn/rnq191},
       URL = {https://doi.org/10.1093/imrn/rnq191},
}
\bibliographystyle{alpha}

\end{document}